\documentclass[11pt]{preprint}
\usepackage[full]{textcomp}
\usepackage[osf]{newtxtext}
\usepackage{comment}

\usepackage{dsfont}
\usepackage{amssymb}
\usepackage{mathtools}
\usepackage{hyperref}
\usepackage{breakurl}
\usepackage{mhenvs}
\usepackage{mhequ}
\usepackage{mhsymb}
\usepackage{booktabs}
\usepackage{tikz}
\usepackage{circuitikz}
\usepackage{textgreek}
\usepackage{listings}
\usepackage{bm}

\usepackage{needspace}

\usetikzlibrary{arrows.meta, positioning, quotes}
\usepackage{mathrsfs}
\usepackage{longtable}
\usepackage{graphicx,caption}
\usepackage{subcaption}
\usetikzlibrary{plotmarks}
\usetikzlibrary{decorations.pathreplacing}
\usetikzlibrary{decorations.pathmorphing}
\usetikzlibrary{fit}

\usepackage{float}
\usepackage{microtype}
\usepackage{comment}
\usepackage{centernot}
\usepackage{enumitem}
\usepackage{bm}
\usepackage{stackrel}
\usepackage[utf8]{inputenc}

\usepackage{mhequ}
\usepackage{enumitem}
\usepackage[T1]{fontenc}
\usepackage{emptypage}
\usepackage{hyperref}
\usepackage{color}
\usepackage{mathrsfs}
\usepackage{tikz-cd}

\DeclareMathAlphabet{\mcb}{U}{BOONDOX-calo}{m}{n}
\SetMathAlphabet{\mcb}{bold}{U}{BOONDOX-calo}{b}{n}

\DeclareSymbolFont{timesoperators}{T1}{ptm}{m}{n}
\SetSymbolFont{timesoperators}{bold}{T1}{ptm}{b}{n}

\makeatletter
\newcommand*\bcdot{\mathpalette\bcdot@{.5}}
\newcommand*\bcdot@[2]{\mathbin{\vcenter{\hbox{\scalebox{#2}{$\m@th#1\bullet$}}}}}
\makeatother

\makeatletter
\renewcommand{\operator@font}{\mathgroup\symtimesoperators}
\makeatother

\colorlet{symbols}{blue!30!black!50}
\colorlet{testcolor}{green!60!black}
\definecolor{purple}{rgb}{0.55,0.05,0.8}

\usetikzlibrary{shapes.misc}
\usetikzlibrary{shapes.symbols}
\usetikzlibrary{shapes.geometric}
\usetikzlibrary{snakes}
\usetikzlibrary{decorations}
\usetikzlibrary{decorations.markings}

\usetikzlibrary{calc}
\usetikzlibrary{external}

\let\oldskull\skull
\def\skull{\mathord{\oldskull}}

\DeclareMathAlphabet{\mathbbm}{U}{bbm}{m}{n}

\DeclareFontFamily{U}{BOONDOX-calo}{\skewchar\font=45 }
\DeclareFontShape{U}{BOONDOX-calo}{m}{n}{
	<-> s*[1.05] BOONDOX-r-calo}{}
\DeclareFontShape{U}{BOONDOX-calo}{b}{n}{
	<-> s*[1.05] BOONDOX-b-calo}{}
\DeclareMathAlphabet{\mcb}{U}{BOONDOX-calo}{m}{n}
\SetMathAlphabet{\mcb}{bold}{U}{BOONDOX-calo}{b}{n}

\setlist{noitemsep,topsep=4pt}

\makeatletter 
\newcommand*{\bigcdot}{}
\DeclareRobustCommand*{\bigcdot}{%
	\mathbin{\mathpalette\bigcdot@{}}%
}
\newcommand*{\bigcdot@scalefactor}{.5}
\newcommand*{\bigcdot@widthfactor}{1.15}
\newcommand*{\bigcdot@}[2]{%
	\sbox0{$#1\vcenter{}$}
	\sbox2{$#1\cdot\m@th$}%
	\hbox to \bigcdot@widthfactor\wd2{%
		\hfil
		\mathbb{R}aise\ht0\hbox{%
			\scalebox{\bigcdot@scalefactor}{%
				\lower\ht0\hbox{$#1\bullet\m@th$}%
			}%
		}%
		\hfil
	}%
}
\makeatother

\def\symbol#1{\textcolor{symbols}{#1}}
\def\1{\mathbf{\symbol{1}}}

\def\bone{\mathbf{1}}

\def\fC{\mathfrak{C}}
\def\bfC{\pmb{\mathfrak{C}}}

\usepackage{mathalfa}
\DeclareMathAlphabet\mathzapf{T1}{pzc}{mb}{sc}

\def\zF{\mathfrak{F}}
\def\bzF{\pmb{\mathfrak{F}}}
\def\roots{R}

\def\eqdef{\stackrel{\text{\tiny def}}{=}}

\def\init{\mathcal{I}}

\usepackage{stmaryrd}
\def\fancynorm#1{{\talloblong #1 \talloblong}}
\SetSymbolFont{stmry}{bold}{U}{stmry}{m}{n} 

\newcommand{\noise}[1]{\llbracket #1 \rrbracket}

\newcommand{\mrd}{\mathrm{d}}
\newcommand{\mri}{\mathrm{i}}
\newcommand{\floor}[1]{\lfloor #1 \rfloor}
\newcommand{\roof}[1]{\lceil #1 \rceil}
\newcommand{\dd}{(2-d)/2}  

\def\tipar{\zeta}    

\colorlet{darkblue}{blue!90!black}
\colorlet{dgray}{green!20!darkgray}
\colorlet{darkgreen}{green!60!black}

\newcommand{\e}{\varepsilon}

\def\${|\!|\!|}

\def\E{\mathbf{E}}
\def\T{\mathbf{T}}

\newcommand{\mfe}{\mathfrak{e}}

\newcommand{\mfC}{\mathfrak{C}}

\newcommand{\mft}{\mathfrak{t}}

\newcommand{\mre}{\mathrm{e}}

\newcommand{\mcB}{\mathcal{B}}

\newcommand{\mcI}{\mathcal{I}}

\newcommand{\mcD}{\mathcal{D}}

\newcommand{\mcX}{\mathcal{X}}

\def\CP{\mathcal{P}}
\def\CC{\mathcal{C}}
\def\CD{\mathcal{D}}
\def\CI{\mathcal{I}}

\def\CM{\mathcal{M}}

\def\CT{\mathcal{T}}
\def\CS{\mathcal{S}}
\def\CB{\mathcal{B}}

\def\K{\mathscr{K}}
\def\I{\mathscr{I}}

\def\N{\mathbb{N}}

\def\R{\mathbb{R}}
\def\T{\mathbb{T}}
\def\Z{\mathbb{Z}}
\def\P{\mathbb{P}}
\def\E{\mathbb{E}}

\newcommand{\bbeta}{\bm{\beta}}
\newcommand{\bdelta}{\bm{\delta}}

\numberwithin{equation}{section}

\def\dash{\leavevmode\unskip\kern0.18em--\penalty\exhyphenpenalty\kern0.18em}
\def\slash{\leavevmode\unskip\kern0.15em/\penalty\exhyphenpenalty\kern0.15em}

\let\emph\textit 

\makeatother

\colorlet{greennode}{green!50!black}
\colorlet{rednode}{red!50!black}
\colorlet{lbluenode}{blue!25}
\colorlet{dbluenode}{blue}
\colorlet{orangenode}{orange}

\definecolor{connection}{rgb}{0.7,0.1,0.1}

\tikzset{
	dot/.style={circle,fill=black,inner sep=0pt, minimum size=1mm},
	dot0/.style={circle,fill=green!80,inner sep=1pt, minimum size=2mm},
	dot2/.style={circle,fill=black!55,inner sep=0pt, minimum size=2mm},
	root/.style={circle,fill=black!50,inner sep=0pt, minimum size=3mm},
	var/.style={circle,fill=black!10,draw=black,inner sep=0pt, minimum size=1.7mm},
	varR/.style={circle,fill=blue!10,draw=blue,inner sep=0pt, minimum size=1.7mm},
	var0/.style={circle,fill=gray!50,draw=black,inner sep=0pt, minimum size=1.7mm},
	var1/.style={rectangle,fill=black!10,draw=black,inner sep=0pt, minimum size=2mm},
	var2/.style={diamond,fill=black!10,draw=black,inner sep=0pt, minimum size=2.6mm},
	Var/.style={circle,fill=black!10,draw=black,inner sep=0pt, minimum size=2.2mm},
	kernel/.style={semithick,shorten >=2pt,shorten <=2pt},
	kernel1/.style={postaction={decorate,decoration={markings,mark=at position 0.45 with {\draw[-] (0,-0.08) -- (0,0.08);}}}},
	kernels/.style={snake=snake,segment amplitude=1pt,segment length=4pt},
	rho/.style={densely dashed,semithick,shorten >=2pt,shorten <=2pt},
	testfcn/.style={dotted,semithick,shorten >=2pt,shorten <=2pt},
	tau/.style={circle,inner sep=1pt,draw=black,fill=white,text=black,thin},
	renorm/.style={shape=circle,fill=white,inner sep=1pt},
	labl/.style={shape=rectangle,fill=white,inner sep=1pt},
	xi/.style={very thin,circle,fill=lbluenode,draw=symbols,inner sep=0pt,minimum size=1.2mm},
	xi1/.style={very thin,rectangle,fill=lbluenode,draw=symbols,inner sep=0pt,minimum size=1.2mm},
	xi2/.style={very thin,diamond,fill=lbluenode,draw=symbols,inner sep=0pt,minimum size=1.6mm},
	xigreen/.style={very thin,circle,fill=greennode,draw=symbols,inner sep=0pt,minimum size=1.2mm},
	xigreen1/.style={very thin,rectangle,fill=greennode,draw=symbols,inner sep=0pt,minimum size=1.2mm},
	xired/.style={very thin,circle,fill=rednode,draw=symbols,inner sep=0pt,minimum size=1.2mm},
	xilblue/.style={very thin,circle,fill=lbluenode,draw=symbols,inner sep=0pt,minimum size=1.2mm},
	xidblue/.style={very thin,circle,fill=dbluenode,draw=symbols,inner sep=0pt,minimum size=1.2mm},
	xiorange/.style={very thin,circle,fill=orangenode,draw=symbols,inner sep=0pt,minimum size=1.2mm},
	xix/.style={crosscircle,fill=lbluenode,draw=symbols,inner sep=0pt,minimum size=1.2mm},
	xix-green-red/.style={circle, fill=greennode!70!white,draw=rednode,inner sep=0pt,minimum size=1.6mm,append after command={node [inner sep=0pt,minimum size=0.8mm,thick, draw = rednode, cross out]{}}},
	xix-green-red1/.style={rectangle, fill=greennode!70!white,draw=rednode,inner sep=0pt,minimum size=1.5mm,append after command={node [inner sep=0pt,minimum size=1mm,thick, draw = rednode, cross out]{}}},
	xib/.style={very thin,circle,fill=lbluenode,draw=symbols,inner sep=0pt,minimum size=1.6mm},
	xib1/.style={very thin,rectangle,fill=lbluenode,draw=symbols,inner sep=0pt,minimum size=1.6mm},
	xie/.style={very thin,circle,fill=greennode,draw=symbols,inner sep=0pt,minimum size=1.6mm},
	xid/.style={very thin,circle,fill=lbluenode,draw=symbols,inner sep=0pt,minimum size=1.6mm},
	xibx/.style={crosscircle,fill=lbluenode,draw=symbols,inner sep=0pt,minimum size=1.6mm},
	kernels2/.style={ultra thick,draw=symbols,segment length=12pt},
	not/.style={thin,regular polygon, regular polygon sides=3,draw=connection,fill=connection,inner sep=0pt,minimum size=1.2mm},
	notlblue/.style={thin,regular polygon, regular polygon sides=3,draw=lbluenode,fill=lbluenode,inner sep=0pt,minimum size=1.2mm},
	notorange/.style={thin,regular polygon, regular polygon sides=3,draw=orangenode,fill=orangenode,inner sep=0pt,minimum size=1.2mm},
	notgreen/.style={thin,regular polygon, regular polygon sides=3,draw=greennode,fill=greennode,inner sep=0pt,minimum size=1.2mm},
	>=stealth,
}

\makeatletter
\def\DeclareSymbol#1#2#3{%
	\expandafter\gdef\csname MH@symb@#1\endcsname{\tikzsetnextfilename{symbol#1}%
		\tikz[baseline=#2,scale=0.15,draw=symbols,line join=round]{#3}}%
	\expandafter\gdef\csname MH@symb@#1s\endcsname{\scalebox{0.75}{\tikzsetnextfilename{symbol#1}%
			\tikz[baseline=#2,scale=0.15,draw=symbols,line join=round]{#3}}}%
	\expandafter\gdef\csname MH@symb@#1ss\endcsname{\scalebox{0.65}{\tikzsetnextfilename{symbol#1}%
			\tikz[baseline=#2,scale=0.15,draw=symbols,line join=round]{#3}}}%
}
\def\<#1>{\ifmmode\mathchoice{\csname MH@symb@#1\endcsname}{\csname MH@symb@#1\endcsname}{\csname MH@symb@#1s\endcsname}{\csname MH@symb@#1ss\endcsname}\else\csname MH@symb@#1\endcsname\fi}
\makeatother

\DeclareSymbol{IXiI'Xi}{-1}{\draw (-1,1) node[xi] {} -- (0,0);
	\draw (1,1)[kernels2] node[xi] {} -- (0,0) {};}

\DeclareSymbol{IXi^3}{0}{\draw (-1.3,1) node[xi] {} -- (0,0);
	\draw (1.3,1) node[xi] {} -- (0,0);
	\draw (0,1.5) node[xi] {} -- (0,0) {};}

\DeclareSymbol{I[IXiI'Xi]I'[IXiI'Xi]}{1}{\draw[kernels2] (2,2) node[xi] {} -- (1,.9);
	\draw (0.5,2) node[xi] {} -- (1,.9) {};
	\draw[kernels2] (1,.9) -- (0,0) {};
	\draw (-2,2) node[xi] {} --  (-1,.9){};
	\draw (-.5,2)[kernels2]node[xi] {}  -- (-1,.9) {};
	\draw (-1,.9)  {} -- (0,0);}

\DeclareSymbol{PsiXi1}{0}{\draw (0,0) node[xi1] {} -- (1.2,1.2) node[xi] {};}

\DeclareSymbol{I(PsiXi1)Xi1}{-2}{\draw  (1.2,1.2) node[xi] {} -- (0,0) node[xi1] {} -- (1.2,-1.2) node[xi1] {}; }

\DeclareSymbol{I(PsiXi1)Xi1less}{-5}{\draw  (0,0) node[xi1] {} -- (-1.2,-1.2) node[xi1] {}; }

\DeclareSymbol{supercherry1}{0}{\draw (-1,1) node[xi] {} -- (0,0);
	\draw  (0,2) node[xi] {}-- (1,1);
	\draw  (1,1)[kernels2] node[xi1] {} -- (0,0) {};}

\DeclareSymbol{supercherry1less}{0}{
	\draw  (0,2) node[xi] {}-- (1,1);
	\draw  (1,1)[kernels2] node[xi1] {} -- (0,0) {};}

\DeclareSymbol{supercherry2}{0}{
	\draw  (0,2) node[xi] {}-- (-1,1);
	\draw (-1,1) node[xi1] {} -- (0,0);
	\draw  (1,1)[kernels2] node[xi] {} -- (0,0) {};}

\DeclareSymbol{cherry231}{2}{		
	\draw  (0,2) node[xi] {}-- (1,1);
	\draw[kernels2] (1,1) node[xi1] {} -- (0,0);
	\draw[kernels2]  (-1,1) node[xi1] {} -- (0,0) {};}

\DeclareSymbol{cherry232}{-1}{		
	\draw[kernels2] (1,1) node[xi1] {} -- (0,0);
	\draw[kernels2]  (-1,1) node[xi1] {} -- (0,0) {};}

\DeclareSymbol{I'Xibar}{-1}{
	\draw (0,1)[kernels2] node[xi1] {} -- (0,-.4);}

\DeclareSymbol{I'Xitilde}{-1}{
	\draw (0,1)[kernels2] node[xi2] {} -- (0,-.4);}

\DeclareSymbol{IXiI'Xibar}{-1}{\draw (-1,1) node[xi] {} -- (0,0);
	\draw (1,1)[kernels2] node[xi1] {} -- (0,0);}

\DeclareSymbol{I'XiIXibar}{-1}{\draw (-1,1) node[xi1] {} -- (0,0);
	\draw (1,1)[kernels2] node[xi] {} -- (0,0);}

\DeclareSymbol{cherryY}{-5}{\draw (-1,1) node[xi] {} -- (0,0);
	\draw (1,1)[kernels2] node[xi] {} -- (0,0);
	\draw (0,-1.5) node[xi1] {} -- (0,0);
}

\DeclareSymbol{cherryYless}{-5}{
	\draw (1,1)[kernels2] node[xi] {} -- (0,0);
	\draw (0,-1.5) node[xi1] {} -- (0,0);
}

\DeclareSymbol{R2-1}{-1}{\draw (-1,1)[kernels2] node[xi] {} -- (0,0);
	\draw (1,1)[kernels2] node[xi1] {} -- (0,0);}

\DeclareSymbol{R2-1new}{-1}{\draw (-1,1)[kernels2] node[xi2] {} -- (0,0);
	\draw (1,1)[kernels2] node[xi2] {} -- (0,0);}

\DeclareSymbol{I(R2-1new)}{-1}{\draw (-1,2)[kernels2] node[xi2] {} -- (0,1);
	\draw (1,2)[kernels2] node[xi2] {} -- (0,1);
	\draw (0,1) {} -- (0,-0.4);}

\DeclareSymbol{I'XiI'[IXiI'Xi]}{1}{\draw[kernels2] (2,2) node[xi] {} -- (1,1);
	\draw (0,2) node[xi] {} -- (1,1) node {};
	\draw[kernels2] (1,1) -- (0,0) node {};
	\draw[kernels2] (-1,1) node[xi1] {} -- (0,0);}

\DeclareSymbol{I'XiI'[IXiI'Xi]less}{1}{\draw[kernels2] (0,2) node[xi] {} -- (1,1);
	\draw[kernels2] (1,1) -- (0,0) node {};
	\draw[kernels2] (-1,1) node[xi1] {} -- (0,0);}

\DeclareSymbol{I'[IXiI'Xi]double}{-1}{\draw[kernels2] (2,2) node[xi] {} -- (1,1);
	\draw (.6,2) node[xi] {} -- (1,1) node {};
	\draw[kernels2] (1,1) -- (0,0) node {};
	\draw[kernels2] (-2,2) node[xi] {} -- (-1,1);
	\draw (-.6,2) node[xi] {} -- (-1,1) node {};
	\draw[kernels2] (-1,1) -- (0,0) node {};
}

\DeclareSymbol{I'[IXiI'Xi]double-less1}{-1}{\draw[kernels2] (0.6,2) node[xi] {} -- (1,1);
	\draw[kernels2] (1,1) -- (0,0) node {};
	\draw[kernels2] (-2,2) node[xi] {} -- (-1,1);
	\draw (-.6,2) node[xi] {} -- (-1,1) node {};
	\draw[kernels2] (-1,1) -- (0,0) node {};
}

\DeclareSymbol{I'[IXiI'Xi]double-less2}{-1}{\draw[kernels2] (.6,2) node[xi] {} -- (1,1);
	\draw[kernels2] (1,1) -- (0,0) node {};
	\draw[kernels2] (-.6,2) node[xi] {} -- (-1,1);
	\draw[kernels2] (-1,1) -- (0,0) node {};
}

\DeclareSymbol{I'[IXiI'Xi]single-less}{-1}{\draw[kernels2] (0,2) node[xi] {} -- (1,1);
	\draw[kernels2] (1,1) -- (0,0) node {};
}

\DeclareSymbol{I'[IXiI'Xi]single}{-5}{\draw (-1,1) node[xi] {} -- (0,0);
	\draw (1,1)[kernels2] node[xi] {} -- (0,0);
	\draw[kernels2] (0,-1.5) {} -- (0,0);
}

\DeclareSymbol{I'(R2-1)}{-1}{
	\draw[kernels2] (0,0) -- (0,-1);
	\draw (-1,1)[kernels2] node[xi] {} -- (0,0);
	\draw (1,1)[kernels2] node[xi1] {} -- (0,0);
}

\DeclareSymbol{I'(R2-1)new}{-1}{
	\draw[kernels2] (0,0) -- (0,-1);
	\draw (-1,1)[kernels2] node[xi2] {} -- (0,0);
	\draw (1,1)[kernels2] node[xi2] {} -- (0,0);
}

\DeclareSymbol{I(R2-1)Xi1new}{-1}{
	\draw (0,0) -- (0,-1.5) node[xi1] {};
	\draw (-1,1)[kernels2] node[xi2] {} -- (0,0);
	\draw (1,1)[kernels2] node[xi2] {} -- (0,0);
}

\DeclareSymbol{IXiI'(R2-1)}{-1}{
	\draw (-2,0) node[xi] {} -- (-1,-1);
	\draw[kernels2] (0,0) -- (-1,-1);
	\draw (-1,1)[kernels2] node[xi] {} -- (0,0);
	\draw (1,1)[kernels2] node[xi1] {} -- (0,0);
}

\DeclareSymbol{IXiI'(R2-1)new}{-1}{
	\draw (-2,0) node[xi] {} -- (-1,-1);
	\draw[kernels2] (0,0) -- (-1,-1);
	\draw (-1,1)[kernels2] node[xi2] {} -- (0,0);
	\draw (1,1)[kernels2] node[xi2] {} -- (0,0);
}

\DeclareSymbol{IXiI(R2-1)new}{-1}{
	\draw (-2,0) node[xi] {} -- (-1,-1);
	\draw (0,0) -- (-1,-1);
	\draw (-1,1)[kernels2] node[xi2] {} -- (0,0);
	\draw (1,1)[kernels2] node[xi2] {} -- (0,0);
}

\DeclareSymbol{I'XiI(R2-1)}{-1}{
	\draw (0,0)  -- (1,-1);
	\draw (2,0)[kernels2] node[xi] {} -- (1,-1);
	\draw (-1,1)[kernels2] node[xi1] {} -- (0,0);
	\draw (1,1)[kernels2] node[xi] {} -- (0,0);
}

\DeclareSymbol{I'XiI(R2-1)new}{-1}{
	\draw (0,0)  -- (1,-1);
	\draw (2,0)[kernels2] node[xi] {} -- (1,-1);
	\draw (-1,1)[kernels2] node[xi2] {} -- (0,0);
	\draw (1,1)[kernels2] node[xi2] {} -- (0,0);
}

\DeclareSymbol{I'Xi1I'(R2-1)new}{-1}{
	\draw[kernels2] (0,0)  -- (1,-1);
	\draw (2,0)[kernels2] node[xi1] {} -- (1,-1);
	\draw (-1,1)[kernels2] node[xi2] {} -- (0,0);
	\draw (1,1)[kernels2] node[xi2] {} -- (0,0);
}

\DeclareSymbol{IXibar}{-1}{
	\draw (0,1) node[xi1] {} -- (0,-.4);}

\DeclareSymbol{PsiPsibarXibar}{-1}{\draw (-1,1.3) node[xi] {} -- (0,0);
	\draw (1,1.3) node[xi1] {} -- (0,0) node[xi1] {};}

\DeclareSymbol{PsiPsibar}{-1}{\draw (-1,1.3) node[xi] {} -- (0,0);
	\draw (1,1.3) node[xi1] {} -- (0,0) {};}

\DeclareSymbol{IPsiXi1}{0}{\draw (0,-0.4) {} -- (0,1) node[xi1] {} -- (1.2,2.2) node[xi] {};}

\DeclareSymbol{r_z_large}{2}{
	\draw (-1,1.3) node[xi] {} -- (0,0) {};
	\draw (0,2.6) node[xi] {} -- (1,1.3) node[xi1] {} -- (0,0) node[xi1] {};}

\DeclareSymbol{r_z_large_pos}{2}{
	\draw (-1,1.3) node[xi] {} -- (0,0) {};
	\draw (0,2.6) node[xi] {} -- (1,1.3) node[xi1] {} -- (0,0) {};}

\DeclareSymbol{r_z_avoid1}{2}{
	\draw (-1,1.3) node[xi] {} -- (0,0) {};
	\draw (0,2.6) node[xi] {} -- (1,1.3) {} -- (0,0) node[xi1] {};
	\draw (2,2.6)[kernels2] node[xi] {} -- (1,1.3) {};}

\DeclareSymbol{r_z_avoid2}{2}{
	\draw (-1,1.3) node[xi] {} -- (0,0) {};
	\draw (0,2.6)[kernels2] node[xi] {} -- (1,1.3) {};
	\draw (1,1.3) {} -- (0,0) node[xi1] {};}

\DeclareSymbol{r_z_avoid3}{2}{
	\draw (-1,1.3) node[xi] {} -- (0,0) {};
	\draw (0,2.6)[kernels2] node[xi2] {} -- (1,1.3);
	\draw (2,2.6)[kernels2] node[xi2] {} -- (1,1.3);
	\draw (1,1.3) {} -- (0,0) node[xi1] {};}

\DeclareSymbol{1}{0}{\draw[white] (-.5,0) -- (.5,0); \draw (0,0)  -- (0,1.5) node[sdot] {};}
\DeclareSymbol{11}{0}{\draw (-0.5,1.2) node[sdot] {} -- (0,0) -- (0.5,1.2) node[sdot] {};}
\DeclareSymbol{111}{0}{\draw (0,0) -- (0,1.2) node[sdot] {}; \draw (-.7,1) node[sdot] {} -- (0,0) -- (.7,1) node[sdot] {};}
\DeclareSymbol{131}{-3}{\draw (0,0) -- (0,-1) -- (1,0) node[sdot] {}; \draw (0,0) -- (0,-1) -- (-1,0) node[sdot] {}; \draw (0,0) -- (0,1.2) node[sdot] {}; \draw (-.7,1) node[sdot] {} -- (0,0) -- (.7,1) node[sdot] {};}
\DeclareSymbol{11}{0}{\draw (-0.5,1.2) node[sdot] {} -- (0,0) -- (0.5,1.2) node[sdot] {};}
\DeclareSymbol{12}{-3}{\draw (-.8,1) node[sdot] {} -- (0,0) -- (0,-1); \draw (1,0) node[sdot] {} -- (0,-1) -- (-1,0) node[sdot] {};}
\DeclareSymbol{30}{-3}{\draw (0,0) -- (0,-1); \draw (0,0) -- (0,1.2) node[sdot] {}; \draw (-.7,1) node[sdot] {} -- (0,0) -- (.7,1) node[sdot] {};}
\DeclareSymbol{10}{-3}{\draw (-.8,1) node[sdot] {} -- (0,0) -- (0,-1);}
\DeclareSymbol{22}{-3}{\draw (0,0.3) -- (0,-1) -- (1,0) node[sdot] {}; \draw (0,0.3) -- (0,-1) -- (-1,0) node[sdot] {};\draw (-.7,1) node[sdot] {} -- (0,0.3) -- (.7,1) node[sdot] {};}
\DeclareSymbol{211}{0}{\draw (-0.5,2.4) node[sdot] {} -- (-1,1.6);\draw (-1.5,2.4) node[sdot] {} -- (-0.5,0.8); \draw (0,1.6) node[sdot] {} -- (-0.5,0.8) -- (0,0) -- (0.5,0.8) node[sdot] {};}

\DeclareSymbol{Xi22}{0.5}{\draw (0,0) node[xi] {} -- (-1,1) node[not] {} -- (0,2) node[xi] {};}

\DeclareSymbol{Xi2}{-2}{\draw (0,-0.25) node[xi] {} -- (-1,1) node[xi] {};}
\DeclareSymbol{Xi3}{0}{\draw (0,0) node[xi] {} -- (-1,1) node[xi] {} -- (0,2) node[xi] {};}
\DeclareSymbol{Xi4}{2}{\draw (0,0) node[not] {} -- (-1,1) node[xi] {} -- (0,2) node[xi] {} -- (-1,3) node[xi] {};}
\DeclareSymbol{Xi2X}{-2}{\draw (0,-0.25) node[xi] {} -- (-1,1) node[xix] {};}
\DeclareSymbol{XXi2}{-2}{\draw (0,-0.25) node[xix] {} -- (-1,1) node[xi] {};}

\DeclareSymbol{IXi^2asym}{-1}{\draw (-1,1) node[xired] {} -- (0,-.4);
	\draw (1,1) node[xigreen] {} -- (0,-.4);}

\DeclareSymbol{IXi^2asym2}{-1}{\draw (-1,1) node[xigreen] {} -- (0,-.4);
	\draw (1,1) node[xired] {} -- (0,-.4);}

\DeclareSymbol{IXi^2asym3}{-1}{\draw (-1,1) node[xired] {} -- (0,-.4);
	\draw (1,1) node[xigreen] {} -- (0,-.4);}

\DeclareSymbol{IXi^2}{-1}{\draw (-1,1) node[xi] {} -- (0,-.4);
	\draw (1,1) node[xi] {} -- (0,-.4);}

\DeclareSymbol{rIXi^2}{-1}{\draw (-1,1) node[xired] {} -- (0,-.4);
	\draw (1,1) node[xired] {} -- (0,-.4);}

\DeclareSymbol{IXi^2green}{-1}{\draw (-1,1) node[xigreen] {} -- (0,-.4);
	\draw (1,1) node[xigreen] {} -- (0,-.4);}

\DeclareSymbol{IXi^2_typed}{-1}{\draw (-1,1) node[xiorange] {} -- (0,-.4);
	\draw (1,1) node[xigreen] {} -- (0,-.4);}

\DeclareSymbol{IXi^2green*}{-1}{\draw (-1,1) node[xigreen1] {} -- (0,-.4);
	\draw (1,1) node[xigreen1] {} -- (0,-.4);}

\DeclareSymbol{I(cherry)}{-1}{\draw (-1,2) node[xi] {} -- (0,1);
	\draw (1,2)[kernels2] node[xi] {} -- (0,1) {};
	\draw (0,1) {} -- (0,-0.4) {};}

\DeclareSymbol{IDPsi}{2}{\draw (0,2)[kernels2] node[xi] {} -- (1,1);
	\draw (1,1) {} -- (0,0) {};}

\DeclareSymbol{I'XiIXi_notriangle}{-1}{\draw (-1,1)[kernels2] node[xi] {} -- (0,0);
	\draw (1,1) node[xi] {} -- (0,0) {};}

\DeclareSymbol{IXiI'Xi_typed1}{-1}{\draw (-1,1) node[xigreen] {} -- (0,0);
	\draw (1,1)[kernels2,greennode] node[xired] {} -- (0,0) node[not] {};}

\DeclareSymbol{IXiI'Xi_typed2}{-1}{\draw (-1,1) node[xigreen] {} -- (0,0);
	\draw (1,1)[kernels2,rednode] node[xigreen] {} -- (0,0) node[not] {};}

\DeclareSymbol{IXiI'Xi_typed3}{-1}{\draw (-1,1) node[xiorange] {} -- (0,0);
	\draw (1,1)[kernels2,orangenode] node[xigreen] {} -- (0,0) node[notgreen] {};}

\DeclareSymbol{IXiI'Xi_typed4}{-1}{\draw (-1,1) node[xiorange] {} -- (0,0);
	\draw (1,1)[kernels2,greennode] node[xiorange] {} -- (0,0) node[notgreen] {};}

\DeclareSymbol{IXiI'Xi_typed2_notriangle}{-1}{\draw (-1,1) node[xigreen] {} -- (0,0);
	\draw (1,1)[kernels2,rednode] node[xigreen] {} -- (0,0) {};}

\DeclareSymbol{IXi^3_typed}{-1}{\draw (-1.3,1) node[xigreen] {} -- (0,0);
	\draw (1.3,1) node[xired] {} -- (0,0);
	\draw (0,1) node[xigreen] {} -- (0,0) node[not] {};}

\DeclareSymbol{IXi^3-typed112}{-1}{\draw (-1.3,1) node[xi] {} -- (0,0);
	\draw (1.3,1) node[xigreen] {} -- (0,0);
	\draw (0,1) node[xi] {} -- (0,0) node[not] {};}

\DeclareSymbol{I'Xi}{-1}{
	\draw[kernels2] (0,1) node[xi] {} -- (0,-0.4) {};}

\DeclareSymbol{I'Xigreen-red}{-1}{
	\draw[kernels2,rednode] (0,1) node[xi,fill=greennode] {} -- (0,-.4);}

\DeclareSymbol{IXi}{-1}{
	\draw (0,1) node[xi] {} -- (0,-.4);}

\DeclareSymbol{IXigreen}{-1}{
	\draw (0,1) node[xigreen] {} -- (0,-.4);}

\DeclareSymbol{green}{-1}{
	\draw[greennode,very thick,line cap=round,scale=1.3] (0.05,0.4) -- (0.2,0.55) -- (0,0) -- (0.6,0.6) -- (0.4,0.05) -- (0.55,0.2);}
\DeclareSymbol{red}{-1}{
	\draw[rednode,very thick,line cap=round,scale=1.3] (0.05,0.4) -- (0.2,0.55) -- (0,0) -- (0.6,0.6) -- (0.4,0.05) -- (0.55,0.2);}
\DeclareSymbol{lblue}{-1}{
	\draw[lbluenode,very thick,line cap=round,scale=1.3] (0.05,0.4) -- (0.2,0.55) -- (0,0) -- (0.6,0.6) -- (0.4,0.05) -- (0.55,0.2);}
\DeclareSymbol{dblue}{-1}{
	\draw[dbluenode,very thick,line cap=round,scale=1.3] (0.05,0.4) -- (0.2,0.55) -- (0,0) -- (0.6,0.6) -- (0.4,0.05) -- (0.55,0.2);}
\DeclareSymbol{orange}{-1}{
	\draw[orangenode,very thick,line cap=round,scale=1.3] (0.05,0.4) -- (0.2,0.55) -- (0,0) -- (0.6,0.6) -- (0.4,0.05) -- (0.55,0.2);}

\DeclareSymbol{I'XXi}{-1}{
	\draw[kernels2] (1,1) node[xi] {} -- (0,0) node[not] {};
	\draw (1,1) node[xix] {};}

\DeclareSymbol{I'XXi_notriangle}{-1}{
	\draw[kernels2] (1,1) node[xi] {} -- (0,0) {};
	\draw (1,1) node[xix] {};}

\DeclareSymbol{IXiI'[IXiI'Xi]}{-1}{\draw[kernels2] (2,2) node[xi] {} -- (1,1);
	\draw (0,2) node[xi] {} -- (1,1) node[not] {};
	\draw[kernels2] (1,1) -- (0,0) node[not] {};
	\draw (-1,1) node[xi] {} -- (0,0);}

\DeclareSymbol{IXiI'[IXiI'Xi]_notriangle}{1}{\draw[kernels2] (2,2) node[xi] {} -- (1,1);
	\draw (0,2) node[xi] {} -- (1,1) {};
	\draw[kernels2] (1,1) -- (0,0) {};
	\draw (-1,1) node[xi] {} -- (0,0);}

\DeclareSymbol{I'[IXiI'Xi]}{-1}{\draw[kernels2] (2,2) node[xi] {} -- (1,1);
	\draw (0,2) node[xi] {} -- (1,1) node[not] {};
	\draw[kernels2] (1,1) -- (0,0) node[not] {};}

\DeclareSymbol{I[IXiI'Xi]}{-1}{\draw[kernels2] (2,2) node[xi] {} -- (1,1);
	\draw (0,2) node[xi] {} -- (1,1);
	\draw (1,1) -- (0,0);}

\DeclareSymbol{I'[IXiI'Xi]_notriangle}{1}{\draw[kernels2] (2,2) node[xi] {} -- (1,1);
	\draw (0,2) node[xi] {} -- (1,1) {};
	\draw[kernels2] (1,1) -- (0,0) {};}

\DeclareSymbol{I[IXi^2]IXi-typed112}{-1}{\draw (-2,2) node[xi] {} -- (-1,1);
	\draw (0,2) node[xi] {} -- (-1,1) node[not] {};
	\draw (-1,1) -- (0,0) node[not] {};
	\draw (1,1) node[xigreen] {} -- (0,0);}

\DeclareSymbol{I[IXi^2]IXi-typed211}{-1}{\draw (-2,2) node[xigreen] {} -- (-1,1);
	\draw (0,2) node[xi] {} -- (-1,1) node[not] {};
	\draw (-1,1) -- (0,0) node[not] {};
	\draw (1,1) node[xi] {} -- (0,0);}

\DeclareSymbol{I[IXiI'Xi]I'Xi}{1}{\draw (-2,2) node[xi] {} -- (-1,1);
	\draw[kernels2] (0,2) node[xi] {} -- (-1,1) {};
	\draw (-1,1) -- (0,0) {};
	\draw[kernels2] (1,1) node[xi] {} -- (0,0);}

\DeclareSymbol{I[I[Xi]]Xi_notriangle}{1}{
	\draw (0,2) node[xi] {} -- (-1,1) {};
	\draw[snake=snake , segment length=1pt, segment amplitude=1pt] (-1,1) -- (0,0) node[xi] {};}

\DeclareSymbol{I[I[Xi]]Xi_typed}{1}{
	\draw (0,2) node[xigreen] {} -- (-1,1) {};
	\draw[snake=snake , segment length=1pt, segment amplitude=1pt] (-1,1) -- (0,0) node[xigreen] {};}

\DeclareSymbol{I[I[Xi]]Xi_typed*}{1}{
	\draw (0,2) node[xigreen1] {} -- (-1,1) {};
	\draw[snake=snake , segment length=1pt, segment amplitude=1pt] (-1,1) -- (0,0) node[xigreen1] {};}

\DeclareSymbol{I[I'Xi]I'Xi}{1}{\draw[kernels2] (0,2) node[xi] {} -- (-1,1);
	\draw (-1,1) node[fill=rednode,not] {} -- (0,0);
	\draw[kernels2] (1,1) node[xi] {} -- (0,0);}

\DeclareSymbol{I[I'Xi]I'Xi_notriangle}{1}{\draw[kernels2] (0,2) node[xi] {} -- (-1,1);
	\draw (-1,1) {} -- (0,0);
	\draw[kernels2] (1,1) node[xi] {} -- (0,0);}

\DeclareSymbol{I[I'Xi]I'Xi_typed}{1}{\draw[kernels2,rednode] (0,2) node[xigreen] {} -- (-1,1);
	\draw (-1,1) node[notorange] {} -- (0,0);
	\draw[kernels2,rednode] (1,1) node[xigreen] {} -- (0,0);}

\DeclareSymbol{I[I'Xi]I'Xi_typed*}{1}{\draw[kernels2,rednode] (0,2) node[xigreen1] {} -- (-1,1);
	\draw (-1,1) node[notorange] {} -- (0,0);
	\draw[kernels2,rednode] (1,1) node[xigreen1] {} -- (0,0);}

\DeclareSymbol{I[I'Xi]I'Xi_typed-h}{1}{\draw[kernels2,greennode] (0,2) node[xigreen] {} -- (-1,1);
	\draw[snake=snake , segment length=3pt, segment amplitude=1pt] (-1,1) node[notgreen] {} -- (0,0);
	\draw[kernels2,greennode] (1,1) node[xigreen] {} -- (0,0);}

\DeclareSymbol{I[I'Xi]I'Xi_typed-h*}{1}{\draw[kernels2,greennode] (0,2) node[xigreen1] {} -- (-1,1);
	\draw[snake=snake , segment length=3pt, segment amplitude=1pt] (-1,1) node[notgreen] {} -- (0,0);
	\draw[kernels2,greennode] (1,1) node[xigreen1] {} -- (0,0);}

\DeclareSymbol{I[I'Xi]I'Xib}{1}{\draw[kernels2] (0,2) node[xi] {} -- (1,1);
	\draw (1,1) node[not] {} -- (0,0);
	\draw[kernels2] (-1,1) node[xi] {} -- (0,0);}

\DeclareSymbol{I[I'Xi]I'Xib_notriangle}{1}{\draw[kernels2] (0,2) node[xi] {} -- (1,1);
	\draw (1,1) {} -- (0,0);
	\draw[kernels2] (-1,1) node[xi] {} -- (0,0);}

\DeclareSymbol{IXiI'[I'Xi]}{1}{\draw[kernels2] (0,2) node[xi] {} -- (1,1);
	\draw[kernels2] (1,1) node[not] {} -- (0,0);
	\draw (-1,1) node[xi] {} -- (0,0);}

\DeclareSymbol{IXiI'[I'Xi]_notriangle}{1}{\draw[kernels2] (0,2) node[xi] {} -- (1,1);
	\draw[kernels2] (1,1)  {} -- (0,0);
	\draw (-1,1) node[xi] {} -- (0,0);}

\DeclareSymbol{IXiI'[I'Xi]_typed}{1}{\draw[kernels2,rednode] (0,2) node[xigreen] {} -- (1,1);
	\draw[kernels2,rednode] (1,1) node[notorange] {} -- (0,0);
	\draw (-1,1) node[xigreen] {} -- (0,0);}

\DeclareSymbol{IXiI'[I'Xi]_typed-h}{1}{\draw[kernels2,greennode] (0,2) node[xigreen] {} -- (1,1);
	\draw[kernels2,greennode,snake=snake , segment length=3pt, segment amplitude=1pt] (1,1) node[notgreen] {} -- (0,0);
	\draw (-1,1) node[xigreen] {} -- (0,0);}

\DeclareSymbol{IXiI'[I'Xi]_typed*}{1}{\draw[kernels2,rednode] (0,2) node[xigreen1] {} -- (1,1);
	\draw[kernels2,rednode] (1,1) node[notorange] {} -- (0,0);
	\draw (-1,1) node[xigreen1] {} -- (0,0);}

\DeclareSymbol{IXiI'[I'Xi]_typed-h*}{1}{\draw[kernels2,greennode] (0,2) node[xigreen1] {} -- (1,1);
	\draw[kernels2,greennode,snake=snake , segment length=3pt, segment amplitude=1pt] (1,1) node[notgreen] {} -- (0,0);
	\draw (-1,1) node[xigreen1] {} -- (0,0);}

\DeclareSymbol{I[I'Xi]}{-1}{\draw[kernels2] (0,2) node[xi] {} -- (1,1);
	\draw (1,1) node[not] {} -- (0,0) node[not] {};}

\DeclareSymbol{I[I'Xi]_typed}{1}{\draw[kernels2,rednode] (0,2) node[xigreen] {} -- (-1,1);
	\draw (-1,1) node[notorange] {} -- (0,0);}

\DeclareSymbol{I'[I'Xi]}{-1}{\draw[kernels2] (0,2) node[xi] {} -- (1,1);
	\draw[kernels2] (1,1) node[not] {} -- (0,0) node[not] {};}

\DeclareSymbol{I'[I'Xi]_notriangle}{1}{\draw[kernels2] (0,2) node[xi] {} -- (1,1);
	\draw[kernels2] (1,1) {} -- (0,0) {};}

\DeclareSymbol{I'[I'Xi]_typed}{1}{\draw[kernels2,rednode] (0,2) node[xigreen] {} -- (1,1);
	\draw[kernels2,rednode] (1,1) node[notorange] {} -- (0,0)  {};}

\DeclareSymbol{IXiI'XXi}{-1}{\draw (-1,1) node[xi] {} -- (0,0);
	\draw[kernels2] (1,1) node[xi] {} -- (0,0) node[not] {};
	\draw (1,1) node[xix] {};}

\DeclareSymbol{IXiI'XXi_notriangle}{-1}{\draw (-1,1) node[xi] {} -- (0,0);
	\draw[kernels2] (1,1) node[xi] {} -- (0,0) {};
	\draw (1,1) node[xix] {};}

\DeclareSymbol{IXiI'XXi_typed}{-1}{\draw (-1,1) node[xigreen] {} -- (0,-0.3);
	\draw[kernels2,rednode] (1,1)  -- (0,-0.3) {};
	\draw (1,1) node[xix-green-red] {};
}

\DeclareSymbol{IXiI'XXi_typed*}{-1}{\draw (-1,1) node[xigreen1] {} -- (0,-0.3);
	\draw[kernels2,rednode] (1,1)  -- (0,-0.3) {};
	\draw (1,1) node[xix-green-red1] {};
}

\DeclareSymbol{IXXiI'Xi}{-1}{\draw (-1,1) node[xix] {} -- (0,0);
	\draw[kernels2] (1,1) node[xi] {} -- (0,0) node[not] {};}

\DeclareSymbol{IXXiI'Xi_notriangle}{-1}{\draw (-1,1) node[xix] {} -- (0,0);
	\draw[kernels2] (1,1) node[xi] {} -- (0,0) {};}

\DeclareSymbol{IXXiI'Xi_typed}{-1}{\draw (-1,1) node[xix-green-red] {} -- (0,-0.3);
	\draw[kernels2,rednode] (1,1) node[xigreen] {} -- (0,-0.3) {};}

\DeclareSymbol{IXXiI'Xi_typed*}{-1}{\draw (-1,1) node[xix-green-red1] {} -- (0,-0.3);
	\draw[kernels2,rednode] (1,1) node[xigreen1] {} -- (0,-0.3) {};}

\DeclareSymbol{XiX}{-2.8}{\node[xibx] {};}
\DeclareSymbol{XXi}{-2.8}{\node[xibx] {};}
\DeclareSymbol{tauX}{-2.8}{ \draw[kernels2] (0,0) node[xibx] {};}
\DeclareSymbol{Xi}{-2.8}{\node[xib] {};}
\DeclareSymbol{I}{0}{\draw (0,2) -- (0,0);}
\DeclareSymbol{I'}{0}{\draw[kernels2] (0,2) -- (0,0);}
\DeclareSymbol{Xi1}{-2.8}{\node[xib1] {};}
\DeclareSymbol{Xigreen}{-2.8}{\node[xigreen] {};}
\DeclareSymbol{Xired}{-2.8}{\node[xired] {};}
\DeclareSymbol{not}{-2.8}{\node[not,minimum size=1.6mm] {};}
\DeclareSymbol{XiY}{-2.8}{\node[xie] {};}
\DeclareSymbol{XiZ}{-2.8}{\node[xid] {};}
\DeclareSymbol{IXiX}{0}{\draw (0,-0.25) node[not] {} -- (0,1.5) node[xix] {};}

\allowdisplaybreaks[0]

\title{Local well-posedness of subcritical non-linear heat equations with Gaussian initial data}
\author{Ilya~Chevyrev$^1$ and Hora~Mirsajjadi$^{2}$}
\institute{SISSA, Trieste, Italy, \email{ichevyrev@gmail.com} 
\and School of Mathematics, The University of Edinburgh,
United Kingdom, \email{h.s.mirsajjadi@sms.ed.ac.uk}}

\date{\today}
\begin{document}
	\maketitle
	\begin{abstract}
		We show that any non-linear heat equation with scaling critical dimension $-1$ is locally well-posed when its initial condition is taken as the Gaussian free field in fractional dimension $d < 4$. Our results in particular extend the well-posedness results of \cite{cao2021yang,CCHS22_3D} from $d=3$ to the entire subcritical regime.
				\\[.4em]
	\end{abstract}
	\setcounter{tocdepth}{1}
	
	\tableofcontents

\section{Introduction}

Consider the non-linear partial differential equation (PDE)
\begin{equs}\label{eq:A_eq}
	\partial_t A &= \Delta A + F(A,DA) \qquad &\text{on } (0,T)\times \T^n\;,
	\\
	A_0 &= X \qquad &\text{on } \T^n\;, \label{eq:A_ic}
\end{equs}
posed for a function $A\in \CC^\infty ((0,T)\times \T^n,E)$, where $n\geq 1$, $\T^n=\R^n/\Z^n$ is the $n$-dimensional torus, $E$ is a finite-dimensional vector space, and
the initial condition $A_0 = X$ is understood
in the sense that $\lim_{t\downarrow 0}A_t = X$ in a suitable space of distributions that we make precise below
(we use the notation $A_t = A(t,\cdot)$).
Here
$DA\in \CC^\infty ((0,T)\times \T^n,E^n)$ is the spatial derivative of $A$, i.e. $(DA)_i = \partial_i A$,
and $F\colon E\times E^n\to E$ is a polynomial of the form
\begin{equ}[eq:F]
	F(x,y) = B(x,y) + P(x) + Q(x,y)\;,
\end{equ}
where $B\colon E\times E^n \to E$ is bilinear,
$P\colon E^3\to E$ is trilinear and we write $P(x)$ for $P(x,x,x)$,
and $Q\colon E\times E^n \to E$ is a polynomial of the form $Q(x,y) = Q_1(x) + Q_2(y)$ where $Q_1$ is quadratic and $Q_2$ is linear.
Hence, the non-linearity in \eqref{eq:A_eq} can be written heuristically as
\begin{equ}[eq:F_heuristic]
F(A,DA) = A DA + A^3 + A^2 + A + DA + c\;,
\end{equ}
where $c\in E$ is a constant.

Our main result is the following well-posedness for \eqref{eq:A_eq}-\eqref{eq:A_ic}
where we take $X$ to be a Gaussian field satisfying the following Assumption~\ref{ass:GF}.
We note that the Gaussian free field (GFF) in fractional dimension $d$ for any $d\in(2,4)$ satisfies Assumption~\ref{ass:GF}
(see Appendix~\ref{app:GFF}).
We set $\N_0 = \{0,1,\ldots\}$ and, for a multi-index $k\in \N^n_0$, use the shorthand $f^{(k)} = D^k f$.

\begin{assumption}\label{ass:GF}
	Suppose $n\geq 2$,  $d\in(2,4)$,
 	and $X$ is a stationary $E$-valued centred Gaussian field on $\T^n$ with covariance function $C(x)\eqdef\E[X(0)\otimes X(x)]$ such that $C(x)=C(-x)$ and there exists $K>0$ for which $\big|C^{(k)}(x)\big|\leq K |x|^{2-d-|k|}$ for all $x\in\T^n\setminus\{0\}$ and $k\in\N^n_0$ with $|k|=\sum_{i=1}^n k_i\le1$.
\end{assumption}

The conditions $n\geq 2$ and $d<4$ ensure that $2-d>-n$, thus $C$ is integrable and convolution with $C$ is well-defined.
We denote by $\CC^\eta = B^\eta_{\infty,\infty}(\T^n,E)$ the H\"older--Besov space of regularity $\eta$ (see Section \ref{Subsec 1.1}). We also use $*$ to denote spatial convolution (i.e. convolution on $\T^n$) throughout.
Under \assu{ass:GF}, it is classical that $X$ admits a version in $\CC^\eta$ for any $\eta<\dd$, and we let $X$ denote this version henceforth.

\begin{theorem}\label{thm:main}
	Let $\chi\in \CC^\infty(\R^n)$ be compactly supported such that $\chi(x)=\chi(-x)$
	and $\int_{\R^n}\chi(x)\,\mrd x =1$, and denote $\chi^\e = \e^{-n}\chi(x/\e)$.
	Suppose Assumption~\ref{ass:GF} holds and denote $X^\e\eqdef \chi^\e* X$,
	a mollification of $X$ at scale $\e>0$.
	Denote by $A^\e$ the (smooth) solution to \eqref{eq:A_eq}  with initial condition $A^\e_0 = X^\e$.
	Then there exists a random variable $T\in (0,1)$ such that, for any $\eta < (2-d)/2$ and $p\in[1,\infty)$,
	\begin{equ}[eq:convergence]
		\lim_{\e\downarrow0} \Big|\sup_{t\in(0,T)}\{|A^\e_t - A_t|_{\CC^\eta}  + t^{-\eta/2}|A^\e_t - A_t|_\infty\} \Big| =0
		\;,\qquad
		\E [T^{-p}] < \infty\;,
	\end{equ}
	where the limit holds in $L^p(\P)$ and $\P$-almost surely (a.s.)
	and	where $A$ is smooth on $(0,T)\times \T^n$, solves \eqref{eq:A_eq}, and
	satisfies $\lim_{t\downarrow 0} |A_t-X|_{\CC^\eta} = 0$ a.s.
\end{theorem}
\begin{remark}
The proof of Theorem \ref{thm:main} actually shows a stronger result, namely that $A^\eps$ decomposes into a perturbative part, which is a polynomial of the linear solution $\partial_t \Psi = \Delta \Psi, \Psi_0=X^\eps$ and which converges on $(0,1)\times\T^n$ in a distributional space,
and a remainder, which converges on $(0,T)\times\T^n$ as $\e\downarrow0$ in $L^\infty$.
\end{remark}
The regime $d<4$ is subcritical in the following sense.
Working over $\R^n$ and supposing for simplicity that $Q=0$ in \eqref{eq:F},
if $A$ solves \eqref{eq:A_eq}-\eqref{eq:A_ic},
then for any $\lambda>0$,
$\tilde A(t,x) \eqdef \lambda A(t\lambda^2,x\lambda)$ likewise solves \eqref{eq:A_eq}
with initial condition $\tilde X(x) = \lambda X(x\lambda)$.
A Besov space that is invariant under such transformations is the homogenous Besov space $\dot B^{-1}_{\infty,\infty}$.
This means that in $\dot B^{-1}_{\infty,\infty}$, the linearity $\Delta A$ and leading non-linearities $ADA + A^3$ in \eqref{eq:A_eq} and \eqref{eq:F_heuristic}
are of the same strength, which we call the critical regime;
in $\dot B^\eta_{\infty,\infty}$ for $\eta<-1$, the non-linearities dominate in the small scale limit $\lambda\to 0$, which we call supercritical,
while for $\eta>-1$, the linear part dominates, which we call subcritical.
Probabilistically, $\lambda X(x\lambda)$ is the scaling that keeps the GFF in dimension $d=4$ invariant in law.
These heuristics, while for $\R^n$ and homogeneous spaces, are still an important guiding principle
for other manifolds, such as $\T^n$, and for inhomogeneous spaces.

Theorem \ref{thm:main} therefore covers the entire subcritical regime $d<4$.
There is much interest in the study of local and global well-posedness for PDEs with random initial data; for a small selection of works, see \cite{Bourgain94, Burq_Tzvetkov_08_I,Burq_Tzvetkov_08_II, NPS_13_NS_random,Oh_Pocovnicu_16,Pocovnicu17,BOP19,OTW_20_4NLS,Sun_Tzvetkov_20,KLS_23,Camps_Gassot_23,Bringmann_24_Hartree,BDNY_24_Gibbs,DNY_24_Gibbs,OOT_24_NLW_GFF} and the references therein,
as well as the works in Section \ref{sec:lit} that we discuss in more detail.
Our work contributes to the body of probabilistic well-posedness results
and is the first, to our knowledge, to address non-linear heat equations in the entire subcritical regime.

We believe Theorem \ref{thm:main} is interesting in its own right and also takes a step towards the critical case $d=4$, which is often the most challenging.
For $d=4$, no such general result is expected and any sensible solution theory for an equation of the type \eqref{eq:A_eq}-\eqref{eq:A_ic} is expected to leverage the specific structure of the equation (such as the sign of the non-linearities, which we do not use).

\begin{remark} \label{rem:generalisationa}
We work with the structure \eqref{eq:F} of the non-linearity $F$ for simplicity.
However, our results extend with minimal changes to other polynomials $F(A,DA)$ with leading terms of the form $A^{2k+1} + A^k DA$.\footnote{This form of $F$ ensures that the most singular Feynman diagrams arising from Picard iterates of \eqref{eq:A_eq}-\eqref{eq:A_ic} vanish, which is important for our argument; see \eqref{eq:I^sig} and the discussion around it.
For general polynomials, we expect that $A^\eps$ may require renormalisation to converge.}
E.g. our method can handle $F=A^{k}$ with $k\geq3$ an odd integer, in which case the lower bound on $\eta$ becomes $\eta-2 < k\eta \Leftrightarrow \eta > 2/(1-k)$ which translates to $(2-d)/2 > 2/(1-k)\Leftrightarrow d<2+4/(k-1)$.
We could similarly change $\partial_t - \Delta$ in \eqref{eq:A_eq} to any parabolic operator (e.g. $\partial_t + \Delta^2$) with a sufficiently well-behaved fundamental solution
(see \eqref{eq:heat_flow_estimates} and Lemma \ref{lem:heat-flow-estimates} which, together with spatial evenness and translation invariance of the heat kernel, are the only properties of the heat operator that we use). 
\end{remark}

\subsection{Method of proof}

Our method of proof is divided into two parts, one deterministic and one probabilistic.

In the deterministic part (Section \ref{sec:well-posedness}),
we define a metric space $(\CI,\Theta)$ of distributions
to which the solution map for \eqref{eq:A_eq}-\eqref{eq:A_ic} extends in a locally Lipschitz manner (see Theorem \ref{thm:lwp}).
The space $\CI$ is defined as the set of all distributions $X$ for which sufficiently many Picard iterates $\{A^{(i)}\}_{i=0}^k$ of \eqref{eq:A_eq}-\eqref{eq:A_ic} starting with $A^{(0)}_t=\mre^{t\Delta} X$
(i.e. terms of the Wild expansion \cite{Wild51})
are well controlled.
While $\CI$ embeds into $\CC^\eta$ for $\eta<\dd$ as above,
a generic element $X\in \CC^\eta$ will not be in $\CI$ as the Picard iterates of $X$ may not even be well-defined.
This part is inspired by several earlier works that we discuss in Section \ref{sec:lit}.

In the probabilistic part, we show that mollifications of the Gaussian field $X$ as in Theorem \ref{thm:main} converge inside $(\CI,\Theta)$ (see Theorem \ref{thm:moment-estimates}).
This part is the main contribution of our paper and our arguments differ from previous works.
We use Feynman diagrams to estimate the moments of the various terms that appear in the Picard iterations of \eqref{eq:A_eq}-\eqref{eq:A_ic}.
As we are in the subcritical regime, the number of such terms is finite.

Our main tool in deriving these moment estimates is a combinatorial algorithm that deletes edges of a Feynman diagram in a way that allows us to proceed by induction.
This algorithm is carried out in the proof of the key Theorem \ref{thm:general-bounds}
and we apply it to estimate all but the most singular Feynman diagrams appearing from the Picard expansion.
For the most singular diagrams, which naive power-counting suggests are non-integrable, we show that their contributions vanish by symmetry (this is the only part where we require the covariance $C$ and mollifier $\chi$ to be even).
Apart from these combinatorial arguments,
we rely only on heat convolution estimates.
Our approach is therefore relatively lightweight and readily generalises to other equations (see Remark \ref{rem:generalisationa}).

We highlight that our stochastic estimates in Theorems \ref{thm:moment-estimates} and \ref{thm:general-bounds} bear similarity to (but neither imply nor are implied by) the bounds on generalised convolutions of \cite{hairer2018class} and the BPHZ theorems for singular SPDEs~\cite{chandra2016analytic,LOTT_24_SG, HS_24_SG}.
However, our proof is different and we notably avoid the use of multi-scale analysis as in~\cite{hairer2018class,chandra2016analytic}
or a spectral gap inequality as in \cite{LOTT_24_SG, HS_24_SG}.
Having said that, we believe it would be interesting to investigate if methods from the aforementioned works could provide an alternative proof of our main result.\footnote{Since completing
this article, we discovered an argument based on the spectral gap inequality which recovers a special case of our main stochastic estimate Theorem \ref{thm:moment-estimates} but which does not rely on Gaussianity of the noise. We intend to present this argument in another work.}

\subsection{Related works and motivations}
\label{sec:lit}

\textbf{The case $\bm{d\leq 3}$.}
	In the regime $d<3$, Theorem \ref{thm:main} is a simple consequence of
the local well-posedness of \eqref{eq:A_eq}-\eqref{eq:A_ic}
for $X\in \CC^\eta$ with $\eta>-\frac12$.
More precisely,
for any $K>0$, there exists $T>0$ sufficiently small depending polynomially on $K$,
such that the solution map $\{X\in \CC^\eta\,:\,|X|_{\CC^\eta}<K\} \ni X\mapsto A\in \CC((0,T),\CC^\eta)$ is well-defined and Lipschitz,
see e.g. \cite[Appendix~B]{Chevyrev22_norm_inf}.
Theorem \ref{thm:main} for $d<3$ thus follows
from the classical fact that mollifications of the GFF converge in $\CC^\eta$ for all $\eta<(2-d)/2$.

The case $d=3$ is more challenging and was handled
in \cite{cao2021yang,CCHS22_3D}.\footnote{\cite{cao2021yang,CCHS22_3D} considered only the DeTurck--Yang--Mills heat flow, which is a specific example of \eqref{eq:A_eq},
but the arguments therein extend without issue to generic equations of the form \eqref{eq:A_eq}.}
A difficulty in this case is that, for generic $B$ in \eqref{eq:F} (more precisely for $B$ that is not a total derivative), 
\eqref{eq:A_eq}-\eqref{eq:A_ic} exhibits norm inflation around every point in $\CC^\eta$ for all $\eta\leq-\frac12$ \cite{Chevyrev22_norm_inf} (see also \cite{COW22} for the case of the cubic heat equation),
which is a strong form of local ill-posedness.
In fact, unlike $d<3$,
the case $d=3$ cannot be handled solely with `linear' arguments in the following sense: there exists no Banach space of distributions $\mcX\subset \mcD'(\T^n,E)$ such that
(a)
$\CC^\infty\hookrightarrow \mcX$,
(b)
Fourier truncations of the 3D GFF converge in $\mcX$, and
(c) the solution map for \eqref{eq:A_eq}-\eqref{eq:A_ic} extends continuously from $\CC^\infty$ to $\mcX$ locally in time, see \cite{Chevyrev22_norm_inf}.

Theorem \ref{thm:main} therefore directly extends these results from \cite{cao2021yang,CCHS22_3D} for $d=3$
to the subcritical regime $d<4$.
Our metric space $(\mcI,\Theta)$ is inspired by these works but our probabilistic estimates are substantially different.

A key motivation of \cite{cao2021yang,CCHS22_3D} is to construct a candidate state space for the Yang--Mills (YM) measure on $\T^3$ (see also \cite{CG13,CG15,Gross2016FiniteAction,cao2024yang_state} and the survey \cite{Chevyrev22_YM};
note that the construction of the YM on $\T^3$ itself is currently an open and difficult problem).
This space supports the 3D GFF and contains distributions at which the DeTurck--Yang--Mills heat flow is locally well-posed and is therefore necessarily non-linear (cf. the GFF and YM measure on $\T^2$, for which a canonical \emph{linear} state space was defined in \cite{Chevyrev19YM, CCHS22_2D}).
The proposed state space admits a canonical notion of gauge equivalence and gauge-invariant observables
and it is shown in \cite{CCHS22_3D} that there exists a canonical Markov process associated to the stochastic quantisation equations of YM on the corresponding space of gauge orbits
(see also \cite{CCHS22_2D,ChevyrevShen23,BC23, BC24} for related work in 2D).
We speculate that the metric space $(\CI,\Theta)$ may similarly serve as a state space for quantum field theories for which naive regularisation, such as smearing against test functions, is not geometrically meaningful.

\medskip

\textbf{Weak coupling.}
Two further works related to ours are \cite{Hairer_Le_Rosati_22,GRZ23}.
In \cite{Hairer_Le_Rosati_22}, 
the Allen--Cahn equation $\partial_t A = \Delta A + A -A^3$ is studied on $[0,\infty)\times\R^n$ with initial condition $X^\eps \eqdef \eps^{\frac{n}{2}-\alpha} \chi^\eps* \xi$,
where $\alpha\in (0,1)$ and $\xi$ is a white noise on $\R^n$.
One can check that $X^\eps$ satisfies the bounds of \assu{ass:GF} uniformly in $\eps \in (0,1)$ with $d$ therein such that $\alpha = (d-2)/2$,
so that $\alpha \in (0,1)\Leftrightarrow d \in (2,4)$.
Moreover, $X^\eps\to 0$ in $\CC^\eta$ for all $\eta<\dd$. 
This suggests, and it is indeed simple to verify, that a minor variation of our results implies that the solutions $A^\eps_t$ converge to $0$ for any fixed $t>0$,
which agrees with \cite{Hairer_Le_Rosati_22} (the main result of \cite{Hairer_Le_Rosati_22} further analyses the long-time behaviour of $A^\eps$).

A similar result is obtained in \cite{GRZ23} for $n=2$ and $X^\eps = \hat\lambda |\log\eps|^{-1/2}\mre^{\eps^2\Delta}\xi$, where $\hat\lambda>0$ is a sufficiently small coupling constant.
The main result of \cite{GRZ23} is that $A^\eps_t \to 0$ for all $t>0$ sufficiently small and establishes Gaussian fluctuations around this limit.
Here $X^\eps\to 0$ in $\CC^{\eta}$ for $\eta<-1$,
so this result is weakly critical.
In particular, this means that it is not sufficient to analyse a fixed truncation of the Wild expansion as $\eps\downarrow0$ and is thus outside the scope of our methods.

While our work and the works \cite{Hairer_Le_Rosati_22,GRZ23} share similarities (e.g. they all analyse the Wild expansion of \eqref{eq:A_eq}), none of these works implies the other.

\medskip

\textbf{Singular SPDEs.}
There have been significant advances in the past decade in understanding dynamics governed by stochastic PDEs, see the seminal articles \cite{Hairer14,gubinelli2015paracontrolled} as well as \cite{DPD02_NS,Kupiainen16,OW19,Duch25}.
In highly singular regimes, there are obstructions to naively using linear distributional spaces as boundary conditions to SPDEs, see \cite{BCCH21,CCHS22_3D}.
For example, for the parabolic dynamical $\Phi^4_d$ model
with $d\ge\frac{10}{3}$, one cannot start the equation from $\CC^{\eta}$ for $\eta<\dd$, see \cite[Sec.~2.8.2]{BCCH21},
which is the optimal H\"older--Besov space in which the solution takes values.
A similar problem appears for the stochastic YM flow in 3D \cite{CCHS22_3D} and was overcome with a simple version of our metric space $(\CI,\Theta)$.

Our results therefore take a step towards understanding suitable state spaces for solutions to singular SPDEs close to criticality.
Specifically, we expect that analogues of our metric space $(\mcI,\Theta)$ can be used as spaces of initial conditions for singular SPDEs in the entire subcritical regime
with potential to study properties such as Markovianity, which is challenging in highly singular regimes \cite{HS22_Support}.

\subsection{Notations and preliminaries}\label{Subsec 1.1}

\textbf{Sets and graphs.} 
	We let $|Z|$ denote the cardinality of a set $Z$.
	We denote $\N = \{1,2,\ldots\}$ and $\N_0 = \N\cup \{0\}$.
	For any $ N\in \N$, we denote $ [N] = \{1,2,\dots, N\} $.
	For $k = (k_1,\ldots,k_n)\in\N_0^n$, we write $|k| = \sum_{i=1}^n k_i$.
	
	All graphs we consider are undirected and finite. We may write $G=(V,E)$ for a graph $G$, which means that $V,E$ are its respective vertex set and edge set.
	A \emph{forest} is a graph without cycles. Every connected component of a forest is called a \emph{tree}.
	We treat each edge $e$ of a graph as a $2$-element subset of the vertex set, and we write $e=(x,y)$ or $e=\{x,y\}$ for $x,y$ the vertices incident with the edge $e$.

	We equip $\R^n$ with the Euclidean norm $|\cdot|$ and for $x\in\T^n = \R^n/\Z^n$, we let $|x|$ denote the corresponding geodesic distance of $x$ from $0$ (which is simply the absolute value of $x$ if we identify $\T^n$ with $[-\frac12,\frac12)^n$ as a set).
	
	\textbf{Relations.}
	We write $X\lesssim Y$ to mean that there exists a constant $K>0$ such that $X\leq KY$. If $X$ and $Y$ are functions, then $K$ is assumed uniform over a given set of variables which is either specified or is clear from the context. If we have both $X\lesssim Y$ and $Y\lesssim X$, then we write $X\asymp Y$.
	We use the standard notation $X\wedge Y = \min\{X,Y\}$.

\textbf{Function spaces.} 
	We fix throughout the paper a number $d\in (2,4)$, a finite dimensional inner product space $(E,\scal{\cdot,\cdot})$
	and functions $B,P,Q$ as described after \eqref{eq:A_eq}-\eqref{eq:A_ic}.
	We suppose without loss of generality that $P$ is symmetric in its 3 arguments.

	We let $\CC,\CC^\infty,\CD'$ denote the spaces of continuous functions, smooth functions, and (Schwartz) distributions respectively.
	Unless otherwise stated, all function and distribution spaces have domain $\T^n$ and values in $E$, e.g. $\CC^\infty$ means $\CC^\infty(\T^n,E)$.

\begin{definition}\label{def:heat-kernel}
	For $X\in\CD'$, we denote by
	\begin{equ}
		\CP X \in \CC^\infty((0,\infty)\times \T^n, E)\;,
		\qquad \CP_t X \eqdef \mre^{t\Delta}X\;,
	\end{equ}
	the solution to the heat equation with initial condition $X$.
	Furthermore, for $X\in \CC^\infty([0,T]\times \T^n,E)$
	we denote by
	\begin{equ}[eq:heat_flow_space_time]
		\CP \star X \in \CC^\infty([0,T]\times \T^n, E)\;,
		\quad \CP_t \star X
		\eqdef \int_0^t \mre^{(t-s)\Delta} X_s \mrd s\;,
	\end{equ}
	the solution to the inhomogeneous heat equation with source $X$ and zero initial condition.
\end{definition}

For $\eta < 0$ we denote by $\CC^\eta$ the (inhomogeneous) H\"older--Besov space of $E$-valued distributions, defined
as the Banach space of all $X \in \mcD'(\T^n,E)$  such that
\begin{equ}
	|X|_{\CC^\eta}\eqdef 
	\sup_{\phi \in \mcB^r}\sup_{x\in\T^n}
	\sup_{\lambda\in (0,1]} \lambda^{-\eta}|\scal{X,\phi^\lambda_x}| < \infty\;,
\end{equ}
where $r=-\floor{\eta}+1$, $\mcB^r$ is the set of all  $\phi \in \CC^\infty(\T^n,E)$ with support in the ball $\{|z|<\frac14\}$ and $|\phi|_{\CC^r}\leq 1$,
and $\phi^\lambda_x (z) = \lambda^{-n}\phi((z-x)/\lambda)$.
Here and below we write $\floor x$ for the floor of $x\in \R$.

We further denote $\CC^0=L^\infty$, and for $\eta>0$, we let $\CC^{\eta}$ be the usual Banach space of functions with derivatives of order $k\eqdef \roof{\eta}-1$
(where $\roof{\eta}$ is the ceiling of $\eta$)
and whose $k$-th order derivatives are $(\eta-k)$-H\"older continuous,
e.g. $\CC^1$ consists of $X$ such that $|DX|_\infty<\infty$, i.e. Lipschitz continuous functions.

We recall the heat flow estimates for
$\eta\leq 0$ and $\gamma\geq 0$ (see, e.g.~\cite[Lem.~A.7]{gubinelli2015paracontrolled}, or the proof of~\cite[Thm.~2.34]{BookChemin}),
uniformly in $t\in(0,1)$,
\begin{equ}[eq:heat_flow_estimates]
	t^{\gamma/2}|\CP_t X|_{\CC^{\eta+\gamma}}\lesssim  |X|_{\CC^\eta}\;.
\end{equ}
\section{Deterministic well-posedness} 
\label{sec:well-posedness}

In this section, we introduce a (non-linear) metric 
space of distributions to which the solution map $X\mapsto A$ for \eqref{eq:A_eq}-\eqref{eq:A_ic}
extends in a locally Lipschitz way.
This section is entirely deterministic. We fix an integer $n\geq 1$ throughout.

We begin by defining a set $\CT$
of trees that encode the most singular terms that appear in the Picard iterations for \eqref{eq:A_eq}-\eqref{eq:A_ic}.

\begin{definition}\label{def:trees}
	For a tree $\tau$ (in the graph theoretic sense), we let $V_\tau$ and $E_\tau$ denote the vertex and edge set respectively of $\tau$ (recall that all graphs we consider are finite).
	\begin{itemize}
		\item A \emph{rooted tree} is a tree $\tau$ with a distinguished vertex $ \rho_\tau \in V_\tau$, called the root.
		\item A \emph{leaf} of a rooted tree $\tau$ is a vertex that either (a) has degree $1$ and is not the root, or (b) has degree $0$ (in which case it is necessarily the root and $\tau$ is just a single vertex).
		\item A \emph{labelled tree} is a rooted tree $\tau$ along with a map $\mfe\colon E_\tau \to \{I,I'\}$,
		where elements of $\{I,I'\}$ are formal symbols which we interpret as labels.	
		We denote by $\CL$ the set of all labelled trees.	
	\end{itemize}
	When drawing elements of $\CL$, we always draw the root at the bottom, edges with label $I$ are denoted by a thin line $\<I>\,$, and edges with label $I'$ are denoted by a thick line $\<I'>\,$.
	Vertices that are leaves are marked with a circle $\<Xi>$ and non-leaf vertices are not marked.
\end{definition}

\begin{definition}\label{def:CT_trees}
	The set $\CT \subset \CL$ is defined inductively as follows.
	We let $\<Xi> \in \CT$ (the tree with a single vertex).
	Then products of the form $\prod_{i=1}^3 I(\tau_i)$ for all $\tau_i\in\CT$ are included in $\CT$,
	where $I(\tau_i)$ is the tree obtained by adjoining the root of $\tau_i$ to a new root with an edge labelled with $I$, e.g. $I(\<Xi>) = \<IXi>$,
	and the `product' $\prod_{i=1}^3 I(\tau_i)$ is obtained by merging the roots of $I(\tau_1),I(\tau_2),I(\tau_3)$,
	e.g. $I(\<Xi>)I(\<Xi>)I(\<Xi>) = \<IXi^3>$.
	
	For all $\tau_1,\tau_2\in\CT$, we furthermore impose that $I(\tau_1)I'(\tau_2)$ is in $\CT$ 
	where $I'(\tau_2)$ is
	the tree obtained by adjoining the root of $\tau_2$ to a new root with an edge with label $I'$, e.g. $I'(\<Xi>) = \<I'Xi>$,
	and $I(\tau_1)I'(\tau_2)$ is again obtained by merging the roots $I(\tau_1)$ and $I'(\tau_2)$,
	e.g. $I(\<Xi>)I'(\<Xi>)=\<IXiI'Xi>$.

	Elements of $\CT$ are called \emph{singular} trees - see Figure~\ref{fig:trees} for some examples.
\end{definition}

\begin{figure}[H]
	\centering
	\begin{subfigure}[t]{0.2\textwidth}
		\centering
		\begin{tikzpicture}[scale=0.46]
			\node at (-.55,1) [var] (a) {}; 
			\node at (.55, 1)  [var] (b)  {};  
			\node at (-.75,-.20) [var] (c) {}; 
			\node at (.75,-.20) [var] (d) {}; 
			\draw[kernels2] (0,0) -- (b);
			\draw (0,0) -- (a);
			\draw[semithick] (0,0) -- (0,-1.1);
			\draw (c) -- (0,-1.1);
			\draw (d) -- (0,-1.1);

		\end{tikzpicture}
		\caption{}
	\end{subfigure}
	\begin{subfigure}[t]{0.21\textwidth}
	\centering
	\begin{tikzpicture}  [scale=0.43]		 
			\node at (-1,1.88)  [var] (a)  {};   
			\node at (-.25,1.98) [var] (b) {};  
			\node at (.5,1.88) [var] (c) {};  
			\node at (1.12,1.02) [var] (d) {};  
			\node at (2.,.55) [var] (e) {};  
			\node at (2.6,1.25) [var] (f) {};  
			\node at (3.5,1.25) [var] (g) {};  
			\draw (-.2,.96) -- (a); 
			\draw (-.2,.96) -- (b); 
			\draw (-.2,.96) -- (c); 
			\draw (-.2,.96) -- (.5,.1); 
			\draw[kernels2] (.5,.1) -- (d); 
			\draw (2.,-.5) -- (e); 
			\draw (.5,.1) -- (2.,-.5); 
			\draw (2.,-.5) -- (3.03,.22); 
			\draw (3.03,.22) -- (f); 
			\draw[kernels2] (3.03,.22) -- (g); 
		\end{tikzpicture} 

	\caption{}
	\end{subfigure}
	\begin{subfigure}[t]{0.19\textwidth}
		\centering
		\begin{tikzpicture}[scale=0.46]		
			
			\node at (5.05,.9)  [var] (b)  {};  
			\node at (3.95,.9) [var] (a) {}; 
			\draw (4.5,-.1) -- (b);
			\draw[kernels2] (4.5,-.1) -- (a);
			\draw[semithick] (4.5,-.1) -- (4.5,-1);

		\end{tikzpicture}
		\caption{}
	\end{subfigure}
	\begin{subfigure}[t]{0.22\textwidth}
	\centering
		\begin{tikzpicture}[scale=0.45]		
			\node at (6.2+1,.85) [var] (a) {}; 
			\node at (6.2+.35, 1.05)  [var] (b)  {}; 
			\node at (6.2-.35, 1.05)  [var] (c)  {};  
			\node at (6.2-1, .85) [var] (d) {}; 
			\node at (6.2-1.40,.25) [var] (e) {}; 
			\node at (6.2-2., -.0) [var] (f) {}; 
			\node at (6.2-2.26,-.58) [var] (g) {}; 
			\draw (6.2,0) -- (a);
			\draw (6.2,0) -- (b);
			\draw (6.2,0) -- (c);
			\draw (6.2,0) -- (d);
			\draw (6.2,0) -- (6.2-1,-.9);
			\draw (e) -- (6.2-1,-.9);
			\draw (f) -- (6.2-1,-.9);
			\draw (g) -- (6.2-1,-.9);

		\end{tikzpicture}
		\caption{}
	\end{subfigure}
	\caption{\small The labelled trees in sub-figures~(a)-(b) are elements of $\CT$ while those in (c)-(d) are not.
	}
	\label{fig:trees}
\end{figure}
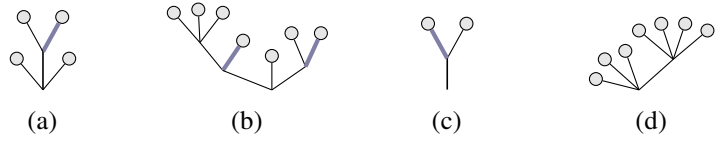
For $X\in \CC^\infty$, denote $X^{\<Xi>}=\delta_{t=0}\otimes X \in \CD'(\R\times\T^n,E)$, i.e. the distribution $X$ concentrated on the time $t=0$ hyperplane.	
We extend notation \eqref{eq:heat_flow_space_time} by writing $\CP_t \star X^{\<Xi>}=\CP_t X \in \CC^\infty([0,\infty)\times\T^n,E)$. 
For $\tau\in \CT\setminus\{\<Xi>\}$, 
we also define  $X^\tau \in \CC^\infty([0,\infty)\times \T^n,E)$ inductively by
\begin{equs}[eq:X_tau_def]
	X^{I(\tau_1)I'(\tau_2)}_t &= B(\CP_t \star X^{\tau_1},D \CP_t \star X^{\tau_2})\;,\\
	X^{I(\tau_1)I(\tau_2)I(\tau_3)}_t &= P(\CP_t \star X^{\tau_1}, \CP_t \star X^{\tau_2},\CP_t \star X^{\tau_3})\;, 
\end{equs}
where we recall the notation \eqref{eq:heat_flow_space_time}.
For example,
\begin{equ}
	P(\CP_t X)= X_t^{\<IXi^3>}, 
	\quad
	B\big(\CP_t \star B(\CP X , D\CP X), D\CP_t \star B(\CP X,D\CP X)\big) = X^{\<I[IXiI'Xi]I'[IXiI'Xi]>}_t.
\end{equ}

\begin{remark}\label{rem:isom}
	Our trees are combinatorial meaning that we do not consider an order for edges leaving a single vertex; more precisely, we identify any two trees which differ by a graph isomorphism preserving roots and labels of edges.
	Note that $X^{I(\tau_1)I(\tau_2)I(\tau_3)}$ in \eqref{eq:X_tau_def} is well-defined on the isomorphism class of $I(\tau_1)I(\tau_2)I(\tau_3)$ since we assumed $P$ is symmetric.
\end{remark}

For $\tau\in \CT$, we denote by $\noise{\tau}$ the number of leaves in $\tau$,
and, for $ i\in\N$, we define
\begin{equ}
	T_i\eqdef \{\tau\in\CT\,:\; \,\noise{\tau}=i\}\;.
\end{equ}
Furthermore, for $ N\in \N$, we denote  
\begin{equ}
\CT^N_{\<Xi>} \eqdef \bigcup_{j=1}^N T_j\;,
\qquad
\CT^N \eqdef \CT^N_{\<Xi>}\setminus\{\<Xi>\}
\;.
\end{equ}
	
\begin{definition}
\label{def:norms+init+}
	For a Banach space $(W,|\cdot|)$ and $\delta\in\R$, let $\CC_{\delta}W$  denote the Banach space of continuous functions $f\colon (0,1)\to W$ with norm
	\begin{equ}
		|f|_{\CC_\delta W} \eqdef \sup_{t\in(0,1)}t^{\delta}|f_t|\;.
	\end{equ}
	Let $X,Y\in \CC^\infty$	and 
	$\delta,\beta \in\mathbb{R}$.
	For $\tau \in \CT$ with $\tau\neq\<Xi>$, define the pseudometric
	\begin{align*}
		\fancynorm{X;Y}_{\tau;(\beta,\delta)} &\eqdef |X^\tau-Y^{\tau}|_{\CC_\delta 	\CC^\beta}\;.
	\end{align*}
	For $ N\in \N$, $\omega_{\<Xi>}\in\R$, and vectors $ \bbeta=(\beta_\tau)_{\tau\in\CT^N}\in \mathbb{R}^{\CT^N} $ and $ \bdelta=(\delta_\tau)_{\tau\in\CT^N}\in \mathbb{R}^{\CT^N}$,
	define
	\begin{equ}[eq:Theta]
		\Theta_{\omega_{\<Xi>},\bbeta,\bdelta}(X,Y)  \eqdef 
		|X-Y|_{\CC^{\omega_{\<Xi>}}}+
		\sum_{\tau\in\CT^N}\fancynorm{X;Y}_{\tau;(\beta_\tau,\delta_\tau)}\,.
	\end{equ}
	Let $\init\equiv\init_{\omega_{\<Xi>},\bbeta,\bdelta}$
	denote the completion of smooth functions under the metric $\Theta\equiv \Theta_{\omega_{\<Xi>},\bbeta,\bdelta}$.
	We further denote
	$\Theta(X) = \Theta(X,0)$.
\end{definition}

\begin{definition}\label{def:CI}
	Consider $N \in\N$,  vectors $ \bbeta=(\beta_\tau)_{\tau\in\CT_{}^N}, \bdelta=(\delta_\tau)_{\tau\in\CT_{}^N}\in \mathbb{R}^{\CT_{}^N}$, and $\omega_{\<Xi>}\le 0$.
	Define, for $\tau\in\CT^N$, $\omega_\tau \eqdef \beta_\tau-2\delta_\tau+2$.
	We say that $\omega_{\<Xi>},\bbeta,\bdelta$ satisfy condition~\eqref{eq:CI} if
	\begin{equs}\label{eq:CI}
		\begin{aligned}
			&\forall \tau\in\CT_{}^N\,:\quad \beta_\tau \in (-1,0)\;,\quad \delta_\tau < 1\;,
			\quad \omega_\tau\le 0\;, \\
			&\omega \eqdef \min\{\omega_\tau \,:\, \tau\in\CT_{\<Xi>}^N\} > -1\;,
			\\
			&\alpha \eqdef \min\{\omega_{\tau_1}+\omega_{\tau_2} \,:\, \tau_1,\tau_2 \in\CT_{\<Xi>}^N\,,\; \noise{\tau_1}+\noise{\tau_2}>N\} > -1\;,
			\\
			&\gamma  \eqdef \min\Big\{\sum_{i=1}^3\omega_{\tau_i} \,:\, \tau_i \in\CT_{\<Xi>}^N\,,\; \sum_{i=1}^3\noise{\tau_i}>N\Big\} > -2\;.
		\end{aligned}\tag{$\CI$}
	\end{equs}
\end{definition}
We remark that, for $\tau\neq\<Xi>$, $\beta_\tau>-1$ and $\omega_\tau\le 0$ together imply $\delta_\tau>1/2$.
\begin{remark}
	One should think of $\omega_\tau$ as the regularity of the space-time function $\CP \star X^\tau$ for any $\tau\in \CT^N_{\<Xi>}$ (under the parabolic scaling)
	-- see \eqref{eq:PtX_bound} below.
	The threshold $\omega>-1$ implies that $\int_0^t |(\CP_s \star X^{\tau})\otimes(\CP_s \star X^{\sigma})|_\infty \mrd s < \infty$.
	Likewise $\alpha>-1$ and $\gamma>-2$
	imply  $\int_0^t |X^{I(\tau_1)I'(\tau_2)}_s|_\infty \mrd s <\infty$
	and $\int_0^t |X^{I(\tau_1)I(\tau_2)I(\tau_3)}_s|_\infty \mrd s < \infty$ respectively for $\tau_i$ satisfying the respective conditions in \eqref{eq:CI}.
\end{remark}
We next show that $\mcI$, under condition \eqref{eq:CI},
can be realised as a subset of $\CC^{\omega_{\<Xi>}}$.
This is not entirely obvious because, a priori, elements like $X^{\<I[IXiI'Xi]I'Xi>}$
are not even well-defined for generic $X\in \mcD'$ (although $X^{\<Xi>},X^{\<IXiI'Xi>},X^{\<IXi^3>}$ are). 
	
\begin{proposition}\label{prop:closable_graph}
	Let $\omega_{\<Xi>}\in\R$, $\bbeta=(\beta_\tau)_{\tau\in\CT^N}, \bdelta=(\delta_\tau)_{\tau\in\CT^N}\in \mathbb{R}^{\CT^N}$
	such that $\beta_\tau\in(-1,0)$ and $\delta_\tau<1$ 
 for all $\tau\in\CT^N$. 
	Then the map $\CC^\infty \ni X \mapsto (X^\tau)_{\tau\in \CT^N_{\<Xi>}}$ has a closable graph in $\CC^{\omega_{\<Xi>}}\times (V_{\tau})_{\tau\in\CT^N_{\<Xi>}}$, where $V_{\<Xi>}= \CC^{\omega_{\<Xi>}}$ and $V_{\tau}=  \CC_{\delta_\tau}\CC^{\beta_\tau}$
	for $\tau\neq\<Xi>$.
In particular, $\CI$ is continuously embedded into $\CC^{\omega_{\<Xi>}}$.
\end{proposition}
	
\begin{proof}
	Suppose that $(X,(X^\tau)_\tau)$ and $(Y,(Y^\tau)_\tau)$ are two Cauchy sequences (we do not write the index of the sequence here) in
	$\CC^{\omega_{\<Xi>}}\times (V_{\tau})_{\tau\in\CT^N_{\<Xi>}}$ such that $X,Y$ converge to the same element in $\CC^{\omega_{\<Xi>}}$.
	We claim that $(X^\tau)_\tau$ and $(Y^\tau)_\tau$ converge also to the same limit, from which the proof follows.

	We proceed by induction. The base case $\tau=\<Xi>$ is true by assumption since $V_{\<Xi>}=\CC^{\omega_{\<Xi>}}$. 

	For the inductive step, let 
	$\tau = I(\sigma)I'(\bar\sigma)$ where the claim is true for $\sigma,\bar\sigma$.
	We will show that $X^\tau$ and $Y^\tau$ converge to the same limit in $V_\tau$.
	For $t>0$, if $\sigma\neq \<Xi>$,
	\begin{equs}
		|\CP_t \star X^\sigma - \CP_t \star Y^\sigma|_\infty &\leq
		\int_0^t |\CP_{t-s} (X^\sigma_s-Y^\sigma_s)|_\infty \mrd s
		\\
		&\lesssim
		\int_0^t (t-s)^{\beta_\sigma/2} s^{-\delta_\sigma}|X^\sigma-Y^\sigma|_{V_\sigma} \mrd s
		\\
		&\asymp t^{\frac{\omega_\sigma}{2}}|X^\sigma-Y^\sigma|_{V_\sigma}
		\\
		&\to 0\;,
	\end{equs}
	where we used \eqref{eq:heat_flow_estimates} in the second line,
	$\beta_\sigma>-2$ and $\delta_\sigma<1$ in the third line,
	and the induction hypothesis in the final line.
	If $\sigma=\<Xi>$, we obtain the same bound $|\CP_t X^\sigma - \CP_t Y^\sigma|_\infty \lesssim t^{\frac{\omega_\sigma}{2}}|X^\sigma-Y^\sigma|_{V_\sigma} \to 0$
	directly from \eqref{eq:heat_flow_estimates}.
	Likewise,
	\begin{equ}
		|D\CP_t \star X^{\bar\sigma} - D\CP_t \star Y^{\bar\sigma}|_\infty\lesssim t^{\frac{\omega_{\bar\sigma}-1}{2}}|X^{\bar\sigma}-Y^{\bar\sigma}|_{V_{\bar\sigma}}
		\to 0\;,
	\end{equ}
	where we now use $\beta_{\bar\sigma}>-1$ and $\delta_{\bar\sigma}<1$ in the case that $\bar\sigma\neq\<Xi>$.
	Therefore, for $t>0$,
	\begin{equs}
		|X^\tau_t-Y^\tau_t|_{\infty} &=
		|B(\CP_{t} \star X^\sigma, D\CP_{t} \star X^{\bar\sigma}) - B(\CP_{t} \star Y^\sigma,
		D\CP_{t} \star Y^{\bar\sigma})|_\infty
		\\
		&\lesssim  t^{\frac{\omega_\sigma + \omega_{\bar\sigma}-1}{2}} o(1)\to 0\;.
	\end{equs}
	Because $X^\tau$ and $Y^\tau$ are assumed to be Cauchy in $V_\tau$, the above shows that they necessarily have the same limit, which completes the inductive step for $\tau = I(\sigma)I'(\bar\sigma)$.
	The case $\tau = \prod_{i=1}^3 I(\sigma_i)$ is completely analogous.
\end{proof}
\begin{remark}
	It may appear natural to substitute both $|\!\cdot\!-\!\cdot\! |_{\CC^{\omega_{\<Xi>}}}$ and $\fancynorm{\cdot;\cdot}_{\tau;(\beta,\delta)}$, $\tau\in \CT\setminus \{\<Xi>\}$, in Definition \ref{def:norms+init+} by the weaker pseudometrics
	\begin{align*}
	\sup_{t\in(0,1)} \{t^{-\frac{\omega_\tau}{2}} | \CP_t \star X^{\tau}- \CP_t\star Y^{\tau}|_{\infty} + t^{\frac{1-\omega_\tau}{2}}| D\CP_t\star X^{\tau}- D\CP_t \star Y^{\tau}|_{\infty} \}\;,
	\end{align*}
	i.e. to only look at $X$ and $Y$ \emph{after} convolution with the heat semi-group.
	An advantage of this would be that we can treat the cases $\tau=\<Xi>$ and $\tau\neq \<Xi>$ in the same way (cf. Definition~\ref{def:norms+init+}); furthermore, it is only this control on $X^{\tau}$ that we use in the local well-posedness result (Theorem~\ref{thm:lwp}).
	A disadvantage, however, is that it is unclear if the analogue of Proposition~\ref{prop:closable_graph} remains true if $X^\tau$ is replaced by $\CP \star X^{\tau}$,
	so the resulting metric space might not be realisable as a subspace of distributions.
\end{remark}
Denote
\begin{equ}[eq:SN_def]
	\CS^{N}_{\<Xi>}X \eqdef
	\sum_{\tau\in\CT^N_{\<Xi>}} c_\tau X^\tau\;,\qquad
	\CS^{N}X \eqdef \CS^N_{\<Xi>} X-X^{\<Xi>} = \sum_{\tau \in \CT^N} c_\tau X^\tau\;,
\end{equ}
where $(c_\tau)_{\tau\in\CT^N_{\<Xi>}}$ are combinatorial constants with $c_{\<Xi>}=1$ that satisfy the recursion
\begin{equ}
c_{I(\tau_1)I'(\tau_2)} = c_{\tau_1}c_{\tau_2}
\;,
\qquad
c_{I(\tau_1)I(\tau_2)I(\tau_3)} = \frac{3!}{(4-\mft)!}c_{\tau_1}c_{\tau_2}c_{\tau_3}\;,
\end{equ}
where $\mft \in \{1,2,3\}$ is the cardinality of the set
$\{\tau_1,\tau_2,\tau_3\}$ (e.g. if $\tau_1=\tau_2 \neq\tau_3$, then $\mft=2$).
Note that, if we write $\{\sigma_1,\ldots,\sigma_\mft\} = \{\tau_1,\tau_2,\tau_3\}$ for distinct $\sigma_i$, then $\frac{3!}{(4-\mft)!}$ is the multinomial coefficient $\binom{3}{k_1,\ldots,k_\mft}$ where $k_i$ is the number of times $\sigma_i$ appears in the tuple $(\tau_1,\tau_2,\tau_3)$.
Note also that $\CS^{N}X\in \CC^\infty([0,\infty)\times \T^n,E)$.

\begin{theorem}[Well-posedness of \eqref{eq:A_eq}-\eqref{eq:A_ic}] \label{thm:lwp}
	Suppose that  $N\in\N$, $\omega_{\<Xi>}\le 0$, and $ \bbeta, \bdelta \in\R^{\CT^N}\!$ satisfy condition~\eqref{eq:CI}.
	Let $\theta> 0$ be such that
	\begin{equ}[eq:theta_assump]
		\omega/2-\theta>-1/2\;,
		\quad
		\alpha/2-\theta>-1/2\;,
		\quad
		\gamma/2-\theta>-1
	\end{equ}
	(such $\theta>0$ exists due to condition \eqref{eq:CI}).
	For $T>0$, let $ \CB_T$ denote the Banach space of functions $R\in\CC([0,T],\CC(\T^n,E))$ for which
	\begin{align*}
		|R|_{\CB_T} \eqdef \sup_{t\in(0,T)}  t^{-\theta}|R_t|_\infty + t^{\frac12-\theta}|R_t|_{\CC^1} < \infty\;.
	\end{align*}
	Then there exist $\kappa,\e>0$ with the following property.
	For all $K>1$
	and $X\in\init$ such that $\Theta(X)\leq K$,
	if $T^\kappa <\e K^{-2}$,
	then there exists a unique function $R(X)\in \CB_T$ such that
	\begin{equ}
		A\eqdef R(X)+\CP\star \CS^N_{\<Xi>}X \colon (0,T]\to \CC^1(\T^n,E)
	\end{equ}
	solves  \eqref{eq:A_eq}
	with initial condition $X$
	in the sense that
	$\lim_{t\downarrow0}|A_t-X|_{\CC^\omega}=0$ where we recall $\omega \eqdef  \min_{\tau\in\CT^N_{\<Xi>}}\omega_\tau$
	and $\CS^N_{\<Xi>}X$ from \eqref{eq:SN_def}.

	Furthermore, $ |R(X)|_{\CB_T} \leq K$
	and the map $\{X\in\init \,:\, \Theta(X) \leq K \} \ni X\mapsto R(X)\in \CB_T$ is $1$-Lipschitz.
\end{theorem}

\begin{proof}
For $T\in (0,1)$ and $X\in\init$,
	consider the map $\CM^X\colon \CB_T \to \CB_T$
	\begin{equ}[eq:contraction_mapping_M]
		\CM^X_t(R)
		\eqdef	\CP_t \star F(A,DA) - \CP_t \star \CS^N X =
		\int_0^t \CP_{t-s} F(A_s, D A_s) \,\mrd s
		- \CP_t\star\CS^{N} X\;,
		\end{equ}
	where we denote $A_t = R_t + \CP_t \star \CS^N_{\<Xi>} X$.
	Note that if $R$ is a fixed point of $\CM^X$, then $A$ solves \eqref{eq:A_eq}-\eqref{eq:A_ic} due to the definition of $X^\tau$ and the constant $c_\tau$.
	
	We show that $ \CM_t^X$ is well-defined, maps $\CB_T$ into itself, and, for $T$ as in the statement, defines a contraction on the ball of radius $K$ of $\CB_T$.
	
	Note that $\CP\star F(A, D A) $ is a linear combination of terms of the following form: 
	\begin{enumerate}
		\item constant,
		\item linear in $(\CP \star \CS^N_{\<Xi>} X,D\CP \star \CS^N_{\<Xi>} X,R)$,
		\item bilinear in $(\CP \star \CS^N_{\<Xi>} X+R,\CP \star \CS^N_{\<Xi>} X+R)$,
		\item $\CP\star B(\CP \star \CS^N_{\<Xi>} X+R, D\CP \star \CS^N_{\<Xi>} X+DR)$,
		\item $\CP \star P(\CP \star \CS^N_{\<Xi>} X+R, \CP \star \CS^N_{\<Xi>} X+R,\CP \star \CS^N_{\<Xi>} X+R)$.
	\end{enumerate}
	We claim that, for all $\tau\in\CT^N_{\<Xi>}$,
	\begin{equs}
		|\CP_t \star X^\tau|_{\infty}
		&\lesssim
		t^{\frac{\omega_\tau}{2}} \Theta(X) \;,\label{eq:PtX_bound}
		\\
		|D\CP_t \star X^\tau|_{\infty} &
		\lesssim
		t^{\frac{\omega_\tau-1}{2}} \Theta(X) \;.\label{eq:DPtX_bound}
	\end{equs}
	Indeed, for $\tau=\<Xi>$, this
	follows directly from \eqref{eq:heat_flow_estimates}
	and the assumption $\omega_\tau\le 0$.
	On the other hand, for $\tau\in\CT^N$,
	we have
	\begin{equs}{}
		&|\CP_t \star X^\tau|_{\infty} \lesssim
		\Theta(X)\int_0^t (t-s)^{\frac{\beta_\tau}{2}}s^{-\delta_\tau} \mrd s
		\lesssim
		\Theta(X)t^{\frac{\beta_\tau}{2}-\delta_\tau+1}\;,
		\\
		&|D\CP_t \star X^\tau|_{\infty} \lesssim
		\Theta(X)\int_0^t (t-s)^{\frac{\beta_\tau}{2}-\frac12}s^{-\delta_\tau} \mrd s
		\lesssim
		\Theta(X) t^{\frac{\beta_\tau}{2}-\delta_\tau+\frac12}\;,
	\end{equs}
	where we applied \eqref{eq:heat_flow_estimates} and the assumptions $\beta_\tau\in (-1,0)$ and $\delta_\tau<1$,
	which proves the claim \eqref{eq:PtX_bound}-\eqref{eq:DPtX_bound}.	
	The bounds \eqref{eq:PtX_bound}-\eqref{eq:DPtX_bound} in particular imply
	\begin{equ}[eq:PX_DPX_bounds]
		|\CP_t \star \CS^N_{\<Xi>}X|_\infty \lesssim t^{\frac{\omega}{2}} \Theta(X)\;,\qquad
		|D\CP_t\star \CS^N_{\<Xi>}X|_\infty
		\lesssim t^{\frac{\omega-1}{2}} \Theta(X)\;.
	\end{equ}
	
	We now analyse the contribution of the terms in the 5 cases above to $\CM_t^X(R) = \CP_t \star F(A,DA)-\CP_t \star \CS^N X$ in \eqref{eq:contraction_mapping_M}.

	\textbf{Case 1} produces a term of order $C^{(1)}_t \eqdef t$ in $L^\infty$ and in $\CC^1$.
	
	\textbf{Case 2} contributes a term of order $C^{(2)}_t\eqdef t^{(\omega+1)/2}\Theta(X)+t^{1+\theta}|R|_{\CB_T}$
	in $L^\infty$, where we used \eqref{eq:PX_DPX_bounds} and the assumption
	$(\omega-1)/2>-1$.
	This case further contributes a term of order
	$t^{-1/2}C^{(2)}_t$
	in $\CC^1$, where we additionally used \eqref{eq:heat_flow_estimates}.

	\textbf{Case 3} produces a term of order $C^{(3)}_t\eqdef t^{\omega+1}\Theta(X)^2 + t^{2\theta+1}|R|_{\CB_T}^2$ in $L^\infty$ and $t^{-1/2}C^{(3)}_t$
	in $\CC^1$
	where we used again \eqref{eq:PX_DPX_bounds},
	the assumption $\omega>-1$,
	and \eqref{eq:heat_flow_estimates}.
	
	\textbf{Case 4.} Terms that are linear and quadratic in $R$ are of order
	\begin{equ}[eq:tilde_C4]
		\tilde C^{(4)}_t \eqdef t^{\frac{\omega}{2}+\theta+\frac12}\Theta(X)|R|_{\CB_T}+t^{2\theta+\frac12}|R|_{\CB_T}^2
	\end{equ}
	in $L^\infty$.
	The first term on the right-hand side of \eqref{eq:tilde_C4}
	is due to the assumptions $\omega>-1$ and $\theta> 0$ and to \eqref{eq:PX_DPX_bounds}, the latter of which implies
	\begin{equ}
		|\CP_s \star \CS^N_{\<Xi>}X|_\infty |DR_s|_\infty + |D\CP_s \star \CS^N_{\<Xi>}X|_\infty |R_s|_\infty\lesssim s^{\frac{\omega}{2}+\theta-\frac12}\;.
	\end{equ}
	The second term on the right-hand side of \eqref{eq:tilde_C4} is due to
	$|R_s|_\infty |DR_s|_\infty \leq s^{2\theta-1/2}|R|_{\CB_T}^2$.
	These terms are also of order $t^{-1/2} \tilde C^{(4)}_t$ in $\CC^1$ due to \eqref{eq:heat_flow_estimates}.
	
	Note that the terms bilinear in $\CS^N_{\<Xi>} X$ in Case 4 
	can be written as a linear combination of 
	terms of the form $\CP_t\star B(\CP \star X^{\tau_1}, D\CP \star X^{\tau_2})$ with $\tau_i \in \CT^N_{\<Xi>}$.
	We now consider two subcases:
	\begin{enumerate}[label=(4\alph*)]
		\item $\noise{\tau_1}+\noise{\tau_2}\leq N$, and
		\item $\noise{\tau_1}+\noise{\tau_2}> N$.
	\end{enumerate}

	In Case (4a),
	the term is well-defined 
	(and is of order $t^{\frac{\omega_\tau}{2}}\Theta(X)$ in $L^\infty$  due to \eqref{eq:PtX_bound} where $\tau = I(\tau_1)I'(\tau_2)\in\CT^N$),
	but, recalling the definition of $\CS^N X$ in \eqref{eq:SN_def},
	this term is cancelled by the term $-\CP_t\star c_\tau X^\tau$ in $-\CP_t\star \CS^N X$ in \eqref{eq:contraction_mapping_M} where $\tau = I(\tau_1)I'(\tau_2)$.
	
	In Case (4b), this term contributes 
	$t^{(\alpha+1)/2}\Theta(X)^2$ in $L^\infty$ and $t^{\alpha/2}\Theta(X)^2$ in $\CC^1$,
	where we used the assumption $\omega_{\tau_1}+\omega_{\tau_2}\geq \alpha>-1$
	and \eqref{eq:heat_flow_estimates}.
	In conclusion, Case 4 contributes
	\begin{equ}
		C^{(4)}_t \eqdef
		\tilde C^{(4)}_t + t^{\frac{\alpha+1}{2}}\Theta(X)^2
	\end{equ}
	in $L^\infty$ and $t^{-1/2}C^{(4)}_t$  in $\CC^1$.
	
	\textbf{Case 5.}
	Terms that are linear, quadratic, and trilinear in $R$ contribute
	\begin{equ}
		\tilde C^{(5)}_t \eqdef t^{\omega+\theta+1}\Theta(X)^2|R|_{\CB_T}
		+ t^{\frac{\omega}{2}+2\theta+1}\Theta(X)|R|_{\CB_T}^2
		+t^{3\theta+1}|R|_{\CB_T}^3
	\end{equ}
	in $L^\infty$
	and $t^{-1/2}\tilde C^{(5)}_t$ in $\CC^1$
	where we used $\omega>-1$ and $\theta> 0$ and \eqref{eq:heat_flow_estimates}.
	
	Terms that are trilinear in $\CS^N_{\<Xi>} X$ 
	can be written as a linear combination of 
	terms of the form $\CP\star P(\CP \star X^{\tau_1}, \CP \star X^{\tau_2},\CP \star X^{\tau_3})$.
	We again consider two subcases:
	\begin{enumerate}[label=(5\alph*)]
		\item $\sum_{i=1}^3\noise{\tau_i}\leq N$, and
		\item $\sum_{i=1}^3\noise{\tau_i}> N$.
	\end{enumerate}

	In Case (5a), this term is well-defined
	(and is of order $t^{\omega_\tau/2}\Theta(X)$ in $L^\infty$ due to \eqref{eq:PtX_bound} where 
	$\tau = \prod_{i=1}^3I(\tau_i)\in\CT^N$),
	but recalling again the definition of $\CS^N X$ in \eqref{eq:SN_def},
	this term is cancelled by the term $-\CP_t\star c_\tau X^\tau$ in $-\CP_t\star \CS^N X$ in \eqref{eq:contraction_mapping_M} where $\tau = I(\tau_1)I(\tau_2)I(\tau_3)$.
	
	In Case (5b), the contribution is of order $t^{\gamma/2+1}\Theta(X)^3$ in $L^\infty$ and $t^{(\gamma+1)/2}\Theta(X)^3$ in $\CC^1$,
	where we used the assumption $\sum_{i=1}^3\omega_{\tau_i} \geq \gamma  > -2$
	and \eqref{eq:heat_flow_estimates}.
	In conclusion, Case 5 contributes
	\begin{equ}
		C^{(5)}_t \eqdef \tilde C^{(5)}_t +
		t^{\frac{\gamma}{2}+1}\Theta(X)^3
	\end{equ}
	in $L^\infty$ and $t^{-1/2}C^{(5)}_t$ in $\CC^1$.
	
	Note that $-\CP \star \CS^N X$ in \eqref{eq:contraction_mapping_M} is completely cancelled by cases (4a) and (5a) above.
	
	Combining the above cases, we have the estimate
	\begin{equ}[eq:M_inf]
		t^{-\theta}|\CM^X_t(R)|_\infty
		\lesssim t^{-\theta} \sum_{i=1}^5 C^{(i)}_t 
		\lesssim t^\kappa (1+\Theta(X)^3 + |R|_{\CB_T}^3)\;,
	\end{equ}
	where
	\begin{equ}
		\kappa \eqdef \min\{(\omega+1)/2,2\theta+1/2,(\alpha+1)/2,\gamma/2+1\}-\theta
	\end{equ}
	(we used here the estimates on $C^{(i)}_t$ above and the assumptions $\theta> 0$ and $\omega>-1$).
	The same estimate holds for $t^{-\theta+1/2}|\CM_t^X(R)|_{\CC^1}$, so in conclusion
	\begin{equ}
		|\CM^X(R)|_{\CB_T} \lesssim T^\kappa (1+\Theta(X)^3 + |R|_{\CB_T}^3)\;.
	\end{equ}
	We remark that $\kappa>0$
	due to \eqref{eq:theta_assump}.
	It follows that there exists $\e>0$ such that, for all $K>1$ and $X\in\init$ with $\Theta(X)\leq K$,
	if $T^{\kappa}\leq \e K^{-2}$,
	then $\CM^X$ stabilises the ball $\{R\in\CB_T\,:\,|R|_{\CB_T}\leq K\}$.

	Furthermore, for another $\bar X\in\init$ with $\Theta(\bar X)\leq \Theta(X)$ and $\bar R\in\CB_T$,
	almost the same considerations imply that
	\begin{equs}[eq:M_X_diffs]
		|\CM^X(R) - \CM^{\bar X}(\bar R)|_{\CB_T}
		&\lesssim T^\kappa \Theta(X,\bar X)(1+\Theta(X)^2)
		\\
		&\quad +
		T^\kappa |R-\bar R|_{\CB_T}(1+\Theta(X)^2+|R|_{\CB_T}^2 + |\bar R|_{\CB_T}^2)\;.
	\end{equs}
	Taking $\bar X=X$, we obtain that $\CM^X$ is a contraction on the ball $\{R\in\CB_T\,:\,|R|_{\CB_T}\leq K\}$
	whenever $T^\kappa < \e K^{-2}$ for $K>1$ and $\Theta(X)\leq K$ and $\e>0$ sufficiently small.
	It follows from Banach's fixed point theorem that there exists a unique fixed point $R$ to $\CM^X$ in this ball.
	Uniqueness in all of $\CB_T$ follows in a standard way by restarting the equation.
	The claimed Lipschitz estimate $|R-\bar R|_{\CB_T}\leq \Theta(X,\bar X)$ for $\eps$ sufficiently small follows from
	taking $R$ and $\bar R$ in \eqref{eq:M_X_diffs} as the unique fixed points.
		
	It remains to show $\lim_{t\downarrow0}|A_t-X|_{\CC^\omega}=0$.
	We first remark that, since $A_t = R_t + \CP_t \star \CS^N_{\<Xi>} X$ and we know that $R$ is the fixed point of $\CM^X$, we have
	\begin{equ}[eq:A_t-X]
		A_t-X= \CP_t X-X + \CP_t\star \CS^N X + \CM_t^X(R)\;.
	\end{equ}
	We claim that for all $\tau\in\CT^N$,
	\begin{equ}[eq:PX_lim]
		\lim_{t\downarrow 0}|\CP_t \star X^\tau|_{\CC^{\omega_\tau}} = 0\;,
	\end{equ}
	which, since $\omega\le\omega_{\tau}$, together with the triangle inequality yields
	\begin{equ}[eq:PS_lim]
		\lim_{t\downarrow 0}|\CP_t\star \CS^NX|_{\CC^\omega} = 0\;.
	\end{equ} 
	Suppose that $Y,\bar Y\in (\init,\Theta)$ and $\tau\in\CT^N$. Recalling the definition of $\omega_{\tau} = \beta_\tau-2\delta_\tau+2$ and $\Theta$ (Definition \ref{def:norms+init+}) and using the assumption $\delta_\tau\in(0,1)$, we have, uniformly in $t\in (0,1)$,
	\begin{equ}
	  	|\CP_t \star Y^\tau-\CP_t \star \bar Y^\tau|_{\CC^{\omega_{\tau}}} \lesssim
	 	 \Theta(Y,\bar Y)\int_0^t (t-s)^{\frac{\beta_\tau-\omega_{\tau}}{2}}s^{-\delta_\tau} \mrd s\lesssim\Theta(Y,\bar Y)\;.
	\end{equ}
	If $\lim_{t\downarrow 0 }|\CP_t \star Y^\tau |_{\CC^{\omega_\tau}} = 0$,
	it follows that $\lim_{t\downarrow 0 }|\CP_t \star \bar Y^\tau |_{\CC^{\omega_\tau}} \lesssim \Theta(Y,\bar Y)$,
	and therefore the space of all $Z\in\init$ such that
	$\lim_{t\downarrow 0 }|\CP_t \star Z^\tau|_{\CC^{\omega_\tau}} = 0$
	is closed in $\CI$.
	Since the space of smooth functions in dense in $\init$ and since all smooth functions $Z$ satisfy $\lim_{t\downarrow 0 }|\CP_t \star Z^\tau|_{\CC^{\omega_\tau}} = 0$, we obtain \eqref{eq:PX_lim} for all $X\in\init$. This proves the claim and thus also \eqref{eq:PS_lim}.

	The conclusion now follows from \eqref{eq:M_inf}, \eqref{eq:A_t-X}, \eqref{eq:PS_lim}, and the fact that
	$
		\lim_{t\downarrow 0}|\CP_t X - X|_{\CC^{\omega}} = 0
	$
	which holds since $X \in \init$ belongs to the closure of smooth functions in $\CC^{\omega_{\<Xi>}}$. 
\end{proof}
	
\section{Probabilistic estimates}
\label{sec:prob_estimates}

In this section, we show that, for any random field $X$ satisfying \assu{ass:GF}, there exist $N\in\N$, $\omega_{\<Xi>}\in\R$, and $\bbeta,\bdelta\in\R^{\CT^N}$ such that mollifications $X^\e$ converge in the state space $\init_{\omega_{\<Xi>},\bbeta,\bdelta}$ almost surely
and in $L^p(\P)$ for all $p\ge 1$.
We make this statement precise in \theo{thm:moment-estimates} and its Corollary~\ref{corr:convergence-of-mollifications}, which are the main results of this section.
We use Corollary~\ref{corr:convergence-of-mollifications} in Section \ref{sec:main_proof} in the proof of the main Theorem \ref{thm:main}.
	
In the sequel, we fix an integer $n\ge 2$ as the dimension of the underlying space $\T^n$.
Recall that we have also fixed $d\in(2,4)$.
For the space of trees $\CL$ from \defi{def:trees},
we define $|\!\cdot\!| \colon \CL \to\R$
recursively by $|\<Xi>|=(2-d)/2-2$ and
\begin{equ}[eq:homogeneity]
		|\tau\sigma| = |\tau|+|\sigma|\quad  \text{ where }\quad |I(\tau)|=|\tau|+2\;,\quad |I'(\tau)|=|\tau|+1\;.
	\end{equ}
Note that applying the identity $|\tau\sigma| = |\tau|+|\sigma|$ twice implies $|\tau_1\tau_2\tau_3| = \sum_{i=1}^3|\tau_i|$.
\begin{remark}\label{rem:homogeneity}
	An easy induction implies that $|\tau| = \noise{\tau}(2-d)/2+\noise{\tau}-3$ for all $\tau\in\CT$ (see \defi{def:CT_trees}), where we recall that $\noise{\tau}$ is the number of leaves in $\tau$.
\end{remark}

\begin{theorem}\label{thm:moment-estimates}
	Suppose that $X$ satisfies \assu{ass:GF}.
	Let $\chi^\e$ be as in \theo{thm:main} and denote $X^{\tau,\e}\eqdef (\chi^\e * X)^\tau$ for any $\tau\in\CT$. 	Let $\tau\in\CT\setminus\{\<Xi>\}$,  $\beta\in[\dd,0]$, $\kappa\in[0,4-d)$, and $\delta= -|\tau|/2+\beta/2+\kappa/4$. 
	Then for all $p\geq 1$,  all test functions $\phi\in\CC^\infty(\T^n)$ supported in $\{z\in\T^n\,:\,|z|\leq \frac14\}$, and uniformly in $0<\bar \e\le\e<1$, $0< s\le t<1$, $\lambda\in(0,1]$, and $z\in\T^n$,
	\begin{equ}[eq:Z_eps_diff]
		\| \scal{t^\delta X^{\tau,\e}_t-t^\delta X^{\tau,\bar\e}_t , \phi^\lambda_z} \|_{L^p} \lesssim |\e-\bar\e|^{\frac{\kappa\wedge 1}{2}} \lambda^{\beta}
	\end{equ}
		and
	\begin{equ}[eq:Z_time_diff]
		\|  \scal{t^\delta X_t^{\tau,\e} -s^\delta X_s^{\tau,\e},\phi^\lambda_z} \|_{L^p} \lesssim |t-s|^{\frac{\kappa}{4}} \lambda^\beta\;,
	\end{equ}
	where $\|\! \cdot \!\|_{L^p} \eqdef (\E|\! \cdot \! |^p)^{1/p}$
	is the stochastic $L^p$-norm and $\phi^\lambda_z (\cdot) \eqdef \lambda^{-n}\phi((\cdot -z)/ \lambda)$.
\end{theorem}
The proof of \theo{thm:moment-estimates} is broken up into several steps and is given at the end of this section.
Before proceeding to the proof, we give a corollary of \theo{thm:moment-estimates}.
\begin{corollary}\label{corr:convergence-of-mollifications}
	Let $N\in\N$, $\omega_{\<Xi>}<\dd$,
	and, for $\tau\in\CT^N$,
	$\beta_\tau \in [(2-d)/2,0)$
	and $\delta_\tau>-|\tau|/2+\beta_\tau/2$. Then, for $X$ satisfying \assu{ass:GF},
	the mollifications $X^{\e}$ defined in \theo{thm:moment-estimates} converge as $\e\downarrow 0$ in $(\init_{\omega_{\<Xi>},\bbeta,\bdelta},\Theta_{\omega_{\<Xi>},\bbeta,\bdelta})$ in $L^p(\P)$ for all $p\in [1,\infty)$ and $\P$-a.s.,
	where $ \bbeta\eqdef(\beta_\tau)_{\tau\in\CT^N}$ and $ \bdelta\eqdef(\delta_\tau)_{\tau\in\CT^N}$.
\end{corollary}

\begin{proof}
	We first estimate the difference between $X^\e$ and $X^{\bar\e}$ in the metric space $(\init,\Theta)\equiv(\init_{\omega_{\<Xi>},\bbeta,\bdelta},\Theta_{\omega_{\<Xi>},\bbeta,\bdelta})$ in $L^p(\P)$.
	We can assume that $\delta_\tau \leq -|\tau|/2$ since the metric $\Theta_{\omega_{\<Xi>},\bbeta,\bdelta}$ becomes
	weaker as $\delta_\tau$ increases.
	
	Take $\kappa\in(0,1\wedge(4-d))$ such that $\omega_{\<Xi>}<\dd-\kappa/2$,
	and $\delta_\tau>-|\tau|/2+\beta_\tau/2+\kappa/4$ for all $\tau\in\CT^N$.
	Then \eqref{eq:Z_eps_diff} and \eqref{eq:Z_time_diff} (applied with $\beta\in(\beta_\tau,0)$ such that $\delta_\tau = -|\tau|/2+\beta/2+\kappa/4$)
	and a Kolmogorov-type argument (see e.g. \cite[Thm.~2.7]{ChandraWeber17})
	imply that, for all $p\ge 1$ and $\tau\in\CT^N$, denoting $Z^{\tau,\e}_t = t^{\delta_\tau} X_t^{\tau,\e}$,
	\begin{equ}
		\||Z^{\tau,\e}_t-Z^{\tau,\bar\e}_t|_{\CC^{\beta_\tau}} \|_{L^p} \lesssim |\e-\bar\e|^{\kappa/2}
	\end{equ}
	and 
	\begin{equ}
		\big\||Z_t^{\tau,\e} - Z_s^{\tau,\e}|_{\CC^{\beta_\tau}}\big\|_{L^p}\lesssim |t-s|^{\kappa/4},
	\end{equ}
	uniformly in $0<\bar \e<\e<1$ and $0\le s\le t<1$.
	Therefore
	\begin{equ}
		\| |Z_t^{\tau,\e}-Z_t^{\tau,\bar\e} - (Z_s^{\tau,\e}-Z_s^{\tau,\bar\e}) |_{\CC^{\beta_\tau}}\big| \|_{L^p} \lesssim |\e-\bar\e|^{\kappa/2}\wedge |t-s|^{\kappa/4} \leq |\e-\bar\e|^{\kappa/4}|t-s|^{\kappa/8}\;.
	\end{equ}
	Hence, by Kolmogorov's theorem (see e.g.~\cite[Thm.~A.10]{FV10}) applied to the $\CC^{\beta_\tau}$-valued stochastic process $t\mapsto Z_t^{\tau,\e}-Z_t^{\tau,\bar\e}$, we have
	\begin{equ}
		\Big\| \sup_{s\neq t\in (0,1)} |t-s|^{-\kappa/16}|Z_t^{\tau,\e}-Z_t^{\tau,\bar\e} - (Z_s^{\tau,\e}-Z_s^{\tau,\bar\e})|_{\CC^{\beta_\tau}} \Big\|_{L^p} \lesssim |\e-\bar\e|^{\kappa/4}\;.
	\end{equ}
	In particular, sending $s$ to $0$,
	we obtain, uniformly in $t\in (0,1)$ and $0<\bar \e<\e<1$,
	\begin{equ}
	\Big\|\sup_{t\in (0,1)}|Z_t^{\tau,\e}-Z_t^{\tau,\bar\e}|_{\CC^{\beta_\tau}}\Big\|_{L^p} \lesssim |\e-\bar\e|^{\kappa/4}\;,
	\end{equ}
	or equivalently,
	\begin{equ}[eq:Lp-eps-diff]
		\big\|\fancynorm{X^{\e};X^{\bar\e}}_{\tau;(\beta_\tau,\delta_\tau)}\|_{L^p}\lesssim |\e-\bar\e|^{\kappa/4} \;.
	\end{equ}
	On the other hand, by \lem{lem:covar_X^e-X} below and using again a Kolmogorov-type argument, we have, for any $\omega_{\<Xi>}<\dd-\kappa/2$,
	\begin{equ}[eq:Lp-eps-diff.]
		\||X^{\e}-X^{\bar\e}|_{\CC^{\omega_{\<Xi>}}} \|_{L^p} \lesssim |\e-\bar\e|^{\kappa/2}\;.
	\end{equ}
	Consequently, recalling the definition of $\Theta$ in \eqref{eq:Theta}, estimates \eqref{eq:Lp-eps-diff}-\eqref{eq:Lp-eps-diff.} imply that 
	$
		\|\Theta(X^{\e},X^{\bar\e})\|_{L^p} \lesssim |\e-\bar\e|^{\kappa/4}
	$,
	which shows the desired $L^p(\P)$-convergence.
	To conclude a.s.-convergence, Kolmogorov's theorem applied to the parameter $\eps\in(0,1)$ implies that
	$\|\sup_{0<\eps<\bar\eps<1}|\eps-\bar\eps|^{-\kappa/8}\Theta(X^{\e},X^{\bar\e})\|_{L^p} < \infty$ for all $p\in [1,\infty)$.
\end{proof}
We used above the next lemma. Its proof is standard but is included for completeness.
\begin{lemma}\label{lem:covar_X^e-X}
	Suppose that $X$ satisfies \assu{ass:GF}. Let $X^\e\eqdef\chi^\e * X$, where $\chi^\e$ is as in \theo{thm:main}. Then for all $p \in [1,\infty)$
	\begin{equ}
		\| \scal{X^\e-X^{\bar\e}, \phi^\lambda_z} \|_{L^p}  \lesssim |\e-\bar\e|^{\kappa/2} \lambda^{1-d/2-\kappa/2}
	\end{equ}
	uniformly in $0\le \bar\e \le \e \leq 1$, $\kappa\in [0,1]$, $\lambda\in(0,1]$ and $z\in\T^n$, and where $\phi$ is as in \theo{thm:moment-estimates}.
\end{lemma}

\begin{proof}
	It suffices to consider the case that $\eps \leq 2\bar\e$.
	Indeed, it follows from this case that, for $\eps > 2\bar\e$, by telescoping,
	$
	\lambda^{-1+d/2+\kappa/2}\| \scal{X^\e-X^{\bar\e}, \phi^\lambda_z} \|_{L^p}  \lesssim \sum_{i=1}^\infty (\e 2^{-i})^{\kappa/2}
	\lesssim |\bar\e-\e|^{\kappa/2}
	$.

	By equivalence of Gaussian moments and 
	integrating against the test function, the desired bound follows once we show
	\begin{equ}[eq:cov_Psi-Psi^e]
		|	\E[X^\e(x)\otimes (X^\e-X^{\bar\e})(\bar x)] |  \lesssim |\e-\bar\e|^\kappa|x-\bar x|^{2-d-\kappa}
	\end{equ}
	uniformly in $\e\in[0,1]$, $\bar\e\in [\frac12\e,2\e]$, $\kappa\in [0,1]$, and $x,\bar x\in\T^n$. 
	
	By translation invariance, we suppose without loss of generality that $\bar x=0$.
	The left-hand side of \eqref{eq:cov_Psi-Psi^e}
	is then equal to
	\begin{equ}[eq:LHS_lemma]
		\scal{C *\chi^\e(\cdot-x), \chi^\e-\chi^{\bar\e}} = \int C *\chi^\e(z-x) \{\chi^\e(z) - \chi^{\bar\e}(z)\} \mrd z\;.
	\end{equ}
	By \assu{ass:GF} with $|k|=0$, if $|x|> 4\e$, then, for all $|z|<2\e$,
	$
		|C *\chi^\e (z-x)| \lesssim |x|^{2-d}
	$
	and therefore
	\begin{equ}
		|\scal{C *\chi^\e(\cdot-x), \chi^\e-\chi^{\bar\e}}| \lesssim |x|^{2-d}\;.
	\end{equ}
	On the other hand, if $|x|\leq 4\e$, then, for all $|z|<2\e$,
	$
		|C *\chi^\e (z-x)| \lesssim \e^{2-d}
	$
	and therefore
	\begin{equ}
		|\scal{C *\chi^\e(\cdot-x), \chi^\e-\chi^{\bar\e}}| \lesssim \e^{2-d}\;.
	\end{equ}
	This proves~\eqref{eq:cov_Psi-Psi^e} for $\kappa=0$.
	(We did not use yet that $\bar\e\in [\frac12\e,2\e]$.)
	
	By interpolation, it remains only to consider the case $\kappa=1$.
	We can suppose without loss of generality that $\chi$ has support in $B(\frac14)$ where $B(r) = \{z\,:\,|z|\leq r\}$.
	For any differentiable function $f\colon \T^n\to\R$, since $\bar\e\in [\frac12\e,2\e]$,
	\begin{equs}[eq:chi_diff]
		\Big|\int f(z) \{\chi^\e(z) - \chi^{\bar\e}(z)\} \mrd z \Big|
		&=
		\Big|\int f(z) \{\e^{-n}\chi(z \e^{-1}) - \bar\e^{-n} \chi(z \bar\e^{-1})\} \mrd z\Big|
		\\
		&= \Big|\int \{f(z) - f(z\bar\e \e^{-1})\}\e^{-n}\chi(z \e^{-1}) \mrd z
		\Big|
		\\
		&\leq \int |\nabla f|_{\infty;B(\e)} |z||1-\bar\e \e^{-1}| \e^{-n}|\chi(z \e^{-1})| \mrd z
		\\
		&\lesssim |\e - \bar \e | |\nabla f|_{\infty;B(\e)}\;,
	\end{equs}
	where we denote $|f|_{\infty;B} = \sup_{x\in B}|f(x)|$ for a subset $B\subset\T^n$.
	
	With the aim to bound~\eqref{eq:LHS_lemma}, we remark that, by \assu{ass:GF}
	with $|k|=1$, if $|x|>4\e$, then, for all $|z|<2\e$,
	$
		|\nabla C *\chi^\e (z-x)| \lesssim |x|^{1-d}
	$
	and therefore, by \eqref{eq:chi_diff},
	\begin{equ}
		|\scal{C *\chi^\e(\cdot-x), \chi^\e-\chi^{\bar\e}}| \lesssim |\e-\bar\e||x|^{1-d}\;.
	\end{equ}
	On the other hand, if $|x|\leq 4\e$, then, for all $|z|<2\e$,
	$
		|\nabla C *\chi^\e (z-x)| \lesssim \e^{1-d}
	$
	and therefore, again by \eqref{eq:chi_diff},
	\begin{equ}
		|\scal{C *\chi^\e(\cdot-x), \chi^\e-\chi^{\bar\e}}| \lesssim \e^{1-d}|\e-\bar\e|\;.
	\end{equ}
	This proves~\eqref{eq:cov_Psi-Psi^e} for $\kappa=1$.
\end{proof}

\subsection{Estimates for iterated convolution integrals}
\label{subsec:general-bounds}

In this subsection, we derive bounds on general iterated convolution integrals.
The main result of this subsection is Theorem \ref{thm:general-bounds}
that controls the covariance of terms $X^{\tau,\e}_t$ appearing in \theo{thm:moment-estimates}.

\begin{definition}
	A \emph{forest} $\zF = \{\tau_1,\ldots,\tau_m\}$ is a finite set of rooted trees (see Definition~\ref{def:trees}),
	which we identify with a (possibly disconnected) graph.

	We write $V_{\zF}$ and $E_{\zF}$ for the sets of vertices and edges of $\zF$ respectively. We let $\roots_{\zF} = \{\rho_{\tau_1},\ldots,\rho_{\tau_m}\}$ denote the set of roots of $\zF$.

	We sometimes use the notation $\zF=(V,E,\roots)$, by which we indicate that $V,E,\roots$ are the respective set of vertices, edges, and roots of $\zF$.
 
	If $m=1$, then we simply identify $\zF=\tau_1$ and $\roots_{\zF} = \rho_{\tau_1}$. 
\end{definition}
Throughout this subsection, unless otherwise stated, $\zF = \{\tau_1,\ldots,\tau_m\}$ denotes a forest. We write $V,E,\roots$ for $V_{\zF},E_{\zF},\roots_{\zF}$.

\begin{notation}\label{not:LI}
	We write $L$ for the set of leaves of $\zF$ (which is the union of the sets of leaves of $\tau_1,\ldots,\tau_m$).
	We also denote $I\eqdef V\setminus L$, elements of which will be referred to as inner vertices/nodes.
	We then define $V^* \eqdef V\setminus \roots$, $L^*\eqdef L\setminus \roots$ and $I^*\eqdef I\setminus \roots$.
	We write $V_{\zF}^*$, $L_{\zF}$,
	etc., when we wish to specify the forest $\zF$.
\end{notation}

\begin{definition}
	For $v,u\in V$, we write $v\preceq u$ if $u$ and $v$ are both vertices of the same tree $\tau_i$ in $\zF$ and $v$ is on the unique path from $u$ to $\rho_{\tau_i}$.
	We write $v\prec u$ if $v\preceq u$ and $v\neq u$.
	We say that $u$ is a child of $v$ and that $v$ is the parent of $u$ if $\{u,v\}\in E$ and $v\prec u$, in which case we write $\hat u = v$ and write $(\hat u,u)$ for the edge $\{u,v\}$.

	A \emph{branch} of $\zF$ is a subtree $\tilde{\tau}$ with root $\tilde{\rho} \in V$, vertex set $\tilde{V}$ formed of all $v \in V$ with $\tilde{\rho}\preceq v$, and edge set $\tilde{E}$ formed of all edges $\{u,v\}\in E$ such that $u,v\in \tilde{V}$.

	We call a branch $\tilde{\tau}$ of $\zF$ a \emph{strict branch} if $\tilde{\rho}\notin \roots$, where $\tilde{\rho}$ is the root of $\tilde\tau$.
\end{definition}

We remark that every vertex $v\in V^*$ has precisely one parent (but every inner node may have many children). 
We also remark that a branch of $\zF$ is uniquely determined by its root, and that any vertex $v$ of $\zF$ is the root of a unique branch. 
Consequently, the set of branches of $\zF$ is in bijection with $V$. 
We write $\zF_v$ for the unique branch of $\zF$ whose root is $v\in  V$.

\begin{definition} \label{def:contraction}  
	Let $\fC$ be a subset of $\left\{\{u,v\}:\; u,v\in L\,,\;u\neq v\right\}$
	such that, for every leaf $w\in L$, there is at most one element of $\fC$ that contains $w$. 
	We call $\fC$ a \emph{contraction of $\zF$} (or simply a contraction when clear from the context). Elements of a contraction are called \emph{contracting edges}.
	If $\{u,v\}\in \fC$, we say that $u$ and $v$ are \emph{paired} in $\fC$. 
\end{definition}
\begin{definition}\label{def:CT} 
	A \emph{contracted forest} is a tuple $(\zF, \fC)$ where $\zF$ is a forest and $\fC$ is a contraction of $\zF$. 
	\begin{itemize}
	\item
	If $\zF$ is a singleton (i.e. a tree), we call $(\zF, \fC)$ a contracted tree.
	\item 
	Given a contracted forest $(\zF, \fC)$ and a branch $\sigma$ of $\zF$, the \emph{induced contraction} on $\sigma$ is the set $\fC_\sigma$ of all pairs $\{u,v\}\in \fC$ such that $u,v$ are both leaves of $\sigma$ (see Figure~\ref{fig:graphs}). 
	\end{itemize}
\end{definition} 

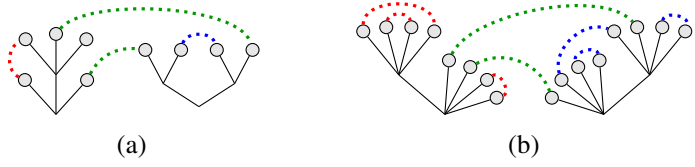
\begin{figure}[H]
\begin{center}
	\begin{subfigure}[t]{0.4\textwidth}
	\centering
		\begin{tikzpicture}[scale=0.47]		
%
			\node at (-.85, 1.2) [var] (a) {}; 
			\node at (0, 1.35)  [var] (b)  {}; 
			\node at (.85, 1.2)  [var] (c)  {};  
			\node at (-.87,.05) [var] (d) {}; 
			\node at (.87,.05) [var] (e) {}; 
			\draw (0,.2) -- (a);
			\draw (0,.2) -- (b);
			\draw (0,.2) -- (c);
			\draw (0,.2) -- (0,-.92);
			\draw (d) -- (0,-.92);
			\draw (e) -- (0,-.92);
					
			\node at (2.5,.9) [var] (a') {}; 
			\node at (3.5,.9)  [var] (b')  {};  
			\draw (3,-.05) -- (b');
			\draw (3,-.05) -- (a');
			\draw (3,-.05) -- (4.,-.7);

			\node at (4.5,.9) [var] (c') {}; 
			\node at (5.5,.9)  [var] (d')  {};
			\draw (5,-.05) -- (c');
			\draw (5,-.05) -- (d');
			\draw (5,-.05) -- (4.,-.7);

%
%
			\draw[bend left =55, darkgreen, very thick, dotted]  (b) to[out=55, in=110, distance=1.15cm] (d');
			\draw[bend left =35, darkgreen, very thick, dotted]  (e) to (a');
			\draw[bend left =65, red, very thick, dotted]  (d) to (a);
			\draw[bend left =75, blue, very thick, dotted]  (b') to (c');

		\end{tikzpicture}
		\caption{}
	\end{subfigure}
	\begin{subfigure}[t]{0.4\textwidth}
	\centering
		\begin{tikzpicture}[scale=0.55]		
			\node at (-.85, 1) [var] (a) {}; 
			\node at (-.3, 1.1)  [var] (b)  {}; 
			\node at (.3, 1.1)  [var] (c)  {};  
			\node at (.85, 1) [var] (d) {}; 
			\node at (1.2,.3) [var] (e) {}; 
			\node at (1.71, .14) [var] (f) {}; 
			\node at (2.11,-.17) [var] (g) {};
			\node at (2.34,-.6) [var] (h) {}; 
			\draw (0,-.05) -- (a);
			\draw (0,-.05) -- (b);
			\draw (0,-.05) -- (c);
			\draw (0,-.05) -- (d);
			\draw (0,-.05) -- (1.1,-1);
			\draw (e) -- (1.1,-1);
			\draw (f) -- (1.1,-1);
			\draw (g) -- (1.1,-1);
			\draw (h) -- (1.1,-1);
			\node at (6+.85, 1) [var] (a') {}; 
			\node at (6+.3, 1.1)  [var] (b')  {}; 
			\node at (6-.3, 1.1)  [var] (c')  {};  
			\node at (6-.85, 1) [var] (d') {}; 
			\node at (6-1.2,.3) [var] (e') {}; 
			\node at (6-1.71, .14) [var] (f') {}; 
			\node at (6-2.11,-.17) [var] (g') {};
			\node at (6-2.34,-.6) [var] (h') {}; 
			\draw (6,-.05) -- (a');
			\draw (6,-.05) -- (b');
			\draw (6,-.05) -- (c');
			\draw (6,-.05) -- (d');
			\draw (6,-.05) -- (6-1.1,-1);
			\draw (e') -- (6-1.1,-1);
			\draw (f') -- (6-1.1,-1);
			\draw (g') -- (6-1.1,-1);
			\draw (h') -- (6-1.1,-1);

			\draw[bend left =90, red, very thick, dotted]  (a) to (d);
			\draw[bend left =80, red, very thick, dotted]  (b) to (c);
			\draw[bend left =65, blue, very thick, dotted]  (g') to (d');
			\draw[bend left =-95, blue, very thick, dotted]  (e') to (f');
			\draw[bend left =45, darkgreen, very thick, dotted]  (e) to[out=55, in=120, distance=1.1cm] (c') ;
			\draw[bend left =95, blue, very thick, dotted]  (b') to (a');
			\draw[bend left =99, red, very thick, dotted]  (g) to (h);
			\draw[bend left =50, darkgreen, very thick, dotted]  (f) to (h');

		\end{tikzpicture}
		\caption{}
	\end{subfigure}
	\caption{\small Examples of \emph{contracted forests}.
	In~(a) and~(b), the edges drawn in red respectively blue are elements of the \emph{induced contraction} on the tree on the left respectively right.
	}
	\label{fig:graphs}
\end{center}
\end{figure}

\begin{definition}\label{def:DCF-I-K}

\begin{enumerate}[label=(\alph*)]
	\item\label{pt:DCF} We call $(\bzF,\bfC)$ \emph{a decorated contracted forest} if it consists of the following data:
	\begin{itemize}
	\item a contracted forest $(\zF, \fC)$, where $\zF=(V_{\zF}, E_{\zF},\roots_{\zF})$ is a forest consisting of two trees and with root set $\roots_{\zF} = \{\rho,\bar \rho\}$;

	\item assignments of real numbers to vertices $\{\beta_v\}_{v\in  V_{\zF}}$ such that $\beta_v\geq -1$ if $v\in L_{\zF}^{*}$;

	\item assignments to edges in $E_{\zF}$ of real numbers $\{\gamma_e\}_{e\in E_{\zF}}$ 
	as well as multi-indexes $\{k_e\}_{e\in E_{\zF}}$ with $k_e\in\N_0^n$;
	\item an assignment to contracting edges of real numbers $(a_e)_{e\in \fC}$.
	\end{itemize}
	By writing $\bzF$, we are indicating that the forest $\zF$ is decorated with the above assignments for its edges and vertices. Similarly, by $\bfC$, we mean that $\fC$ is considered together with the assignments to its elements. 
	
	A branch $\sigma$ of $\zF$ with the edge and vertex decorations inherited from $\bzF$ is called a \emph{decorated tree}.

	We call $v\in V_\zF$ a \emph{Dirac vertex} if $v\in L_{\zF}^{*}$ and $\beta_v=-1$, 
	and otherwise we call $v$ a \emph{non-Dirac vertex} (see \eqref{eq:gv_def} for motivation).

	\item For $(\bzF,\bfC)$ as above and $v\in V_{\zF}$, define the function/distribution $g_v$ on $\R$ by
	\begin{equ}[eq:gv_def]
	g_v(t)
	=
	\begin{cases}
	\delta_0(t) &\quad \text{ if $v$ is a Dirac vertex}\;,
	\\
	t^{\beta_v}\bone_{t>0} &\quad \text{ otherwise} \;.
	\end{cases}
	\end{equ}
	For $k\in\N_0^n$, let $G^{(k)}_{\cdot}(\cdot)$ denote the $k$-th spatial derivative of the heat kernel
	\begin{equ}
		G_t(x) = \sum_{m \in\Z^n} (4\pi t)^{-n/2}\mre^{-|x+m|^2/(4t)}\;,\qquad \forall t>0\;,\;\, x\in\T^n\;.
	\end{equ}
	Consider time-space points
	\begin{equ}
		(t,x), (\bar t, \bar x)\in D \eqdef (0,1)\times\T^n
	\end{equ}
	and denote $(t_\rho,x_{\rho})=(t,x)$ and  $(t_{\bar\rho},x_{\bar \rho})=(\bar t,\bar x)$. 
	We define
		\begin{align}
		\label{eq:general-conv}
		\!\!\!\!\!\I^{\bzF,\bfC}(t,\bar t,x-\bar x)&= \int_{D^{V^{*}_{\zF}}}
		\prod_{\{u,v\}\in \fC}\!\! |x_u-x_v|^{-a_{u,v}}
		\prod_{w\in V_{\zF}} g_w (t_w)
		\\&\quad\qquad
		\prod_{w \in  V^{*}_{\zF}}\!\left|(t_{\hat w}-t_{w})^{\gamma_{\hat 	w,w}}G^{(k_{\hat w,w})}_{t_{\hat w}-t_{w}}(x_{\hat w}-x_{w})\right| \, \mrd x_w\mrd t_w\;.\nonumber
	\end{align}
	\item\label{pt:kernel_diff} Suppose in addition that $t=\bar t$ and consider $s\in(0,1)$.
    Furthermore, suppose that we are given two sets $U,W \subset C_\rho\cup C_{\bar\rho}$,	where
	\begin{equ}
	C_u \eqdef \{v \in  V_{\zF} \,:\, v \text{ is a child of } u \}\;,
	\end{equ} such that $|U\cap C_\rho|\leq 1$, $|U\cap C_{\bar\rho}|\leq 1$, and $U\subset W$.
	We then define 
	\begin{equs}{}
	& \K^{\bzF,\bfC}_{W;U}(t,x-\bar x;s) = \\
	 	\!&\int_{D^{V_{\zF}^{*}}} 	\Big\{\!\!\prod_{w \in  V^{*}_{\zF}} \!\! \mrd 	x_w\mrd t_w\Big\}\;
	    \Big|\Big\{\!\!\prod_{w\in V_{\zF}} g_w(t_w)\Big\}
		\prod_{\{u,v\}\in \fC}\!\! |x_u-x_v|^{-a_{u,v}}
		\label{eq:general-conv_ts}
		\\&\quad
		\prod_{u\in V^{*}_{\zF}\setminus W}\!\!\!(t_{\hat u}-t_u)^{\gamma_{\hat 	u,u}}G^{(k_{\hat u,u})}_{t_{\hat u}-t_{u}}(x_{\hat u}-x_{u})
		\!\!\prod_{v\in W\setminus U} \!\! (s-t_v)^{\gamma_{\hat v,v}} G^{(k_{\hat v,v})}_{s-t_{v}}(x_{\hat v}-x_{v})
		\\&\quad\; \prod_{w \in U}\! \Big\{ (t-t_w)^{\gamma_{\hat w,w}} 	G^{(k_{\hat w,w})}_{t-t_{w}}(x_{\hat w}-x_{w})-(s-t_w)^{\gamma_{\hat w,w}} G^{(k_{\hat w,w})}_{s-t_{w}}(x_{\hat w}-x_{w})\Big\}\Big|\;.
	\end{equs}
Here and below, for any edge~$e=(\hat w,w)\in E_{\zF}$, we write both $\gamma_e,\gamma_{\hat w,w}$ to refer to the corresponding decoration assigned to the edge $e$, and likewise for $k_e,k_{\hat w,w}$ as well as $a_f,a_{u,v}$ for $f=\{u,v\}\in\fC$.
\end{enumerate}
\end{definition}

\begin{remark}\label{rem:intuition}
The quantities $\I^{\bzF,\bfC}(t,\bar t,x-\bar x)$ and $\K^{\bzF,\bfC}_{W;U}(t,x-\bar x;s)$ appear in the proof of Theorem \ref{thm:moment-estimates} as follows. 
Suppose, for simplicity, that $X$ is $\R$-valued and
that we aim to show $\|\scal{t^\delta X^{\tau,\eps}_t,\phi^\lambda_z}\|^2_{L^2} \lesssim 1$ uniformly in $\eps \in (0,1)$.
We write
\begin{equ}
\|\scal{t^\delta X^{\tau,\eps}_t,\phi^\lambda_z}\|^2_{L^2} = \int_{\T^n}\!\!\int_{\T^n} t^{2\delta}\E[X^{\tau,\eps}_t(y)X^{\tau,\eps}_t(\bar y)]\phi^\lambda_z(y)\phi^\lambda_z(\bar y)\mrd y\mrd \bar y
\end{equ}
and proceed to bound the 2-point function $\E[X^{\tau,\eps}_t(y)X^{\tau,\eps}_t(\bar y)]$.
To this end, we use the definition of $X^{\tau,\eps}_t$ to write
$X^{\tau,\eps}_t(y)X^{\tau,\eps}_t(\bar y)$
as a linear combination of integrals involving heat kernels (possibly with derivatives) and $2\noise{\tau}$ Gaussian fields $X^\eps$.
By the Isserlis--Wick theorem, the expectation of every such integral is given by a sum over pairings of leaves $\fC$ of integrals of the form \eqref{eq:general-conv} with $\zF_
\rho,\zF_{\bar\rho}$ isomorphic to $\tau$, the factor $|x_u-x_v|^{-a_{u,v}}$ replaced by the
covariance $\E X^\e(u)X^\e(v)$,
with $\gamma=0$ at each edge and $\beta=0$ at every inner node, and with all leaves being Dirac vertices.
Therefore, to bound $\|\scal{t^\delta X^{\tau,\eps}_t,\phi^\lambda_z}\|^2_{L^2}$, it suffices to bound every such integral,
which is precisely the purpose of \theo{thm:general-bounds}.
The reason we consider the more general quantities \eqref{eq:general-conv}-\eqref{eq:general-conv_ts} (involving non-zero $\gamma,\beta$, etc.) is that they appear in our induction when proving the above special case.

We in fact use $\I^{\bzF,\bfC}(t,\bar t,x-\bar x)$ to prove \eqref{eq:Z_eps_diff} by considering one of the Gaussian fields in the integral associated to each tree as $X^\eps-X^{\bar\eps}$,
see \eqref{eq:covar-Y^tau-Wick}.
Likewise the more complicated quantity $\K^{\bzF,\bfC}_{W;U}(t,x-\bar x;s)$, which involves two time points $s,t$, naturally arises in the proof of \eqref{eq:Z_time_diff}, see \eqref{eq:covar-Y^tau-Wick'}.
\end{remark}

\begin{remark}
	We remark that $\I^{\bzF,\bfC}(t,\bar t,x-\bar x)$ and $\K^{\bzF,\bfC}_{W;U}(t,x-\bar x;s)$, as suggested by our notation, depend on $x$ and $\bar x$ only through $x-\bar x$. 
\end{remark}

\begin{remark}
	We suppose that $\zF$ in Definition \ref{def:DCF-I-K} consists of precisely two trees only for simplicity and because this is the situation we encounter later.
	Theorem \ref{thm:general-bounds} below extends, with suitable modifications, to situations with more trees.
	Moreover, the only properties of the heat kernel $G$ in \eqref{eq:general-conv}-\eqref{eq:general-conv_ts} that we use are the estimates in Lemma \ref{lem:heat-flow-estimates},
	so $G$ could be replaced by any kernel satisfying analogous estimates.
\end{remark}

\begin{definition}\label{def:Lambda}
	Let $(\bzF,\bfC)$ be as in \defi{def:DCF-I-K}\ref{pt:DCF}. 
	Define $b\colon L_{\zF}\to \R$ by $b_\ell=a_e$
	if there is a contracting edge $a_e$ incident to the leaf $\ell$ (note that such an edge is unique), and by $b_{\ell}=0$ otherwise. 
	For a branch $\sigma$ of $\zF$, we define 
	\begin{equ}
		\Lambda(\sigma) \eqdef |E_\sigma| + \sum_{e\in E_\sigma} \Big\{\gamma_e-\frac12 |k_e|\Big\} +\sum_{u\in  V_\sigma} \beta_u
		-\frac14\sum_{\ell\in L_\sigma} b_\ell \;.
	\end{equ}
	We then set $\Lambda(\zF)\eqdef\Lambda( \zF_{\rho})+\Lambda({ \zF_{\bar\rho}})$.
	(Recall $\zF_v$ is the unique branch of $\zF$ with root $v\in  V_\zF$.)
\end{definition}
Below, for a contraction $\fC$ and a subset of vertices $O\subset  V_{\zF}$, we denote by $\fC(O)$ the set of all the edges $e\in \fC$ such that $e$ is incident to an element of $O$.
We also use the notation $V^c_{\zF_{v}}\eqdef V_{\zF}\setminus  V_{\zF_{v}}$.
\begin{remark}\label{rem:Lambda}
	An equivalent definition of $\Lambda(\sigma)$ is to set $\Lambda(\zF_{\ell})=\beta_\ell-a_e/4$ for all $\ell\in L_{\zF}$ where $e\in \fC(\{\ell\})$, and then define by recursion for $u\in I_{\zF}\,$,
	\begin{equ}[eq:Lambda_induct]
		\Lambda( \zF_{u}) \eqdef \beta_u + \sum_{\ell\in C_u} \{\Lambda( \zF_{\ell}) + 1
		+ \gamma_{u,\ell} - |k_{u,\ell}|/2\}\;.
	\end{equ}

\end{remark}
\begin{example}\label{ex:Lambda} 
Suppose $(\bzF,\bfC)$ is a decorated contracted forest as defined in \defi{def:DCF-I-K}\ref{pt:DCF} with 
\begin{equ}
	V_{\zF}=\{\rho,\bar\rho,v,\bar v\}\,,\quad E_{\zF}=\{(\rho,v),(\bar\rho,\bar v)\}\,,\quad  R_{\zF}=\{\rho,\bar\rho\}\,,
\end{equ} 
(i.e. $\zF=\{\zF_\rho,\zF_{\bar\rho}\}$, where the two trees are isomorphic, each with a single edge, with respective roots $\rho,\bar\rho$ and leaves $v,\bar v$)
and with $\fC=\{\{v,\bar v\}\}$.
Then, by \defi{def:Lambda} (or using the inductive formula \eqref{eq:Lambda_induct}), we have
\begin{equs}[eq:ex-Lambda]
	\Lambda({\zF_{\rho}}) = 1+\beta_v+\beta_\rho+\gamma_{\rho,v}-a_{v,\bar v}/4-|k_{\rho,v}|/2\;, 
	\\ \Lambda({\zF_{\bar\rho}}) = 1+\beta_{\bar v}+\beta_{\bar{\rho}}+\gamma_{\bar\rho,\bar v}-a_{v,\bar v}/4-|k_{\bar\rho,\bar v}|/2\;.
\end{equs}
\end{example}
\begin{remark} 
	In light of Remark~\ref{rem:Lambda}, one should think of $2\Lambda(\sigma)$ as a generalised version of the function $|\sigma|$ defined by~\eqref{eq:homogeneity} and the line above it. 
	Intuitively, $2\Lambda(\tau)$
	represents the scaling dimension of the random field $X^\tau$ from \rem{rem:intuition}, which is reflected in the bounds \eqref{eq:general-bound} and \eqref{eq:Z_eps_diff}-\eqref{eq:Z_time_diff}.
\end{remark}
Below, we use the convention that $\max\emptyset = \min\emptyset = 0$.
\begin{theorem}\label{thm:general-bounds}
	Let $(\bzF,\bfC)$ be a decorated contracted forest as defined in \defi{def:DCF-I-K}\ref{pt:DCF} such that $\fC(\roots_{\zF})=\emptyset$. 
	Suppose that
	
	\begin{enumerate}[label=(\alph*)]
	\item \label{pt:edges} for all non-Dirac vertices $v\in V^{*}_{\zF}$,
	\begin{equ}[eq:edges]
	z_v\eqdef -\max \{a_e/4 \,:\, e\in \fC(V_{\zF_{v}})\cap \fC(V^c_{\zF_{v}}) \} - |k_{\hat v,v}|/2 + \gamma_{\hat v,v} >-1\;,
	\end{equ}
	
	\item \label{pt:branches} for all $v\in I^{*}_{\zF}$,
	\begin{equ}[eq:branches]
	 \Lambda( \zF_{v})+ \min \{a_e/4 \,:\, e\in \fC(V_{\zF_{v}})\cap \fC(V^c_{\zF_{v}}) \} > -1\;,
	 \end{equ}
	
	\item \label{pt:contractions} $a_e \in (0,n)$ for all $e\in \fC$.
	\end{enumerate}
	
	Then, 
	for any
	$0\le \theta \le \max \{a_e \,:\; e\in \fC(L_{\zF_{\rho}})\cap \fC(L_{\zF_{\bar\rho}})\}\,$,
	\begin{equ}[eq:general-bound]
		\I^{\bzF,\bfC}(t,\bar t,x-\bar x)
		\lesssim
		t^{\Lambda({\zF_{\rho}})+ \frac{\theta}{4}}\bar t^{\Lambda({\zF_{\bar\rho}})+ \frac{\theta}{4}}|x-\bar x|^{-\theta} \;.
	\end{equ}

	Suppose further that $U$ and $W$ are as in Definition~\ref{def:DCF-I-K}\ref{pt:kernel_diff}  and $0<s\le t\le 2s$.
	Consider $\kappa,\bar\kappa\in[0,1]$ such that
	\begin{equs}[eq:tipar-cond]{}
		&\kappa<1 + z_\ell\quad
		\textnormal{if $U\cap C_{\rho}$ contains one element $\ell$ which is non-Dirac},\\
		&\bar\kappa<1 + z_{\bar \ell}\quad
		\textnormal{if $U\cap C_{\bar \rho}$ contains one element $\bar\ell$ which is non-Dirac},
	\end{equs}
	where $z_\ell,z_{\bar\ell}$ are defined in \eqref{eq:edges}.
	Denote $\tipar\eqdef\kappa\bone_{|U\cap C_\rho|=1}+\bar\kappa\bone_{|U\cap C_{\bar\rho}|=1}$.
	Then,
	\begin{equ}[eq:general-bound_diff]
		\K^{\bzF,\bfC}_{W;U}(t,x-\bar x;s)
		\lesssim
		|t-s|^{\tipar} 	t^{\Lambda(\zF)+\frac{\theta}{2}-\tipar} |x-\bar x|^{-\theta} \;.
	\end{equ}
\end{theorem}

\begin{remark}
In the second part of Theorem \ref{thm:general-bounds},
due to condition~\ref{pt:edges}, there indeed exist $\kappa,\bar\kappa\in[0,1]$ satisfying \eqref{eq:tipar-cond} and,
if $U\neq\emptyset$, one can take $\tipar>0$ in \eqref{eq:general-bound_diff}.
\end{remark}
%
%
For the proof of Theorem \ref{thm:general-bounds}, we require several definitions and lemmas.
\begin{definition}
	Let $(\zF, \fC)$ be a contracted forest. We call a branch $\sigma$ of $\zF$ \emph{minimal for $\fC$} if the induced contraction on $\sigma$ is not empty and,
	for every strict branch $\tilde\sigma$ of $\sigma$, the induced contraction on $\sigma$ is empty.
	For instance, in Figure~\ref{fig:graphs}(a), both trees are minimal branches, but no strict branches are minimal,
	while in Figure~\ref{fig:graphs}(b) the only minimal branches are the leftmost and rightmost (non-leaf) strict branches. 

	We say that $\{u,v\}\in \fC$ is a \emph{safe deletion} if there exists a minimal branch $\sigma$ such that $u,v$ are both leaves of $\sigma$.
	See Figure~\ref{fig:unsafe-deletion} for an example. 
\end{definition}

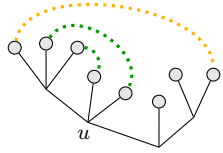
\begin{figure}[H] 
	\begin{center} 
		\begin{tikzpicture}  [scale=0.55,baseline=0]		 
			\node at (-.95,1.9)  [var] (a)  {};   
			\node at (-.20,2) [var] (b) {};  
			\node at (.55,1.9) [var] (c) {};  
			\node at (.95,1.2) [var] (d) {};  
			\node at (1.7,.8) [var] (e) {};  
			\node at (2.5,.6) [var] (f) {};  
			\node at (2.9,1.25) [var] (g) {};  
			\node at (3.8,1.25) [var] (h) {};  
			\node at (.8,.1) at (.7,-.18) {\scriptsize$u$};
			\draw (-.2,.9) -- (a); 
			\draw (-.2,.9) -- (b); 
			\draw (-.2,.9) -- (c); 
			\draw (-.2,.9) -- (.8,.1); 
			\draw (.8,.1) -- (d); 
			\draw (.8,.1) -- (e); 
			\draw (2.5,-.5) -- (f); 
			\draw (.8,.1) -- (2.5,-.5); 
			\draw (2.5,-.5) -- (3.33,.22); 
			\draw (3.33,.22) -- (g); 
			\draw (3.33,.22) -- (h); 
			\draw[orange!60!yellow, very thick, dotted]  (a) to[out=80, in=90, distance=1.55cm] (h) ; 
			\draw[bend left =95, darkgreen, very thick, dotted]  (b) to[out=90, in=95, distance=1cm] (e) ; 
			\draw[bend left =60, darkgreen, very thick, dotted]  (c) to (d) ; 
		\end{tikzpicture} 
	\end{center}\caption{\small The figure shows a \emph{contracted tree} with the dotted edges representing elements of the contraction;
	the green edges are \emph{safe deletion}s, while the orange edge is not. The figure further marks one of the inner nodes with `$u$' which is the root of the only \emph{minimal branch} of the tree for the drawn contraction.} 
	\label{fig:unsafe-deletion} 
\end{figure}

\begin{lemma}\label{lem:safe_del_exist} 
	Suppose that $(\zF, \fC)$ is a contracted tree with $\fC$ non-empty. Then there exists a safe deletion of $(\zF, \fC)$.
\end{lemma} 

\begin{proof} 
	There exists a minimal branch $\sigma$ of $\zF$. Then any $\{u,v\}\in \fC$ where $u,v$ are leaves of $\sigma$ is a safe deletion.
\end{proof} 

\begin{remark}\label{rem:min_branch} 
	If $\{u,v\}\in \fC$ is a safe deletion, then there exists a unique branch which contains $u,v$ as leaves and which is minimal -- this is the branch with root $o$, where $o$ is the maximal element (for the relation $\prec$) such that $o\prec u$ and $o\prec v$. 
\end{remark}
We furthermore require the following technical lemma, which includes a refinement of \cite[Lem.~3.3]{CCHS22_3D}.
For  $\gamma\in\R$ and a function $f\colon \T^n\setminus\{0\}\to\R$, define
\begin{equ}
	\|f\|_\gamma = \sup_{x\in \T^n\setminus\{0\}} |x|^\gamma |f(x)|\;.
\end{equ}
	\begin{lemma}\label{lem:heat-flow-estimates}
		Let $ n \ge 1 $, $ \gamma \in [0, n)$, and $k\in\N_0^n$.
		Let  $f\colon \T^n\setminus\{0\}\to\R$ be non-negative such that
		$	\|f\|_\gamma <\infty$.
		Then, there exists $C>0$, depending only on $n,\gamma,k$, such that
		for all $ \alpha\in [0, \gamma]$, $ t \in (0,1) $, and $ x\in \T^n\setminus\{0\}$, we have
		\begin{equ}[eq:heat-flow-est]
			\big(|G^{(k)}_t| * f\big)(x) \leq C t^{-\frac{\alpha+|k|}{2}}|x|^{\alpha-\gamma}\|f\|_\gamma\;.
		\end{equ}
		In particular,
		\begin{equ}[eq:double-heat-flow-est]
			\hspace*{-20pt}	\big(|G_t^{(k)}|  * f * |G_{\bar t}^{(\bar k)}|\big)(x)
			 \le C
			t^{-\frac{\alpha+|k|}{2}}
			\bar t^{-\frac{\bar\alpha+|\bar k|}{2}}\,
			|x|^{\alpha+\bar\alpha-\gamma}\|f\|_\gamma\;,
		\end{equ} 
		where $\bar t\in (0,1)$, $\bar k\in\N_0^n$, and $\bar\alpha\ge0$ is such that $\alpha+\bar\alpha\le\gamma$, with $C$ depending also on $\bar k$.

		Suppose further that $\delta\in\R$ and $0<s\le t\le 2s$. Then, for any $\kappa\in[0,1]$,
		\begin{equ}[eq:conv-est]
			\big(|t^\delta G^{(k)}_t-s^\delta G^{(k)}_s|*f\big)(x) \le C |t-s|^{\kappa} t^{\delta-\frac{\alpha+|k|}{2}-\kappa}|x|^{\alpha-\gamma}\|f\|_\gamma\;,
		\end{equ} 
		where the constant $C$ depends also on $\delta$.
		In particular,
		\begin{equ}[eq:double-conv-est_diff.]
			\big(| t^\delta G^{(k)}_{ t} \!- s^\delta G^{(k)}_{s}| * f * |G^{(\bar k)}_{\bar t}|\big)(x)
			\le C\,
			|t-s|^{\kappa} t^{\delta-\frac{\alpha+|k|}{2}-\kappa}
             \bar t^{-\frac{\bar\alpha+|\bar k|}{2}}		
           |x|^{\alpha+\bar\alpha-\gamma}\|f\|_\gamma\,,
		\end{equ}	
		and if, additionally, $\bar\delta\in\R$ and $0< \bar s \le \bar t \le 2\bar s$, then for any $\kappa,\bar\kappa\in[0,1]$,
		\begin{equs}[eq:double-conv-est_diff]{}
			\big(|t^\delta G^{(k)}_t \!-s^\delta G^{(k)}_s| * f * |\bar t^{\bar\delta}G^{(\bar k)}_{\bar t}\!-\bar s^{\bar\delta}G^{(\bar k)}_{\bar s}|\big)(x)
			\le C\, 
			|t-s|^{\kappa} |\bar t-\bar s|^{\bar\kappa}\\\times
			t^{\delta-\frac{\alpha+|k|}{2}-\kappa}
			\bar t^{\bar\delta-\frac{\bar\alpha+|\bar k|}{2}-\bar\kappa}|x|^{\alpha+\bar\alpha-\gamma}\|f\|_\gamma\;,
		\end{equs}
		with the constants $C$ in \eqref{eq:double-conv-est_diff.} and \eqref{eq:double-conv-est_diff} depending also on $\bar k$ and $\bar k,\bar\delta$ respectively.
	\end{lemma}
For the proof, note that, since $ G_t(x) = \sum_{m \in\Z^n} (4\pi)^{-n/2}t^{-n/2}\mre^{-|x+m|^2/(4t)}$,
	\begin{equ}[eq:G_k_upperbound]
		|G^{(k)}_t(x)|\lesssim \mre^{-|x|^2/(4t)} t^{-n/2} (t^{-|k|/2} + t^{-|k|}|x|^{|k|})\;,
	\end{equ}
	which combined with
	$ \mre^{-|z|} \lesssim |z|^{-p}$, $ p\geq 0$, implies
	\begin{equ}[eq:heat_kernel_bound]
		|G^{(k)}_t(x)|\lesssim \mre^{-|x|^2/(8t)}
		t^{-\frac n2 -|k|+p} |x|^{-2p + |k|}\;,\qquad \forall p\geq |k|/2\;.
	\end{equ}	
	In particular, taking $p=|k|/2$ in \eqref{eq:heat_kernel_bound}, we get
	\begin{equ}[eq:G^k-bound]
		|G^{(k)}_t(x)|\lesssim \mre^{-|x|^2/(8t)}							t^{-\frac n2 -\frac{|k|}{2}} \asymp t^{-\frac{|k|}{2}} G_{2t}(x)\;.
	\end{equ}

	\begin{proof}[of Lemma \ref{lem:heat-flow-estimates}]
		To show \eqref{eq:heat-flow-est}, it suffices to prove the bound for $\alpha=0$ and $\alpha=\gamma$ with the claim for $\alpha\in(0,\gamma)$ following by interpolation.
		
			Suppose first $\alpha=\gamma$. Then by $|f|_{\CC^{-\gamma}}\lesssim \|f\|_\gamma$ and \eqref{eq:heat_flow_estimates},
			$|G_t * f|_\infty \lesssim t^{-\gamma/2}\|f\|_\gamma$
			and, by also using \eqref{eq:G^k-bound}, we have
			\begin{equ}
			\big||G_t^{(k)}| * f\big|_\infty \lesssim t^{-\frac{|k|}{2}}|G_{2t} * f|_\infty \lesssim t^{-\gamma/2-|k|/2}\|f\|_\gamma\;,
			\end{equ}
			which proves the desired bound.
			
			Suppose next $\alpha=0$.
			We write
			\begin{equs}[eq:Pkf_split]{}
				&\big(|G^{(k)}_t|* f\big)(x) = \int_{\T^n}\!\!| G^{(k)}_t(x-z)| f(z) \mrd z\,
				\\
				&\qquad \lesssim  \|f\|_\gamma\Big(
				\int_{|z|<|\frac{x}{2}|}\hspace*{-5pt} |G^{(k)}_t(x-z)| |z|^{-\gamma} \mrd z\,
				+ \int_{|z|\geq |\frac{x}{2}|}\hspace*{-5pt}  |G^{(k)}_t(x-z)| |z|^{-\gamma} \mrd z
				\Big)\;.\quad
			\end{equs}
			Applying \eqref{eq:G^k-bound} and the fact that $\int\mrd x\mre^{-|x|^2/(4t)} t^{-\frac{n}{2}} \lesssim 1$,
			we obtain 
			$\int |G^{(k)}_t(y)|\mrd y \lesssim t^{-|k|/2}$.
			Therefore, the second integral in the second line of \eqref{eq:Pkf_split} is of order $t^{-|k|/2}|x|^{-\gamma}$.
			
			For the first integral, we remark that $\int_{|z|<|x|/2} |z|^{-\gamma}\mrd z \asymp |x|^{n-\gamma}$.
			Furthermore, $|z|<|x|/2$ implies $|x-z|\asymp |x|$ and thus $|G^{(k)}_t(x-z)|\lesssim t^{-|k|/2}|x|^{-n}$ where we used \eqref{eq:heat_kernel_bound} with $p=\frac{n+|k|}{2}$.
			Therefore the first integral is also of order $t^{-|k|/2}|x|^{-\gamma}$ which completes the proof of~\eqref{eq:heat-flow-est}.
			
			The estimate~\eqref{eq:double-heat-flow-est} follows from applying~\eqref{eq:heat-flow-est} twice with $t, k,\alpha, \gamma$ and $\bar t, \bar k,\bar\alpha, \gamma-\alpha$ (instead of $t,k,\alpha,\gamma$ therein) respectively.
			
	To prove \eqref{eq:conv-est} for $0<s< t\le 2s$, suppose first $\delta=0$.
	Similarly to \eqref{eq:G_k_upperbound}, one has, for all $t>0$ and $x\in\T^n$,
	\begin{equs}
		\big|\partial_t G^{(k)}_t(x)\big| &\lesssim
		\big(t^{-1}+ |x|^2 t^{-2}\big) t^{-\frac{n}{2}} 
		\mre^{-\frac{|x|^2}{4t}} \big(t^{-\frac{|k|}{2}} + t^{-|k|}|x|^{|k|}\big)\\
		&\lesssim t^{-\frac{n}{2}} 
		\mre^{-\frac{|x|^2}{8t}} t^{-\frac{|k|}{2}-1}\;,
	\end{equs}
	where we used $ \mre^{-|z|} \lesssim |z|^{-|k|/2-1}$ in the second inequality.
	It thus follows that for all $x\in\T^n$, there exists $r\in(s,t)$ such that
	\begin{equ}
		\big|G^{(k)}_t(x)-G^{(k)}_s(x)\big|\lesssim
		|t-s| r^{-1-\frac{|k|}{2}} G_{2r}(x)\;,
	\end{equ}
	and hence, using the assumption $ s< t\le 2s$, we have, for all $\kappa\in[0,1]$,
	\begin{equ}
		\big|G^{(k)}_t(x)-G^{(k)}_s(x)\big|\lesssim
		|t-s|^\kappa t^{-\kappa-\frac{|k|}{2}} G_{2t}(x)\;.
	\end{equ}
    Therefore, by \eqref{eq:heat-flow-est} (applied with $k=0$), we obtain
	\begin{equ}[eq:heat-flow-est_diff]
		\big( \big|G^{(k)}_t-G^{(k)}_s\big| * f \big)(x) \le C |t-s|^\kappa t^{-\frac{\alpha+|k|}{2}-\kappa}|x|^{\alpha-\gamma}\|f\|_\gamma\;,
	\end{equ}
	which proves \eqref{eq:conv-est} for $\delta=0$.
	Suppose now $\delta\neq 0$ and write
	\begin{equ}
		\big|t^\delta G^{(k)}_t-s^\delta G^{(k)}_s\big| \le \big|t^\delta-s^\delta\big| 
		\big|G^{(k)}_t(x)\big| + s^\delta\big|G^{(k)}_t(x)-G^{(k)}_s(x)\big|\;,
	\end{equ}
	from which we have, for all $\kappa\in[0,1]$,
	\begin{equ}
		\textnormal{LHS of } \eqref{eq:conv-est} \lesssim |t-s|^\kappa t^{\delta-\kappa} \big(\big|G^{(k)}_t\big|*f\big)(x) + t^\delta\big(\big|G^{(k)}_t-G^{(k)}_s\big|*f\big)(x)\;,
	\end{equ}
	where we also used that $t\in [s,2s]$ which implies $|t^\delta-s^\delta| \lesssim (t-s) t^{\delta-1}$. Therefore, applying \eqref{eq:heat-flow-est} and \eqref{eq:heat-flow-est_diff}, we obtain the bound \eqref{eq:conv-est}.

	Finally, the last two estimates are simply applications of \eqref{eq:conv-est} and \eqref{eq:heat-flow-est}.
	\end{proof}

\begin{proof}[of Theorem~\ref{thm:general-bounds}]
	\emph{Proof of \eqref{eq:general-bound}}.
	We proceed by induction on $ |E_{\zF}|$.
	There are two base cases.

	\textbf{Base case 1.} \label{bcase1}
	$|E_{\zF}|=0$, in which case ${ \zF_{\rho}}=(\{\rho\},\emptyset)$, ${ \zF_{\bar\rho}}=(\{\bar\rho\},\emptyset)$, and thus
	\begin{equ}
		\I^{\bzF,\bfC}(t,\bar t,x-\bar x)= g_\rho(t_{\rho}) g_{\bar \rho}(t_{\bar\rho}) =  t^{\beta_{\rho}} \bar t^{\beta_{\bar\rho}}\;,
	\end{equ} 
	which is the requested bound since $\Lambda( \zF_{\rho})=\beta_\rho$ and $\Lambda( \zF_{\bar\rho})=\beta_{\bar\rho}$ and the upper bound on $\theta$ is zero by our convention $\max\emptyset=0$.

	\textbf{Base case 2.} \label{bcase2}
	$(\bzF,\bfC)$ is as in Example~\ref{ex:Lambda}:
	$ \zF_{\rho}$ and $ \zF_{\bar\rho}$ are respectively
	$I^{(j)}(v)$ and $I^{(\bar j)}(\bar v)$ and $\fC\neq\emptyset$,
	where we write $I^{(j)}(v)$ for the decorated tree $\<IXi>$ with a root $\rho$, one leaf $v$, and one edge $f=(\rho,v)$ with decoration $j=k_f\in\N^n_0$,
	and similarly for $I^{(\bar j)}(\bar v)$ with $\bar j=k_{\bar f}$ and $\bar f=(\bar\rho,\bar v)$.

	In this case,
	$\I^{\bzF,\bfC}(t,\bar t,x-\bar x)$ equals
	\begin{equ}[eq:general-bound-first]
		t^{\beta_\rho} \bar t^{\beta_{\bar{\rho}}}\!\! 
		\int_{0}^{t}\!\!\int_{0}^{\bar t}\!\mrd t_{\bar v}\mrd t_v (t-t_v)^{\gamma_f}(\bar t-t_{\bar v})^{\gamma_{\bar{f}}} g_v(t_v) g_{\bar v}(t_{\bar v})
		\big(\big|{G}_{t-t_v}^{(j)}\big| * \big|\cdot\big|^{-a_e} * \big|{G}_{\bar t-t_{\bar v}}^{(\bar j)}\big| \big)(x-\bar x)\;,
	\end{equ}
	where $e=\{v,\bar v\}$ is the only element of $\fC$.
	Then, using \eqref{eq:double-heat-flow-est} with $ \gamma=a_{{e}}$ (which is in $(0,n)$ by assumption~\ref{pt:contractions}),  $\alpha=\bar\alpha=(a_{{e}}-\theta)/2$ (so $\alpha+\bar\alpha\in[0,\gamma]$), and $k=j, \bar k=\bar j$, we have, uniformly in $z\eqdef x-\bar x\in \T^n\setminus\{0\}$,
	\begin{equs}
		\big(\big|{G}_{t-t_v}^{(j)} \big| * \big|\cdot\big|^{-a_e} * \big|{G}_{\bar t-t_{\bar v}}^{(\bar j)}\big| \big)(z)
		\lesssim (t-t_v)^{\frac{-2|j|-a_{{e}}+\theta}{4}}(\bar t- t_{\bar v})^{\frac{-2|\bar j|-a_{{e}}+\theta}{4}}|z|^{-\theta}\;.
	\end{equs}
	Estimating \eqref{eq:general-bound-first} using the last inequality, integrating over time, and 
	recalling the quantities $\Lambda({\zF_{\rho}}),\Lambda({\zF_{\bar\rho}})$ \eqref{eq:ex-Lambda} from Example~\ref{ex:Lambda},
	we obtain the claimed bound. 
	(Note that, if the leaf $u\in\{v,\bar v\}$ is non-Dirac, then
	by assumption~\ref{pt:edges} applied to $u$, the exponent of $(t-t_u)$ in the resulting time integral is greater than $-1$
	and, by definition, $g_u(t) = t^{\beta_u}\bone_{t>0}$ and $\beta_u>-1$,
	and thus the singularities are integrable.
	On the other hand, if $v$ is Dirac, then $g_v=\delta_0$ and the corresponding time integral simply equals $t^{\gamma_f-|j|/2 -a_e/4+\theta/4}$, and likewise for $\bar v$.)
	
	Suppose now we are not in either of the base cases (in particular $|E_{\zF}|\ge 1$). We consider the following three (mutually exclusive) cases (which,
	since we assumed that $\fC(\roots_\zF)=\emptyset$, cover all possible cases)
	and prove the bound \eqref{eq:general-bound} for each case using the inductive hypothesis.

	\textbf{Case 1.}\label{case1}
	There exists a leaf $v\in L^*_{\zF}$ which is not paired with any other leaf. 
	Denote by $u$ the parent of $v$.
	In this case, the variables $x_{v},t_{v}$  in the integral~\eqref{eq:general-conv}
	appear only in the term
	\begin{equ}[eq:case1]
		\int_{D} \mrd t_{v}\mrd x_{v} \big|G^{(k_{u,v})}_{t_{u}-t_v}(x_{u}-x_v)\big| g_v(t_{v})(t_{u}-t_{v})^{\gamma_{u,v}}\;.
	\end{equ}
	Using the estimate $\int_{\T^n} |G^{(k)}_t(y)|\mrd y \lesssim t^{-|k|/2}$ (which follows from \eqref{eq:G^k-bound}),
	we obtain
	\begin{equ}
		\eqref{eq:case1}\lesssim \int_{0}^{t_{u}} g_v(t_{v})(t_{u}-t_{v})^{\gamma_{u,v}-\frac{|k_{u,v}|}{2}} \mrd t_{v}
		\propto t_{u}^{1+\beta_{v}+\gamma_{u,v}-\frac{|k_{u,v}|}{2}}\;,
	\end{equ}
	where in the end we used that, if $v$ is non-Dirac, we have $\gamma_{u,v}-|k_{u,v}|/2>-1$ (which is due to assumption~\ref{pt:edges}
	applied to $v\in V_{\zF}$) and $\beta_v>-1$ (which holds by definition).

	Now, let $({\tilde{\bzF}},\bfC)$ be the decorated 
	contracted forest left after deleting $v,(u,v)$, so that $\tilde{\zF}=(V_{\zF}\setminus\{v\}, E_{\zF}\setminus\{(u,v)\})$,
	and changing the decoration at $u$ to
	$\tilde\beta_{u}\eqdef \beta_{u} + 1 +\beta_v+\gamma_{u,v}-|k_{u,v}|/2$.
	(All other assignments remain unchanged.)

	We remark that, if $u \in L^*_{\tilde\zF}$, then $\tilde\beta_u = \Lambda(\zF_u)$ and thus  $\tilde{\beta}_{u}>-1$ due to assumption~\ref{pt:branches} applied to $u\in V^*_{\zF}$. (In particular, $u$ is a non-Dirac vertex in the new graph.)
	Therefore, $(\tilde{\bzF},\bfC)$ is a decorated contracted forest in the sense of \defi{def:DCF-I-K}\ref{pt:DCF}.

	It thus follows that
	\begin{equ}
	\I^{\bzF,\bfC}(t,\bar t, x-\bar x)\lesssim \I^{\tilde{\bzF},\bfC}(t,\bar t, x-\bar x)\;.
	\end{equ}
	Note that $\Lambda(\tilde{\zF}_{\rho})=\Lambda(\zF_{\rho})$ and that the restrictions on $\theta$ do not change. It therefore remains to show that $(\tilde{\bzF},\bfC)$
	satisfies conditions~\ref{pt:edges}-\ref{pt:contractions} to conclude the proof by applying the inductive hypothesis.
	Conditions~\ref{pt:edges} and \ref{pt:contractions} are clearly satisfied while condition \ref{pt:branches} follows from the fact that $\Lambda(\tilde\zF_{w})=\Lambda(\zF_{w})$ for all $w\in V_{\tilde\zF}$
	by our choice of $\tilde\beta_u$.
	
	\textbf{Case 2.}\label{case2}
	Every leaf in $L^*_{\zF}$ is paired in $\fC$
	and there exist $u\in I^*_{\zF}$
	and $v\in L_{\zF}$ such that $v$ is the only child of $u$.
	Let $w$ be the leaf paired with $v$ in $\fC$.
	In this case, the variables $x_{v},t_{v}$  in the integral~\eqref{eq:general-conv}
	appear only in the term
	\begin{equ}
		\int_{D} \mrd t_{v}\mrd x_{v} |x_v-x_w|^{-a_{v,w}}(t_{u}-t_{v})^{\gamma_{u,v}}\big| G^{(k_{u,v})}_{t_{u}-t_v}(x_{u}-x_v) \big| g_v(t_{v})\;.
	\end{equ}
	By \eqref{eq:heat-flow-est} (applied with $\gamma=a_{v,w}$, $k=k_{u,v}$, and $\alpha=0$), this integral is bounded above by a multiple of
	\begin{equ}
		|x_u-x_w|^{-a_{v,w}}\!\int_{0}^{t_{u}} (t_{u}-t_v)^{\gamma_{u,v}-\frac{|k_{u,v}|}{2}} g_v(t_v)\, \mrd t_v
		\propto
		|x_u-x_w|^{-a_{v,w}}t_{u}^{1+\beta_v+\gamma_{u,v}-\frac{|k_{u,v}|}{2}},\quad
	\end{equ}
	where, in case $v$ is non-Dirac, we use assumption~\ref{pt:edges} (which implies $\gamma_{u,v} - |k_{u,v}|/2>-1$) and $\beta_v>-1$ when taking the integral.
	We therefore obtain
	\begin{equ}[eq:case2-induct_I]
		\I^{\bzF,\bfC}(t,\bar t,x-\bar x) \lesssim \I^{{\tilde{\bzF}},\tilde{\bfC}}(t,\bar t,x-\bar x)\;,
	\end{equ}
	where $({\tilde{\bzF}},\tilde{\bfC})$ consists of the contracted tree $(\tilde{\zF},\tilde{\fC})$, with $\tilde{\zF}\eqdef( V_{\zF}\setminus\{v\}, E_{\zF}\setminus\{(u,v)\})$,
	and $\tilde{\fC}\eqdef \fC\sqcup\{\{u,w\}\}\setminus\{\{v,w\}\}$, which is decorated as follows.	
	The contracting edge $\{u,w\}$ is assigned $a_{v,w}$, and the vertex decoration at $u$ changes to $\tilde\beta_{u} \eqdef \beta_{u}+1+\beta_{v}+\gamma_{u,v}-|k_{u,v}|/2$.
	We let all other assignments remain the same.

	Observe that $\Lambda(\tilde \zF_{\nu}) = \Lambda( \zF_{\nu})$ for all $\nu\in V^{*}_{\tilde{\zF}}$ due to the choice of $\tilde{\beta}_u$.
	In particular, for the leaf $u\in V_{\tilde{\zF}}$,
	\begin{equ}
	\tilde\beta_u = \Lambda(\zF_{u}) + a_{v,w}/4 > -1\;,
	\end{equ}
	where in the inequality we applied assumption~\ref{pt:branches} to $u\in V_{\zF}$.
	As a consequence, $(\tilde{\bzF},\tilde{\bfC})$ possesses the structure of a decorated contracted forest in the sense of \defi{def:DCF-I-K}\ref{pt:DCF} and,
	in particular, the right-hand side of \eqref{eq:case2-induct_I} indeed makes sense.
	We remark further that $\tilde{\fC}(\roots_{\tilde\zF})=\emptyset$ since $u\notin \roots_\zF$ and ${\fC}(\roots_{\zF})=\emptyset$ by assumption.
	
	We now claim that $({\tilde{\bzF}},\tilde{\bfC})$ satisfies conditions~\ref{pt:edges}-\ref{pt:contractions}.
	First, condition~\ref{pt:contractions} obviously holds since $\{a_e\,:\; e\in\tilde{\fC}\}=\{a_e\,:\; e\in \fC\}$.
	Furthermore, supposing without loss of generality that $u,v\in V_{\zF_{\rho}}$, we have $\tilde{\zF}_{\bar\rho}=\zF_{\bar\rho}$, so the conditions which correspond $\tilde\zF_{\bar\rho}$ clearly hold.
	It thus suffices to verify the conditions associated with $\tilde{ \zF}_{\rho}$. 
	Condition~\ref{pt:edges} immediately follows for each $\nu\in I^*_{\tilde \zF_{\rho}}$ since it holds for each $\nu\in I^*_{\zF_{\rho}}$, the edge decorations are intact, and the $\max$ in \eqref{eq:edges} does not change.
	Finally, condition~\ref{pt:branches} holds because it holds in $(\bzF,\bfC)$ and, due to the change of decorations, for non-leaf vertices in $V^*_{\tilde{\zF}}$ the left-hand side of \eqref{eq:branches} remains the same. 
	
	Therefore, conditions~\ref{pt:edges}-\ref{pt:contractions} are satisfied for the new graph and thus the conclusion follows by the inductive hypothesis. (We remark that $\Lambda(\tilde \zF_{\rho})=\Lambda(\zF_{\rho})$ and that the restrictions on $\theta$ do not change.)

	\textbf{Case 3.}\label{case3} 
	Every leaf in $L^*_{\zF}$ is paired in $\fC$
	and if any $w\in I^*_{\zF}$ is a parent of a leaf, then $w$ has at least two children (and, as mentioned earlier, we are not in \hyperref[bcase2]{Base case 2}).
	In this case, we let $e = \{u,v\}\in\fC$
	be a safe deletion if one exists in $\fC$, and if no safe deletions exist,
    then we choose $e = \{u,v\}\in\fC$ such that $a_e = \min_{f\in\fC} a_f$.

	We claim that
	\begin{equ}[eq:case3_induct_I]
		\I^{\bzF,\bfC}(t,\bar t,x-\bar x) \lesssim
		\I^{\tilde{\bzF},\tilde{\bfC}}(t,\bar t,x-\bar x)\;,
	\end{equ}
	for a decorated contracted forest $({\tilde{\bzF}},\tilde{\bfC})$ which we define as follows and to which we can apply the inductive hypothesis.

	Delete $e$ from the graph and denote $\tilde{\fC}\eqdef \fC\setminus\{e\}$ and let all the assignments to its contracting edges remain unchanged. Remove moreover the vertices $u,v$ and edges $(\hat u,u)$, $(\hat v,v)$ (recall $\hat u$ respectively $\hat v$ are the parents of $u$ respectively $v$)
	and denote $\tilde{ \zF}\eqdef( V_{\zF}\setminus\{u,v\}, E_{\zF}\setminus \{(\hat u,u),(\hat v,v)\})$ (so indeed $\tilde \zF_{\rho}$ and $\tilde \zF_{\bar\rho}$ are the two connected components, i.e. trees, of the forest $\tilde{ \zF}$).
	We then let all the vertex and edge decorations be inherited from the original graph except for the vertex decorations at $\hat u,\hat v$, which  are changed to the following. If $\hat u\neq \hat v$, we set
	\begin{equs}
		\tilde\beta_{\hat u}&\eqdef 1+\beta_{\hat u}+\beta_{u}-a_e/4 - |k_{\hat u,u}|/2+\gamma_{\hat u,u}\;,\,\qquad  
		\\
		\tilde\beta_{\hat v}&\eqdef 1+\beta_{\hat v}+\beta_{v}-a_e/4 - |k_{\hat v,v}|/2+\gamma_{\hat v,v}\,,\qquad   
	\end{equs}
	otherwise we set
	\begin{equ}
		\tilde\beta_{\hat u}\eqdef 2+\beta_{\hat u}+\beta_{u}+\beta_v-a_e/2 - |k_{\hat u,u}|/2+\gamma_{\hat u,u} - |k_{\hat u,v}|/2+\gamma_{\hat u,v}\;.\qquad
	\end{equ}
	If $\hat u \in L^*_{\tilde\zF}$, then $\hat u =\hat v$ and the only children of $\hat u$ in $\zF$ are $u,v$
	(this follows from the conditions of Case 3, in particular the fact that we are not in Cases 1 or 2).
	Moreover, by the above assignments, one has $\tilde\beta_{\hat u} = \Lambda(\zF_{\hat u})$ and thus $\tilde\beta_{\hat u}>-1$ due to assumption~\ref{pt:branches} applied to $\hat u\in I^*_{\zF}$ (recall that $\min\emptyset=0$ by our convention). In particular, if $\hat u \in L^*_{\tilde\zF}$, then it is non-Dirac in the new graph $\tilde\zF$.

	Therefore, $({\tilde{\zF}},\tilde{\fC})$ equipped with the data described above constitutes a decorated contracted forest, which we denote as usual by $(\tilde{\bzF},\tilde{\bfC})$.
	
	To prove~\eqref{eq:case3_induct_I}, we remark that
	the inner integral therein corresponding to the vertices $u,v$ is equal to
	\begin{equs}[eq:inner_int]
		\!\!\!\!\int_{0}^{t_{\hat u}}\!\!\!\int_0^{t_{\hat v}}\!\!\!\int_{(\T^n)^{\{u,v\}}} 
		\mrd t_u \mrd t_v\mrd x_u \mrd x_v \,
		(t_{\hat u}-t_u)^{\gamma_{\hat u,u}}
		(t_{\hat v}-t_v)^{\gamma_{\hat v,v}} g_u(t_u) g_v(t_v) \\
		\,  |x_u-x_v|^{-a_{e}} \big| G^{(k_{\hat u,u})}_{t_{\hat u}-t_u}(x_{\hat u}-x_u)G^{(k_{\hat v,v})}_{t_{\hat v}-t_v}(x_{\hat v}-x_v)\big|\;.
	\end{equs}
	Applying~\eqref{eq:double-heat-flow-est} with $\gamma = 2\alpha= 2\bar\alpha= a_e$ (which is in $(0,n)$ by assumption~\ref{pt:contractions}), the spatial integral is of order
	\begin{equ}
		(t_{\hat u}-t_u)^{-\frac{1}{4}a_e - \frac{1}{2}|k_{\hat u,u}|} (t_{\hat v}-t_v)^{-\frac14 a_e - \frac12|k_{\hat v,v}|}\;.
	\end{equ}
	 Using this in \eqref{eq:inner_int}, we obtain 
	\begin{equs}
		\int_{0}^{t_{\hat u}}\!\!\!\!\int_0^{t_{\hat v}}\!
		\mrd t_u \mrd t_v (t_{\hat u}-t_u)^{\gamma_{\hat u,u}-\frac{1}{4}a_e - \frac{1}{2}|k_{\hat u,u}|} (t_{\hat v}-t_v)^{\gamma_{\hat v,v}-\frac14 a_e - \frac12|k_{\hat v,v}|} g_u(t_u) g_v(t_v) 
		\\
		\propto t_{\hat u}^{\tilde\beta_{\hat u}}t_{\hat v}^{\tilde\beta_{\hat v}}\,,
	\end{equs}
	where in the proportionality we used that if $u$ is non-Dirac, by condition~\ref{pt:edges} respectively by definition, the exponent of $(t_{\hat u}-t_u)$ respectively $t_u$ is greater than $-1$, and likewise for $v$ and the integrands indexed by $v$.
	Remarking that $x_u,x_v,t_u,t_v$ do not feature in any terms in the integrand of~\eqref{eq:general-conv} except for~\eqref{eq:inner_int},
	the above yields the claimed bound~\eqref{eq:case3_induct_I}.
	
	Next, note that one still has $\tilde{\fC}(\{\rho,\bar\rho\})=\emptyset$ as it is satisfied by $\fC$ and we made only a deletion from $\fC$ without any addition to the contracting edges, from which we also deduce that condition~\ref{pt:contractions} holds for $\tilde\bfC$.
	
	We next verify conditions~\ref{pt:edges}-\ref{pt:branches} for $(\tilde{\bzF},\tilde{\bfC})$.
	For condition~\ref{pt:edges}, if $w\in V^{*}_{\tilde{\zF}}$ is non-Dirac, then we also have that $w\in V^{*}_{\zF}$ is non-Dirac
	and moreover
	\begin{equ}
		-\max \{a_f/4 \,:\, f\in \tilde \fC(V_{\tilde \zF_{w}})\cap \tilde \fC(V^c_{\tilde \zF_{w}}) \} \geq - \max \{a_{f}/4 \,:\, f\in \fC(V_{\zF_{w}})\cap \fC(V^c_{\zF_{w}}) \} 
	\end{equ}
	because removing $e$ from $\fC$ can only cause the $\max$ to decrease (recall our convention $\max \emptyset =0$).
	Therefore condition~\ref{pt:edges} is satisfied for $({\tilde{\bzF}},\tilde{\bfC})$ since it is satisfied for $({\bzF},\bfC)$.

	For condition~\ref{pt:branches}, consider $w\in I^*_{\tilde{\zF}}$.
	By construction, $\Lambda(\tilde  \zF_{w}) = \Lambda( \zF_{w})$.
	It is thus enough to show that
	\begin{equ}[eq:min_increase]
		\min \{a_f/4 \,:\, f\in \tilde \fC(V_{\tilde \zF_{w}})\cap \tilde \fC(V^c_{\tilde \zF_{w}}) \} \geq \min \{a_f/4 \,:\, f\in \fC(V_{\zF_{w}})\cap \fC(V^c_{\zF_{w}}) \}\;.
	\end{equ}
	First, if $\{u,v\} \notin \fC(V_{\zF_{w}})\cap \fC(V^c_{\zF_{w}})$,
	then we have equality in \eqref{eq:min_increase}
	since the corresponding sets over which the $\min$ is taken are the same.
	It remains to consider the case that $\{u,v\} \in \fC(V_{\zF_{w}})\cap \fC(V^c_{\zF_{w}})$. 
	Recalling our convention that $\min \emptyset = 0$, it suffices to show that $\tilde \fC(V_{\tilde \zF_{w}})\cap \tilde \fC(V^c_{\tilde \zF_{w}})$ is not empty.

	To this end, note first that since $\{u,v\} \in \fC(V_{\zF_{w}})\cap \fC(V^c_{\zF_{w}})$ and $w\notin L_{\zF}$, either  $w\prec u$ and $ w\npreceq v$, or $w\npreceq u$ and $ w\prec v$. Suppose without loss of generality that $w\prec u$ and $w\npreceq v$, so in particular $w\preceq \hat u$ 
	and thus $\hat u\notin \roots_\zF$ since $w\notin \roots_\zF$.
	Therefore, by the assumptions of this case, $\hat u$ has at least two children, and thus there exists $u'\in L_\zF \setminus \{u,v\}$ with $w\prec u'$.
	Again by the assumptions of this case, there exists $\bar e \in \fC(V_{\zF_{w}})$ which is incident to $u'$. We clearly have $\bar e\neq e$.

	We next claim that $\fC_{\zF_{w}}=\emptyset$, where $\fC_{\zF_{w}}$ is as in Definition \ref{def:CT}.
	Indeed, if $e$ is a safe deletion, then let $o \in I_{\zF}$ be as in \rem{rem:min_branch}, so that $\zF_o$ is minimal and $o\prec u$ and $o\prec v$. Then necessarily $o\prec w$, so one has $\fC_{\zF_w} = \emptyset$ by minimality of $\zF_o$.
	If, on the other hand, $e$ is not a safe deletion, then, by our choice of $e$, no safe deletions exist in $(\zF,\fC)$ and thus $\fC_{\zF_\rho}\cup \fC_{\zF_{\bar\rho}} = \emptyset$ by \lem{lem:safe_del_exist}, so again $\fC_{\zF_{w}}=\emptyset$, which proves the claim.
	

	Since $\fC_{\zF_{w}}=\emptyset$, it follows that $\bar e$ necessarily belongs to $\fC(V_{\zF_{w}}^c)$.
	Finally, since $\bar e$ is still present in $\tilde{\fC}$, it follows that $\bar e \in \tilde \fC(V_{\tilde \zF_{w}})\cap \tilde \fC(V^c_{\tilde \zF_{w}})$, which proves that $\tilde \fC(V_{\tilde \zF_{w}})\cap \tilde \fC(V^c_{\tilde \zF_{w}}) \neq\emptyset$ and completes the verification of condition~\ref{pt:branches} for $({\tilde{\bzF}},\tilde{\bfC})$.

	To complete the proof in \hyperref[case3]{Case 3} by induction, we remark that
	\begin{equ}[eq:max_increase]
		\max \{a_f \,:\; f\in \tilde\fC(L_{\tilde \zF_{\rho}})\cap \tilde\fC(L_{\tilde \zF_{\bar\rho}})\} = \max \{a_f \,:\; f\in \fC(L_{\zF_{\rho}})\cap \fC(L_{\zF_{\bar\rho}})\}\;,
	\end{equ}
	which means that the upper bound on $\theta$ does not decrease in the new graph.
	To prove~\eqref{eq:max_increase}, if $e$ is a safe deletion, then we have equality in \eqref{eq:max_increase} because the corresponding sets over which the max is taken are the same.
	If, on the other hand, $e$ is not a safe deletion, then recall that there are no safe deletions in $\fC$ and moreover $a_e = \min\{a_f : f \in \fC(L_{\zF_{\rho}})\cap \fC(L_{\zF_{\bar\rho}})\}$, so \eqref{eq:max_increase} follows once we show that $|\fC|\geq 2$.
	To this end, note that the absence of safe deletions implies $\fC_{\zF_{\rho}}\cup \fC_{\zF_{\bar\rho}} = \emptyset$ by Lemma~\ref{lem:safe_del_exist},
	which, combined with the conditions of \hyperref[case3]{Case 3} (recalling that we are not in \hyperref[bcase2]{Base case 2} and $\fC(\roots_{\zF})=\emptyset$) implies $|\fC|\geq 2$ as required.

	Finally, by the aforementioned changes to decorations, $\Lambda( \zF_{\rho}) = \Lambda(\tilde{\zF}_{\!\rho})$, and likewise for $\bar\rho$.
	Consequently, applying the inductive hypothesis to $(\tilde{\bzF},\tilde{\bfC})$ together with \eqref{eq:case3_induct_I} concludes the proof.

	\emph{Proof of \eqref{eq:general-bound_diff}}.
	Suppose first $U=\emptyset$ and denote
	\begin{equ} 
		\tilde\I^{\bzF,\bfC}_{W}(t,s,x-\bar x)\eqdef
		\K^{\bzF,\bfC}_{W;\emptyset}(t,x-\bar x;s) \;,
	\end{equ}
	which, given that $t=\bar t$, is a slight variant of $\I^{\bzF,\bfC}(t,\bar t,x-\bar x)$ except that for all vertices $w\in W \subset C_{\rho}\cup C_{\bar\rho}$ belonging to $ W$, the term $(t-t_w)^{\gamma_{\hat{w},w}} G^{(k_{\hat{w},w})}_{t-t_w}(x_{\hat{w}}-x_w)$ in the expression \eqref{eq:general-bound} is replaced by $(s-t_w)^{\gamma_{\hat{w},w}} G^{(k_{\hat{w},w})}_{s-t_w}(x_{\hat{w}}-x_w)$.
	
	In this case, since $s\leq t\le 2s$, it is straightforward to obtain
	\begin{equ}[eq:general-bound'] 
		\tilde\I^{\bzF,\bfC}_W(t,s,x-\bar x)
		\lesssim t^{\Lambda(\zF)+ \frac{\theta}{2}}|x-\bar x|^{-\theta}
	\end{equ}
	by repeating the argument presented in the first part that shows \eqref{eq:general-bound} (recall $\Lambda(\zF)=\Lambda(\zF_{\rho})+\Lambda(\zF_{\bar\rho})$). Noticing that $|U|=0$, \eqref{eq:general-bound'} agrees with the claimed bound \eqref{eq:general-bound_diff}.
	
	Next, suppose $|U|=1$. We proceed by considering each of the cases discussed in the proof of \eqref{eq:general-bound} except \hyperref[bcase1]{Base case 1} which is not possible.
	We argue either directly, by reduction to $|U|=0$, or by induction on $|E_{\zF}|$.
	We assume without loss of generality that $U=\{\ell\}\subset C_\rho$.
	
	First, in \hyperref[bcase2]{Base case 2}\label{bcase2'}, where now $v=\ell$,
	it follows from the assumption $s\le t\le 2s$ and the estimate \eqref{eq:double-conv-est_diff.} (applied with $\delta=\gamma_f$, $\alpha=\bar\alpha=(a_e-\theta)/2$, $\gamma=a_e$, and $k=j$, $\bar k=\bar j$) 
	that
	\begin{equs}
		{}\big(\big|(t-t_\ell)^{\gamma_f} {G}_{t-t_\ell}^{(j)}-(s-t_\ell)^{\gamma_f}{G}_{s-t_\ell}^{(j)}\big| * \big|\cdot\big|^{-a_e} * \big|\big(r-t_{\bar v})^{\gamma_{\bar{f}}}{G}_{r-t_{\bar v}}^{(\bar j)}\big| \big)(x-\bar x)
		\\\qquad	\lesssim
		|t-s|^{\tipar}(t-t_\ell)^{\gamma_f+\frac{\theta-a_{e}-2|j|}{4}-\tipar}
		(r-t_{\bar v})^{\gamma_{\bar{f}}+\frac{\theta-a_{e}-2|\bar j|}{4}}|x-\bar x|^{-{\theta}}\;,
	\end{equs} 
	where $r\eqdef s\bone_{\bar v\in W}+t\bone_{\bar v\notin W}$ and, as earlier, ${e}$ is the only element of $\fC$, $f,\bar f$ are the only edges and $\ell,\bar v$ are the only leaves of $ \zF_{\rho},{ \zF_{\bar\rho}}$ respectively.
	Multiplying both sides by $g_\ell(t_\ell)g_{\bar v}(t_{\bar v})$,  using the assumption on $\tipar=\kappa$ given in \eqref{eq:tipar-cond} and condition~\ref{pt:edges} (if the corresponding leaf is non-Dirac) to ensure that the singularities are integrable, integrating over time, and considering $s\le t\le 2s$, we obtain the requested bound \eqref{eq:general-bound_diff} for $\K^{\bzF,\bfC}_{W;U}(t,x-\bar x;s)$.
	
	Consider now \hyperref[case1]{Case 1}\label{case1'} with $v=\ell$. We proceed to bound the integral
	\begin{equ}[eq:case1']
		\int_{D} \mrd t_{\ell}\mrd x_{\ell} \big|(t-t_{\ell})^{\gamma_{\rho,\ell}}G^{(k_{\rho,\ell})}_{t-t_\ell}(x_{\rho}-x_\ell)-(s-t_{\ell})^{\gamma_{\rho,\ell}}G^{(k_{\rho,\ell})}_{s-t_\ell}(x_{\rho}-x_\ell)\big| g_\ell(t_{\ell})\;,
	\end{equ} 
	which is the contribution of the variables $x_{\ell},t_{\ell}$ to the integral~\eqref{eq:general-conv_ts}.
	We have
	\begin{equ}
		\eqref{eq:case1'}\lesssim |t-s|^\kappa \int_{0}^{t_{}} g_\ell(t_{\ell})(t_{}-t_{\ell})^{\gamma_{\rho,\ell}-\frac{|k_{\rho,\ell}|}{2}-\kappa} \mrd t_{\ell}
		\propto |t-s|^\tipar t_{}^{1+\beta_{\rho}+\gamma_{\rho,\ell}-\frac{|k_{ \rho,\ell}|}{2}-\tipar}\;,
	\end{equ}
	where, for the spatial integral, we used \eqref{eq:conv-est} with $f\equiv 1$ and $\gamma=\alpha=0$,
	and for the time integral, we used again, if $\ell$ is non-Dirac, that $\gamma_{\rho,\ell}-|k_{\rho,\ell}|/2-\kappa>-1$ due to~\eqref{eq:tipar-cond} and  $\beta_\ell>-1$ by definition.
	One then defines $\tilde{\bzF}$ as in \hyperref[case1]{Case 1}, and the rest of the argument proceeds completely analogously with the difference that the decoration at $\rho$ changes to $\tilde\beta_{\rho}\eqdef 1+\beta_{\rho}+\beta_\ell+\gamma_{\rho,\ell}-|k_{\rho,\ell}|/2-\tipar$. We thus obtain
	\begin{equ}
	\K^{\bzF,\bfC}_{W;U}(t,\bar t, x-\bar x)\lesssim  |t-s|^\tipar \tilde\I^{\tilde{\bzF},\bfC}_W (t,s, x-\bar x)\;,
	\end{equ}
	and apply \eqref{eq:general-bound'} from the case $|U|=0$ to $(\tilde{\bzF},\bfC)$ to conclude the proof in this case.

	Consider next the subcase of \hyperref[case3]{Case 3}\label{case3'} in which the chosen edge $e=\{u,v\}\in\fC$ contains $\ell$.
	Assume without loss of generality that $u=\ell$.
	The inner integral in the definition of $\K^{\bzF,\bfC}_{W;U}$ corresponding to the leaves $\ell,v$ is then equal to
	\begin{equs}[eq:case3-st_integral]
		\int_{0}^{t}\!\!\int_0^{t_{\hat v}}\!\!
		g_\ell(t_\ell) g_{v}(t_v)
		\int_{(\T^n)^{\{\ell,v\}}} 
		\big|\big((t-t_\ell)^{\gamma_{\rho,\ell}}G^{(k_{\rho,\ell})}_{t-t_\ell}-(s-t_\ell))^{\gamma_{\rho,\ell}}
		G^{(k_{\rho,\ell})}_{s-t_\ell}\big)(x-x_\ell)\big|\\
		\big|G^{(k_{\hat v,v})}_{t_{\hat v}-t_v}(x_{\hat v}-x_v)\big|
		\big|x_\ell-x_v\big|^{-a_e}	\, \mrd x_v \mrd x_\ell \mrd t_v \mrd t_\ell\;.
	\end{equs}
	Using the assumption $s \le t\le 2s$ and the estimate \eqref{eq:double-conv-est_diff.} applied with $\delta=\gamma_{\rho,\ell}$, $\gamma = 2\alpha= 2\bar\alpha= a_e$, and $k=k_{\rho,\ell}, \bar k=k_{\hat v,v}$, 
	the spatial integral is bounded above by a multiple of
	\begin{equ}[eq:case3_inner_est]
		|t-s|^{\tipar} (t-t_\ell)^{-\frac{1}{4}a_e - \frac{1}{2}|k_{\rho,\ell}|-{\tipar}} (t_{\hat v}-t_v)^{-\frac14 a_e - \frac12|k_{\hat v,v}|}\;.
	\end{equ}
	Now, let $({\tilde{\bzF}},\tilde{\bfC})$ be the decorated contracted forest obtained after removing the leaves $\ell,v$, the edges $(\rho,\ell),(\hat v,v)$, and the contracting edge $\{\ell,v\}$, and making the following changes to decorations.
	If $\hat v\neq \rho$,
	\begin{equs}[eq:case3_decor']
		\tilde\beta_{\rho}&\eqdef 1+\beta_\rho+\beta_\ell + \gamma_{\rho,\ell}- |k_{\rho,\ell}|/2 -a_e/4 -\tipar\;,\qquad
		\\
		\tilde\beta_{\hat v}&\eqdef 1+\beta_{\hat v}+\beta_v +\gamma_{\hat v,v} -|k_{\hat v,v}|/2-a_e/4\;,\qquad 
	\end{equs}
	and otherwise,
	\begin{equ}[eq:case3_decor'.]
		\tilde\beta_{\rho}\eqdef 2+\beta_{\rho}+\beta_{\ell}+\beta_v-a_e/2 - |k_{\rho ,\ell}|/2+\gamma_{\rho,\ell} - |k_{\rho,v}|/2+\gamma_{\rho,v}-\tipar\;.\quad
	\end{equ}
	(All other  assignments remain the same.)
	Then, estimating \eqref{eq:case3-st_integral} using \eqref{eq:case3_inner_est} and taking the integrals over time
	(we remark that if $\ell$ and/or $v$ are non-Dirac, then correspondingly by \eqref{eq:tipar-cond} for $\ell$, by \eqref{eq:edges} for $v$, and that $\beta>-1$, 
	the singularities in time are integrable),
	we obtain
	\begin{equ}[eq:case3'_induct]
		\K^{\bzF,\bfC}_{W;U}(t,x-\bar x;s) \lesssim
		 |t-s|^{\tipar} \tilde{\I}^{\tilde{\bzF},\tilde{\bfC}}_{\tilde{W}}(t,s,x-\bar x)\;,
	\end{equ}
	where $\tilde{W}\eqdef W\setminus\{\ell\}$.

	By the argument in \hyperref[case3]{Case 3}, $(\tilde{\bzF},\tilde{\bfC})$ satisfies the conditions~\ref{pt:edges}-\ref{pt:contractions} since $(\bzF,\bfC)$ does.
	Therefore, \eqref{eq:general-bound_diff} in this case follows from \eqref{eq:general-bound'} for the case $|U|=0$ applied to $(\tilde{\bzF},\tilde{\bfC})$,
	the bound \eqref{eq:case3'_induct}, and the fact that $\Lambda(\tilde \zF_{\rho})+\Lambda(\tilde \zF_{\bar\rho})=\Lambda( \zF_{\rho})+\Lambda({ \zF_{\bar\rho}})-\tipar$. (We remark that the upper-bound on $\theta$ does not decrease as argued in \hyperref[case3]{Case 3}.)

	For the remaining cases we prove the claim by induction on the number of edges.
	The two base cases are (i) $ \zF_{\rho}$ and ${ \zF_{\bar\rho}}$ are $I^{(j)}(\ell)$ and $(\{\bar\rho\},\emptyset)$ respectively,
	which follows from \hyperref[case1']{Case 1} above,
	and (ii) \hyperref[bcase2']{Base case 2}, which we already considered above.

	For the inductive step, consider first the remaining subcase of \hyperref[case3]{Case 3}\label{case3''}
	in which the chosen edge $e=\{u,v\}\in\fC$ does not contain $\ell$.
	In this case, defining $(\tilde{\zF},\tilde{\fC})$  and $(\tilde{\bzF},\tilde{\bfC})$ and arguing as in \hyperref[case3]{Case 3}, we arrive at
	\begin{equ}
		\K^{\bzF,\bfC}_{W;U}(t,x-\bar x;s) \lesssim
		 \K^{\tilde{\bzF}, \tilde{\bfC}}_{W;U}(t,x-\bar x;s)\;.
	\end{equ}
	Consequently, \eqref{eq:general-bound_diff} for $(\bzF,\bfC)$ in this case follows from the inductive hypothesis.

	In \hyperref[case2]{Case 2}, where necessarily $(u,v)\neq (\rho,\ell)$, the conclusion follows in exactly the same way as in \hyperref[case2]{Case 2} for \eqref{eq:general-bound}. In fact, the desired estimate can be directly derived by rewriting the argument in \hyperref[case2]{Case 2} (with $\I$ replaced by $\K$ in \eqref{eq:case2-induct_I} accordingly) followed by the application of our inductive hypothesis here.
	
	Finally, the argument for \hyperref[case1]{Case 1} with $(u,v)\neq (\rho,\ell)$ also proceeds analogously to that of \hyperref[case1]{Case 1} for \eqref{eq:general-bound}.
	This completes the proof of \eqref{eq:general-bound_diff} in the case that $|U|=1$.

	To conclude the proof of \eqref{eq:general-bound_diff}, it remains to consider $U=\{\ell,\bar\ell\}$ with $\ell\in C_\rho$ and $\bar\ell\in C_{\bar\rho}$.
	The bound in this case is deduced by almost the same considerations as in the case $|U|=1$, so we only highlight the differences.

	For the base case, where $ \zF_{\rho}$ and ${\zF_{\bar\rho}}$ are respectively $I^{(j)}(\ell)$
	and $I^{(\bar j)}(\bar \ell)$,
	there are two possibilities; either $\fC=\emptyset$ or $\fC\neq\emptyset$.
	For $\fC=\emptyset$,
	applying \eqref{eq:double-conv-est_diff} (with $f\equiv 1$, $\delta=\gamma_f$, $\bar\delta=\gamma_{\bar{f}}$, $\alpha=\bar \alpha=\gamma=0$, $k=j$, $\bar k=\bar j$),
	and recalling $\zeta = \kappa + \bar\kappa$,
	we have
	\begin{equs} 
		\K^{\bzF,\bfC}_{W;U}(t, x-\bar x;s) & =\! 
		\int_{D}\!\mrd r \mrd y\, t^{\beta_{\rho}} \big|\big((t-r)^{\gamma_f}{G}_{t-r}^{(j)}\!\!-(s-r)^{\gamma_f}{G}_{s-r}^{(j)}\big)(x-y)\big|g_{\ell}(r)
		\\& \quad\times \!\!\int_{D}\! \mrd r \mrd y\, t^{\beta_{\bar\rho}} \big|\big((t-r)^{\gamma_{\bar{f}}}{G}_{t-r}^{(\bar j)}\!\!-(s-r)^{\gamma_{\bar f}}{G}_{s-r}^{(\bar j)}\big|\big)(\bar x-y)\big|g_{\bar \ell}(r)\\
		& \lesssim
		|t-s|^\tipar t^{\beta_{\rho}+\beta_{\bar\rho}+2+\beta_v+\beta_{\bar v}+\gamma_f+\gamma_{\bar f}-\frac{|j|+|\bar j|}{2}-\tipar}\;,
	\end{equs}
	as desired.
	For $\fC\neq \emptyset$, we apply \eqref{eq:double-conv-est_diff}
	with $\delta=\gamma_f$, $\bar\delta=\gamma_{\bar{f}}$, $\alpha=\bar\alpha=(a_e-\theta)/2$,  $\gamma=a_e$,  $k=j$, $\bar k=\bar j$
	to obtain, for $z\eqdef x-\bar x$,
	\begin{equs}
		{}
		\big(\big|(t\!-t_\ell)^{\gamma_f}{G}_{t-t_\ell}^{(j)}\!\!-(s\!-t_\ell)^{\gamma_f}{G}_{s-t_\ell}^{(j)}\big|\!*\!  \big|\! \cdot \!\big|^{-a_e}\!\!  *\!\big|(t-t_{\bar \ell})^{\gamma_{\bar{f}}}{G}_{t\!-t_{\bar \ell}}^{(\bar j)}\!\!-(s\!-t_{\bar \ell})^{\gamma_{\bar f}}{G}_{s-t_{\bar \ell}}^{(\bar j)}\big|\big)(z) \\
		 \lesssim |t-s|^{\tipar}
		(t-t_{\ell})^{\gamma_f-\frac{j}{2}+\frac{\theta-a_{e}}{4}-\kappa}
		(t-t_{\bar \ell})^{\gamma_{\bar{f}}-\frac{\bar j}{2}+\frac{\theta-a_{e}}{4}-\bar\kappa}|z|^{-\theta}\,.
	\end{equs}
	The conclusion follows by multiplying this bound by $g_\ell(t_\ell) g_{\bar \ell}(t_{\bar \ell})$ and integrating in $t_\ell,t_{\bar \ell}$.
	(Note that, when either of the $\ell,\bar \ell$ is non-Dirac, due to \eqref{eq:tipar-cond} and $\beta>-1$, the singularities in the corresponding time integrals in both aforementioned cases are integrable. We also remark that we used $s\leq t\leq 2s$ in both inequalities above.)
	
	Next, for $|E_{\zF}|>2$, recall that we argue by removing an edge (possibly a contracting edge) with endpoints $u,v$. 
	We now consider three possible situations: 
    1) $u,v\in U$, 2) $|\{u,v\}\cap U|=1$, and 3) $u,v\notin U$.
	For the first situation, which only arises in \hyperref[case3]{Case 3}, 
	the argument is completely analogous to \hyperref[case3']{Case 3} for $|U|=1$, with the only difference being that one applies  \eqref{eq:double-conv-est_diff} (with the same parameters as in the corresponding case, except that we use both $\kappa,\bar\kappa$ therein) when deleting the contracting edge and its neighbouring edges,
	which implies
	that the spatial integral corresponding to the set of vertices $\{\ell,\bar\ell\}=U$ is bounded above by a multiple of
	\begin{equ}
		|t-s|^{\tipar} (t-t_\ell)^{-\frac{1}{4}a_e - \frac{1}{2}|k_{\rho,\ell}|-\kappa} (t-t_{\bar\ell})^{-\frac14 a_e - \frac12|k_{\bar\rho ,\bar\ell}|-\bar\kappa}
	\end{equ}
	instead of \eqref{eq:case3_inner_est}, and accordingly, the assignments 
	\begin{equs}
		\tilde\beta_{\rho}&\eqdef 1+ \beta_\rho + \beta_\ell + \gamma_{\rho,\ell} - |k_{\rho,\ell}|/2 -a_e/4 -\kappa\;, \qquad 	
		\\
		\qquad\tilde\beta_{\bar\rho}&\eqdef 1+\beta_{\bar\rho} + \beta_{\bar\ell}+ \gamma_{\bar\rho,\bar\ell} - |k_{\bar\rho,\bar\ell}|/2 - a_e/4 -\bar\kappa \;\qquad
	\end{equs}	
	replace those of \eqref{eq:case3_decor'} (and \eqref{eq:case3_decor'.}).
	Analogously to \hyperref[case3']{Case 3} for $|U|=1$, we thus reduce the proof of \eqref{eq:general-bound_diff} to the case $|U|=0$, which we already considered.
	
	Consider now the second situation, which can only occur in \hyperref[case1]{Case 1} (when, say, $u=\rho\neq\bar\ell, v=\ell$)
	and \hyperref[case3]{Case 3} (when, say, $u=\ell, v\neq\bar\ell$). The argument for \hyperref[case3']{Case 3} is the same as what we used to arrive at \eqref{eq:case3'_induct}
	except that we apply \eqref{eq:double-conv-est_diff.}
	to arrive at
	\begin{equ}
		\K^{\bzF,\bfC}_{W;U}(t,x-\bar x;s) \lesssim
		|t-s|^{\kappa} \K^{\tilde{\bzF},\tilde{\bfC}}_{\tilde{W};\tilde{U}}(t,x-\bar x;s)\;,
	\end{equ}
	where $\tilde W \eqdef W\setminus \{\ell\}$ as earlier and $\tilde{U} \eqdef \{\bar \ell\}$.
	Correspondingly, in the change of decorations \eqref{eq:case3_decor'}, we replace $\tipar$ by $\tipar-\bar\kappa=\kappa$.
	We conclude by applying the bound \eqref{eq:general-bound_diff} in the case $|U|=1$ and with $\tipar$ therein replaced by $\bar\kappa$, which we already proved.
	For \hyperref[case1']{Case 1} with $(u, v)=(\rho,\ell)$, the argument is analogous to before 
	except that we reduce the bound to an application of \eqref{eq:general-bound_diff} (with $U=\{\bar\ell\}$) instead of \eqref{eq:general-bound'}.
	
	Lastly, in the third situation (i.e. $u,v\notin U$,
	which can occur in any of the Cases~1-3),
	the conclusion follows by induction on the number of edges as earlier.
\end{proof}

\subsection{Convergence of mollifications}\label{subsec:convergence_mollifiers}

In this subsection, we apply Theorem \ref{thm:general-bounds} to prove Theorem \ref{thm:moment-estimates}.

\begin{definition}\label{def:saturation-safe} 
	Let $(\zF, \fC)$ be a contracted forest (see Definition \ref{def:CT}).
	\begin{itemize}
	\item If $|\fC|=|L_{\zF}|/2$ (i.e. $\fC$ is a pairing of the leaves), then $\fC$ is called a \emph{complete contraction}.
	For instance, Figure \ref{fig:graphs}(b) depicts a complete contraction.

	\item 
	A \emph{saturation} of $(\zF, \fC)$ is a contracted forest $(\sigma,\fC_\sigma)$ in which $\sigma$ is a branch of $\zF$, $\fC_{\sigma}$ is the induced contraction on $\sigma$ as in Definition \ref{def:CT}, and $\fC_{\sigma}$ (as a contraction of $\sigma$) is complete.

	\item 
	We call a contracted forest \emph{safe} if it has no saturations. For instance,
	Figures~\ref{fig:graphs}(a) and \ref{fig:unsafe-deletion} 
	depict safe contracted forests. 
	\end{itemize}
\end{definition} 

For the proof of \theo{thm:moment-estimates}, we recall the Isserlis--Wick theorem (see \cite{janson1997gaussian}, Thm. 1.28):
if $ \xi_1,\dots,\xi_m$ are centred jointly normal random variables,
then
\begin{equ}[eq:Wick]
	\E(\xi_1\cdots \xi_m) = \sum\prod_{k}\E(\xi_{i_k}\xi_{j_k})\,,
\end{equ}
where the sum is over all partitions of $ \{1,\dots, m\} $ into disjoint pairs $ \{i_k,j_k\}$.

We furthermore use the following elementary lemma in the proof of \theo{thm:moment-estimates}.

\begin{lemma}\label{lem:oddness-of-leaves-plus-derivatives}
	Consider $\tau\in\CT$. Let
	$\mcb{k}(\tau)$ be the number of edges in $E_\tau$ labelled with $I'$, and denote $\mcb{l}(\tau)\eqdef|L_\tau|$. Then $\mcb{l}(\tau)+\mcb{k}(\tau)$ is an odd number.
\end{lemma}

\begin{proof}
	We proceed by induction on $\mcb{l}(\tau)$. The base case, $\mcb{l}(\tau)=1$ where $\tau=\<Xi>$, is trivial.
	
	Suppose that the claim holds for any $\sigma$ with $\mcb{l}(\sigma)\in [m]$ and $m\ge 1$, and consider any $\tau$ with $\mcb{l}(\tau)= m+1$. 
	We clearly have
	\begin{equ}
	\mcb{k}(\tau)+\mcb{l}(\tau)=
	\begin{cases}
		1+\sum_{i=1}^2\!\big\{\mcb{k}(\tau_i)+\mcb{l}(\tau_i)\big\} & \text{if } \tau=I(\tau_1)I'(\tau_2) \text{ with } \tau_i\in\CT^{m}_{\<Xi>},
		\\
		\sum_{i=1}^3\!\big\{\mcb{k}(\tau_i)+\mcb{l}(\tau_i)\big\} & \text{else if } \tau= \prod_{i=1}^{3}\!I(\tau_i) \text{ with } \tau_i\in\CT^{m}_{\<Xi>}.
	\end{cases}
	\end{equ}
	Therefore, the oddness of $\mcb{k}(\tau)+\mcb{l}(\tau)$ in both cases follows immediately from the inductive hypothesis. 
\end{proof}

\begin{proof}[of \theo{thm:moment-estimates}]
	By Gaussian hypercontractivity (i.e. the equivalence of moments in a fixed Wiener chaos), we only need to prove~\eqref{eq:Z_eps_diff}-\eqref{eq:Z_time_diff} for $p=2$.
			
	\textit{Proof of~\eqref{eq:Z_eps_diff}}. By telescoping as in the proof of Lemma \ref{lem:covar_X^e-X}, it suffices to show the bound in the case of $\eps \leq 2\bar\e$.
	Consider a nested family $\{L_{i}\}_{i=0}^{\noise{\tau}}$ of subsets of $L_\tau$ with $|L_{i}|=i$. Let $X^{\tau,\e}_{t;i}$ be defined similar to $X^{\tau,\bar\e}_t$ with the only difference being that in the construction of $X^{\tau,\e}_{t;i}$, if a leaf of $\tau$ belongs to $L_{i}$, the corresponding $X^{\bar\e}$ is replaced by $X^\e$. 		
	We write, by an obvious telescoping,
	\begin{equs}
		X^{\tau,\e}_{t}-X_t^{\tau,\bar\e} = \sum_{i=1}^{\noise{\tau}}\{X^{\tau,\e}_{t;i}-X^{\tau,\e}_{t;i-1}\}\;,
	\end{equs}
	and then, by the triangle inequality, 
	\begin{equ}
		\|\scal{X^{\tau,\e}_{t}-X_t^{\tau,\bar\e},\phi_z^\lambda}\|_{L^2} \leq \sum_{i} 	\|\scal{X^{\tau,\e}_{t;i}-X^{\tau,\e}_{t;i-1},\phi_z^\lambda}\|_{L^2}\;.
	\end{equ}
	Hence, it suffices to show, for each $X^{\tau,\e}_{t;i}-X^{\tau,\e}_{t;i-1}$, the covariance bound 
	\begin{equ}
		|\E \big(X^{\tau,\e}_{t;i}-X^{\tau,\e}_{t;i-1}\big)(x)\otimes \big(X^{\tau,\e}_{t;i}-X^{\tau,\e}_{t;i-1}\big)(\bar x) | \lesssim
		t^{-2\delta} |x-\bar x|^{2\beta}|\e-\bar\e|^{\kappa\wedge 1}\;.
	\end{equ}
	Denote $Y^\tau_{t}(x)\eqdef (X^{\tau,\e}_{t;i}-X^{\tau,\e}_{t;i-1})(x)$ for any $(t,x)\in D = (0,1)\times\T^n$, whose construction is similar to 
	$X^\tau_t(x)$ with the difference that in $Y^\tau$ one of the leaves of $\tau$ represents $X^\eps-X^{\bar\eps}$ and the others represent 
	$X^\e$ or $X^{\bar\e}$. 
	We can write $Y^\tau_t(x)$ as a linear combination of a finite number of terms with coefficients of the form
	\begin{equ}[eq:Y^tau]
		Y^{\tau,h,k}_t(x) = \int_{D^{V^{*}_{\tau}}}
		\prod_{v\in  L_{\tau}}  Y^{h_v}(x_{v})\delta_0(t_{v})
		\prod_{u\in  V^{*}_{\tau}} 	\!\!G^{(k_{\hat{u},u})}_{t_{\hat{u}}-t_{u}}(x_{\hat{u}}-x_{u})\, \mrd x_u \mrd t_u\,,
	\end{equ} 
	where $h \in (E^{*})^{L_\tau}$ and $Y^{h_v}(\cdot) \eqdef h_v(Y^v(\cdot))$ with $Y^v$ being the function represented by the leaf $v$
	(either $X^\eps$, $X^{\bar\eps}$, or $X^{\eps}-X^{\bar\eps}$),
	and we denote $(t_\rho,x_\rho) = (t,x)$ for $\rho$ the root of $\tau$. 
	Above, we further have $k\in (\N^n_0)^{E_\tau}$ with $|k_{\hat u,u}| = 1$ if $(\hat{u},u)$ carries the label $I'$ and with
	$ k_{\hat u,u} =0$ if $(\hat{u},u)$ carries the label $I$. 

	Consider now the forest $\zF=\{\tau,\bar\tau\}$, where $\bar\tau$ is a rooted tree isomorphic to $\tau$, by which we mean that there is a root- and label-preserving graph isomorphism between the two trees (we also regard functions $Y^v\in \{X^\e, X^{\bar\e}, X^\e - X^{\bar\e}\}$ as labels for the leaves here).
	Denote by $\bar\rho$ the root of $\bar\tau$ and let $t_{\bar\rho} = t$ and $x_{\bar\rho} = \bar x$.
	In view of \eqref{eq:Y^tau} and by Fubini's theorem and \eqref{eq:Wick} (note that since $X$ is a centred Gaussian field by assumption, $\{Y^{h_v}(x_v)\}_{v\in L_{\tau}}$ are centred jointly normal random variables), we can then write
	\begin{equs}[eq:covar-Y^tau-Wick]
		\E\big[Y_{t}^{\tau,h,k}(x) Y_{t}^{\bar\tau,h,k}(\bar x)\big]= 
		\sum_{\mfC} \int_{D^{V^{*}_{\zF}}}\!
		\prod_{\{u,v\}\in \fC} \!\!\! \E\big[Y^{h_u}(x_{u})Y^{h_v}(x_{v})\big]\delta_0(t_{u})\delta_0(t_{v})\\
		\prod_{w\in  V^{*}_{\zF}} 	\!G^{(k_{\hat{w},w})}_{t_{\hat{w}}-t_{w}}(x_{\hat{w}}-x_{w}) \,\mrd x_w \mrd t_w\,,
	\end{equs}
	where the sum runs over all partitions $\fC$ of $ L_{\zF}$ into disjoint pairs and, in line with our earlier notations, we write $V_{\zF}$ for $V_\tau\cup V_{\bar\tau}$, $L_{\zF}$ for $L_\tau\cup L_{\bar\tau}$, and $V^{*}_{\zF}$ for 
	$V_{\zF}\setminus\{\rho,\bar\rho\}$.
	The linear functions $h \in (E^{*})^{L_{\bar\tau}}$ and multi-indexes $k\in(\N^n_0)^{E_{\bar\tau}}$ for $\bar \tau$ are canonically determined by those for $\tau$ and the isomorphism between $\tau$ and $\bar\tau$.
	It thus remains to control the summands appearing on the right-hand side of \eqref{eq:covar-Y^tau-Wick}, each of which corresponds to
	a contracted forest $(\zF,\fC)$ with $\fC$ being a complete contraction of $\zF$ (recall Definitions \ref{def:CT} and \ref{def:saturation-safe}).

	Suppose first $\fC$ is any contraction of $\zF$ such that the contracted forest $(\zF,\fC)$ is not safe. Let then $(\sigma,\fC_{\sigma})$ be a saturation of $(\zF,\fC)$
	and consider 
	\begin{equs}[eq:I^sig]
		I^{\sigma,\fC_\sigma}(t_{\rho_\sigma},x_{\rho_\sigma}) \eqdef \int_{D^{V^{*}_{\sigma}}}
		\prod_{\{u,v\}\in \fC_{\sigma}}\!\!\E\big[Y^{h_u}(x_{u})Y^{h_v}(x_{v})\big]\delta_0(t_{u})\delta_0(t_{v})\\
		\prod_{w\in  V^{*}_{\sigma}} 	\!\!G^{(k_{\hat{w},w})}_{t_{\hat{w}}-t_{w}}(x_{\hat{w}}-x_{w})\, \mrd x_w \mrd t_w\,,
	\end{equs} 
	which is exactly the factor associated with $(\sigma,\fC_\sigma)$ that appears in those terms on the right-hand side of \eqref{eq:covar-Y^tau-Wick} in which the leaves of $\sigma$ are paired according to $\fC_\sigma$. ($\rho_\sigma$ denotes, as usual, the root of $\sigma$.)
	
	Since the covariance function of $X$ and the mollifier $\chi$ are even functions by assumption, 
	we have $\E\big[Y^{h_u}(x_{u})Y^{h_v}(x_{v})\big] = \E\big[Y^{h_u}(-x_{u})Y^{h_v}(-x_{v})\big]$. 
	Furthermore, by stationarity of $X$,
	$I^{\sigma,\fC_\sigma}(t_{\rho_\sigma},x_{\rho_\sigma})=I^{\sigma,\fC_\sigma}(t_{\rho_\sigma},0)$.
	We also note that each $G_s(\cdot)$ is an even function (in space)
	and that the number of $k_e$ in \eqref{eq:I^sig} with $|k_e|=1$ is an odd number—the latter being implied by Lemma~\ref{lem:oddness-of-leaves-plus-derivatives}, given that $\sigma$ has an even number of leaves. 
	It follows that the integrand defining $I^{\sigma,\fC_\sigma}(t_{\rho_\sigma},0)$ as in \eqref{eq:I^sig} is an odd function in the integration variable $(x_v)_{v\in V^*_\sigma}$,
	and therefore $I^{\sigma,\fC_\sigma}(t_{\rho_\sigma},x_{\rho_v})=I^{\sigma,\fC_\sigma}(t_{\rho_\sigma},0)=0$. 
	This implies that the summands in \eqref{eq:covar-Y^tau-Wick} corresponding to non-safe contractions vanish.

	Fix then a complete contraction $\fC$ such that $(\zF,\fC)$ is safe.
	Since we have only for one of the leaves of $\tau$ the occurrence of $X^\eps-X^{\bar\eps}$ and for others $X^\eps$ or $X^{\bar\eps}$, any complete contraction includes either one edge $\{u,\bar u\}$ connecting the two leaves associated with $X^\e-X^{\bar\e}$, or two contracting edges $f=\{u,v\}$ and $\bar f=\{\bar u,w\}$ incident to these two leaves.
	Define, in the former case, 
	\begin{gather*}
		K_{\{u,\bar u\}}(x_u-x_{\bar u})\eqdef 
		\E\big[\big(X^{\e,h_{u}}(x_u)-X^{\bar\e,h_{u}}(x_u)\big)\big(X^{\e,h_{\bar u}}(x_{\bar u})-X^{\bar\e,h_{\bar u}}(x_{\bar u})\big)\big]\;,
		\\
		K_{\{\ell,\bar \ell\}}(x_\ell-x_{\bar \ell})\eqdef \E\big[X^{\e_\ell,h_{\ell}}(x_{\ell})X^{\e_{\bar\ell},h_{\bar\ell}}(x_{\bar \ell})\big]\;,
		\;\qquad \forall \{\ell, \bar\ell\}\in \fC\setminus\{\{u,\bar u\}\}\;,
\end{gather*}
	and, in the latter case,
	\begin{gather*}
		K_{\{u,v\}}(x_u-x_{v})\eqdef \E\big[\big(X^{\e,h_{u}}(x_u)-X^{\bar\e,h_{u}}(x_u)\big)X^{\e_v,h_{v}}(x_{v})\big]\;,
		\\
		K_{\{\bar u,w\}}(x_{\bar u}-x_{w})\eqdef \E\big[\big(X^{\e,h_{\bar u}}(x_{\bar u})-X^{\bar\e,h_{\bar u}}(x_{\bar u})\big)X^{\e_w,h_{w}}(x_{w})\big]\;,	
		\\
		K_{\{\ell,\bar \ell\}}(x_\ell-x_{\bar \ell})\eqdef \E\big[X^{\e_\ell,h_{\ell}}(x_{\ell})X^{\e_{\bar\ell},h_{\bar\ell}}(x_{\bar \ell})\big]\;,
		\qquad \forall \{\ell, \bar\ell\}\in \fC(L_\zF\setminus\{u,\bar u, v, w\})\;,
	\end{gather*}
	where $\e_{\bcdot}\in\{\e,\bar\e\}$ and $X^{\e_{\bcdot},h_{\bcdot}}(\cdot)\eqdef h_{\bcdot}(X^{\e_{\bcdot}}(\cdot))$.
	It then follows from \eqref{eq:cov_Psi-Psi^e} (which uses $\eps\leq 2\bar\eps$ and \assu{ass:GF} with $|k|\leq 1$),
	that in the former case above,
	\begin{equ}
		|K_{\{u,\bar u\}}(\cdot)|\lesssim|\!\cdot\!|^{2-d-(\kappa\wedge 1)}
		|\e-\bar\e|^{\kappa\wedge 1}\;,
	\end{equ}
	and in the latter case, 
	\begin{equ}
		|K_{f}(\cdot)|\lesssim|\!\cdot\!|^{2-d-\frac{\kappa\wedge 1}{2}}|\e-\bar\e|^{\frac{\kappa\wedge 1}{2}}
		\;,
		\qquad |K_{\bar f}(\cdot)|\lesssim|\!\cdot\!|^{2-d-\frac{\kappa\wedge 1}{2}}|\e-\bar\e|^{\frac{\kappa\wedge 1}{2}}\;.
	\end{equ} 
	In both cases, on the other hand,
 	$|K_e(\cdot)|\lesssim|\!\cdot\!|^{2-d}$ 
	for the other edges in the contraction. 
	We therefore have,
	for each edge $e \in\fC$,
	$|K_{e}(\cdot)|\lesssim|\!\cdot\!|^{a_e}|\e-\bar\e|^{2-d-a_e}$, where  $a_e\in\{d-2, d-2+(\kappa\wedge 1), d-2+(\kappa\wedge 1)/2\}$ is identified by the aforementioned bound for the kernel $K_e$.

	Consider now the decorated graph $(\bzF,\bfC)$ consisting of $(\zF, \fC)$ and the following data:
	\begin{equ}
	\beta_\ell=-1\;,\quad \forall \ell\in L_{\zF}\;, \qquad
	\beta_u=0\;,\quad \forall u\in I_{\zF}\;, \qquad
	\gamma_e = 0 \;,\quad \forall e\in E_{\zF}\;,
	\end{equ}
    $k_{\hat w, w}\in\N^n_0$ for each edge $(\hat w,w)$ is as in a summand in \eqref{eq:covar-Y^tau-Wick}, and
	$\{a_e\}_{e\in \fC}$ is as specified above.
	Then $(\bzF,\bfC)$ is a decorated contracted forest in the sense of \defi{def:DCF-I-K}\ref{pt:DCF}.
	Furthermore, using the definitions of $|\<Xi>|$ and $\Lambda(\<Xi>)$ and the recursive formulas \eqref{eq:homogeneity} and \eqref{eq:Lambda_induct}, we have, for any branch $\sigma$ of $\zF$,
	\begin{equ}[eq:Lambda_T'^eps]
		2\Lambda(\sigma) - |\sigma|\in\{0,-(\kappa\wedge 1)/4,-(\kappa\wedge 1)/2\}\;
	\end{equ}
	and, in particular,
	\begin{equ}
	\Lambda(\bar\tau)=\Lambda(\tau)=|\tau|/2 -(\kappa\wedge 1)/4\;.
	\end{equ}
	Denote by $C^{\tau,\fC}(t, x-\bar x)$ the corresponding summand in \eqref{eq:covar-Y^tau-Wick}.
	Recalling \eqref{eq:general-conv}, it follows from the bounds above on $K_e$ that
	\begin{equ}[eq:C_first_bound]
		C^{\tau,\fC}(t, x-\bar x)\lesssim |\e-\bar \e|^{\kappa\wedge 1}\I^{\bzF,\bfC}(t,t, x-\bar x)\;.
	\end{equ}
	Therefore, we obtain from \eqref{eq:general-bound} (with $\theta =-2\beta \in [0,d-2]$ therein) that
	\begin{equ}[eq:I^YM_P-est]
		C^{\tau, \fC}(t, x-\bar x) \lesssim
		t^{|\tau|-(\kappa\wedge 1)/2 - \beta}
		|x-\bar x|^{2\beta}|\e-\bar\e|^{\kappa\wedge 1}
		\leq
		t^{-2\delta}|x-\bar x|^{2\beta}|\e-\bar\e|^{\kappa\wedge 1}
	\end{equ}
	once we verify conditions~\ref{pt:edges}-\ref{pt:contractions} for $(\bzF,\bfC)$, which we do next.
	\begin{enumerate}[label=\arabic*.,leftmargin=13pt]
	\item
	\ref{pt:edges} holds since, for a non-root vertex $v$,
	\eqref{eq:edges} is equivalent to $-(d-2+\tilde{\kappa})/4- |k_{\hat v,v}|/2>-1$, where $\tilde{\kappa}\in\{0,(\kappa\wedge 1)/2,\kappa\wedge 1\}$ depending on the position of the leaf labelled $X^{\e}-X^{\bar\e}$, which holds because $d<4$, $|k_{\hat v,v}|\leq 1$, and $\kappa\in[0,4-d)$.

	\item
	\ref{pt:branches} holds
	since, for $\<Xi>\neq\sigma\in\CT$,
	\begin{equs}
		\Lambda(\sigma)+\min_{e\in \fC(V_{\sigma})}\!\! a_e/4 &
		\ge \noise{\sigma}(2-d)/4+\noise{\sigma}/2-3/2-(\kappa\wedge 1)/4 +(d-2)/4\\&
		\ge \dd-1/2-(\kappa\wedge 1)/4+(d-2)/4 \\&
		\ge -d/4-\kappa/4 > -1\,,
	\end{equs}
 	where we used $\min_{e\in \fC(V_{\sigma})} a_e=d-2$, \eqref{eq:Lambda_T'^eps}, and that $|\sigma|=\noise{\sigma}(2-d)/2+\noise{\sigma}-3$ by \rem{rem:homogeneity} in the first line, the fact that $|\!\cdot\!|$ is increasing in $\noise{\cdot}$ (since $d<4$) and that $\noise{\sigma}\geq 2$ in the second line, and the assumption $\kappa<4-d$ in the end.

	\item
	\ref{pt:contractions} holds because $n\ge 2$, $d\in(2,4)$, and $\kappa\in[0,4-d)$, which imply that $d-2$ and $d-2+\kappa$ are in $(0,2)\subset (0,n)$.
	\end{enumerate}
	Therefore Theorem \ref{thm:general-bounds} and \eqref{eq:C_first_bound} imply the bound \eqref{eq:I^YM_P-est}. It remains to integrate this bound against $|\varphi^\lambda_{z}(x)\varphi^\lambda_{z}(\bar x)|$ to obtain the desired estimate \eqref{eq:Z_eps_diff}.

	\textit{Proof of \eqref{eq:Z_time_diff}}.
	By telescoping as in the proof of Lemma \ref{lem:covar_X^e-X},
	it suffices to consider $t\leq 2s$.
	We assume that $ \tau= I(\tau_1)I'(\tau_2)$ 
	where $\tau_i\in\CT$
	(the argument for the case $ \tau= \prod_{i=1}^3I(\tau_i)$ is completely analogous).
	Denoting $Z^{\e}_{t,s} = t^{\delta}X_t^{\tau,\e}-s^{\delta}X^{\tau,\e}_{s}$ and $\CP_{t,s}=\CP_{t}-\CP_{s}$, we write
	\begin{equ} [eq:time-diff_split]
		Z^{\e}_{t,s}\!
		= (t^\delta-s^\delta) X^{\tau,\e}_{t}\!
		+  s^\delta B(\CP_{t} \star X^{\tau_1,\e}\!,D\CP_{t,s}\star X^{\tau_2,\e})
		+  s^\delta B(\CP_{t,s}\star X^{\tau_1,\e}\!,D\CP_{s} \star X^{\tau_2,\e})\,.
	\end{equ}
	We next bound the ${L^2}$-norm of each term in \eqref{eq:time-diff_split} integrated against $\phi^\lambda_z$.
		
	For the first term in \eqref{eq:time-diff_split}, by \eqref{eq:Z_eps_diff} (applied with $\kappa=0$ therein), we have
	\begin{equ}
		(t^\delta-s^\delta)
		\|\scal{X^{\tau,\e}_{t},\phi_z^\lambda}\|_{L^2}
		\lesssim
		(t^{\delta}-s^{\delta}) t^{|\tau|/2-\beta/2}\lambda^\beta \lesssim
		|t-s|^{\kappa/4}\lambda^\beta\;,
	\end{equ}
	where we used $t\le 2s$ and that $\delta= -|\tau|/2+\beta/2+\kappa/4$ in the last inequality.

	For the second term on the right-hand side of \eqref{eq:time-diff_split} (the third can be estimated in exactly the same way), arguing as in the proof of \eqref{eq:Z_eps_diff}, we estimate the covariance
	\begin{equ}[eq:diff_split2]
		\E\big[B(\CP_{t} \star X^{\tau_1,\e}(x),D\CP_{t,s} \star X^{\tau_2,\e}(x))\otimes B(\CP_{t}\star X^{\tau_1,\e}(\bar x),D\CP_{t,s} \star X^{\tau_2,\e}(\bar x))\big]\;
	\end{equ}
	by writing $B(\CP_{t} \star X^{\tau_1,\e}(x),D\CP_{t,s} \star X^{\tau_2,\e}(x))$ as linear combinations with coefficients of the form
	\begin{equ}
		Y^{\tau,h,k}_t(x) = \int_{D^{V^{*}_{\tau}}}
		\prod_{v\in  L_{\tau}}  Y^{h_v}(x_{v})\delta_0(t_{v})
		\prod_{u\in  V^{*}_{\tau}} 	\!\!\tilde{G}^{(k_{\hat{u},u})}_{t_{\hat{u}}-t_{u}}(x_{\hat{u}}-x_{u}) \,\mrd x_u \mrd t_u\,,
	\end{equ} 
	where $k_e \in \N_0^n$ are subject to the same conditions as in \eqref{eq:Y^tau} and
	where $\tilde{G}^{}_{t_{\hat{u}}-t_{u}}(x_{\hat{u}}-x_{u})=\big(G^{}_{t-t_{u}}-G^{}_{s-t_{u}}\big)(x-x_{u})$
	if $\hat u=\rho$ and $|k_{\hat{u},u}|=1$, and  
	$\tilde{G}= G$ 
	otherwise. 
	For an isomorphic tree $\bar{\tau}$, we define $Y_{\bar t}^{\bar\tau,h,k}(\bar x)$ in an analogous manner.
	As in the proof of \eqref{eq:Z_eps_diff},
	by \eqref{eq:Wick},
	to estimate \eqref{eq:diff_split2} it suffices to estimate
	\begin{equ}
		\sum_{\mfC} \int_{D^{V^{*}_{\zF}}}\!
		\prod_{\{u,v\}\in \fC} \!\!\! \E\big[Y^{h_u}(x_{u})Y^{h_v}(x_{v})\big]\delta_0(t_{u})\delta_0(t_{v})
		\prod_{w\in  V^{*}_{\zF}} 	\!\tilde{G}^{(k_{\hat{w},w})}_{t_{\hat{w}}-t_{w}}(x_{\hat{w}}-x_{w}) \,\mrd x_w \mrd t_w\,.
	\end{equ}
	Again as in the proof of \eqref{eq:Z_eps_diff}, the summands for which $(\zF,\fC)$ is not safe vanish.
	We henceforth fix a complete contraction $\fC$ such that $(\zF,\fC)$ is safe.

	Note that all the leaves of~$\zF=\{\tau,\bar\tau\}$ represent the occurrence of $X^\e$ so that, by \assu{ass:GF} with $|k|=0$, $|\E\big[Y^{h_u}(x_{u})Y^{h_v}(x_{v})\big]|\lesssim|x_u-x_v|^{2-d}$ for every contracting edge $\{u,v\}$. It therefore suffices to estimate
	\begin{equ}[eq:covar-Y^tau-Wick']
	\int_{D^{V^{*}_{\zF}}}\!\prod_{\{u,v\}\in \fC} \!\!\! |x_u-x_v|^{2-d}\delta_0(t_{u})\delta_0(t_{v})
		\prod_{w\in  V^{*}_{\zF}} 
		\!\tilde{G}^{(k_{\hat{w},w})}_{t_{\hat{w}}-t_{w}}(x_{\hat{w}}-x_{w}) \,\mrd x_w \mrd t_w
		\;,
	\end{equ}
	which is a special case of \eqref{eq:general-conv_ts},
 	where the assignments in the definition of $(\bzF,\bfC)$ are the same as in the proof of \eqref{eq:Z_eps_diff} except with $a_e=d-2$ for all of the contracting edges.
	We remark that $(\bzF,\bfC)$ satisfies assumptions~\ref{pt:edges}-\ref{pt:contractions} of \theo{thm:general-bounds} (to see this, take $\kappa=0$ in 1,2,3 above, where we verified conditions \ref{pt:edges}-\ref{pt:contractions} in the proof of \eqref{eq:Z_eps_diff} for any $\kappa\in[0,4-d)$ therein) and therefore, by \eqref{eq:general-bound_diff} applied with $\theta=-2\beta\in [0, \max_{e\in \fC} a_e]$ and any $\tipar \in [0,1-\max_{e\in \fC} a_e/2]$,
	\begin{equ}[eq:quad_split2]
		\eqref{eq:covar-Y^tau-Wick'}\lesssim |t-s|^{\tipar} 	t^{2\Lambda(\tau)-\beta-\tipar} |x-\bar x|^{2\beta}\;.
	\end{equ} 
	Since $a_e=d-2$ for all $e\in \fC$, we have $\min_{e\in \fC}a_e=\max_{e\in \fC}a_e=d-2$ and $2\Lambda(\tau)=|\tau|$, and hence, from \eqref{eq:quad_split2} and the assumptions on $d$ and $\kappa$,
	\begin{equ}
		\eqref{eq:covar-Y^tau-Wick'}\lesssim |t-s|^{\kappa/2} t^{|\tau|-\beta-\kappa/2} 	|x-\bar x|^{2\beta}\;,
	\end{equ} 
	for any $\beta\in[\dd,0]$ and $\kappa\in[0,4-d)$.
	Therefore \eqref{eq:diff_split2} satisfies the same bound and \eqref{eq:Z_time_diff} follows by integrating the bound against $|\varphi^\lambda_{z}(x)\varphi^\lambda_{z}(\bar x)|$ (notice the factor $s^\delta$ in \eqref{eq:time-diff_split} and that $\delta = -|\tau|/2+\beta/2+\kappa/4$ and $s\le t\le 2s$).
\end{proof}

\section{Proof of \theo{thm:main}}
\label{sec:main_proof}

To finally prove \theo{thm:main}, we combine our deterministic result of \theo{thm:lwp} with the probabilistic estimates of \cor{corr:convergence-of-mollifications}.
	
\begin{proof} [of \theo{thm:main}]
	Set $N = \floor{2/(4-d)}$ and, for $\kappa>0$ sufficiently small,
	define $\omega_{\<Xi>}\eqdef\dd-\kappa$
	and, for all $\tau\in \CT^N$,
	\begin{equ}
		\beta_\tau \eqdef (2-d)/2\;,\qquad 	
		\delta_\tau \eqdef -|\tau|/2+\beta_\tau/2 +\kappa/2\;,  \qquad
		\omega_\tau \eqdef \beta_\tau-2\delta_\tau+2\;.
	\end{equ}
	Note that, by Remark \ref{rem:homogeneity},
	$\omega_{\tau}=\noise{\tau}(4-d)/2-1-\kappa$ for all $\tau\in\CT^N_{\<Xi>}$.
	We claim that, for $\kappa$ sufficiently small, $\omega_{\<Xi>}, \bbeta,\bdelta$ satisfy condition~\eqref{eq:CI}.

	To prove the claim, remark first that since $d\in(2,4)$, we have $\beta_\tau\in(-1,0)$  for all $\tau\in\CT^N$. 
	Furthermore, for any $ \tau\in \CT^N$, we have $ \delta_\tau<1$ as it is equivalent to
	\begin{equs}
		\noise{\tau}(d-2)/4-(\noise{\tau}-3)/2+(2-d)/4+\kappa/2 < 1 \, 	\Leftrightarrow( \noise{\tau}-1)(2-d/2) > \kappa
	\end{equs}		
	(and the last inequality holds for sufficiently small $\kappa$ because $\noise{\tau}>1$ and $d<4$).
	Next, since $\noise{\tau}\le 2/(4-d)$ and $\kappa>0$, it is clear that $\noise{\tau}(4-d)/2<1+\kappa$, which shows $\omega_\tau < 0$. 

	It remains to verify that $\omega,\alpha,\gamma/2 >-1$ for $\omega,\alpha,\gamma$ as in \eqref{eq:CI}.
	Observe that $\omega_{\tau}=\noise{\tau}(4-d)/2-1-\kappa$ is increasing in $\noise{\tau}$ (since $d<4$).
	Therefore, $\omega = \omega_{\<Xi>} = \dd -\kappa > -1$, where the inequality holds for $\kappa$ sufficiently small.
	
	To check $\alpha,\gamma/2 >-1$, 
	since $\omega_{\tau}$ is an affine function of $\noise{\tau}$, we have
	\begin{equs}
		\min&\{\omega_{\tau_1}+\omega_{\tau_2} \,:\, \tau_1,\tau_2 \in\CT_{\<Xi>}^N\,,\; \noise{\tau_1}+\noise{\tau_2}>N\}=N(4-d)/2-d/2-2\kappa\;,\\
		 \min&\Big\{\sum_{i=1}^3\omega_{\tau_i} \,:\, \tau_i \in\CT_{\<Xi>}^N\,,\; \sum_{i=1}^3\noise{\tau_i}>N\Big\}=(N-1)(4-d)/2+1-d-3\kappa\;.
	\end{equs}
	By the above identities, $\alpha>-1$ and $\gamma>-2$ are equivalent to $N>(4\kappa+d-2)/(4-d)$ and $N>(6\kappa+d-2)/(4-d)$ respectively, which, considering $N = \floor{2/(4-d)}$, hold for sufficiently small $\kappa$.

	In conclusion, the above shows that $\omega_{\<Xi>},\bbeta,\bdelta$ satisfy condition~\eqref{eq:CI}.
	Define $\init\equiv\init_{\omega_{\<Xi>},\bbeta,\bdelta}$ and $\Theta\equiv\Theta_{\omega_{\<Xi>},\bbeta,\bdelta}$ as in Definition \ref{def:CI}. 
	Since $\beta_\tau>-1$ and $\delta_\tau<1$ for all $\tau \in \CT^N$,
	Proposition \ref{prop:closable_graph} implies that $(\CI,\Theta)$ is continuously embedded in $\CC^{\omega_{\<Xi>}}$.

	Next, the assumptions of \cor{corr:convergence-of-mollifications} hold with our choice of parameters, and therefore $\lim_{\eps\downarrow0} X^\eps = X$ in $(\init,\Theta)$ a.s. and in $L^p(\P)$ for all $p \in [1,\infty)$ (since the only possible limit point of $X^\eps$ is $X$).
	In particular, $X$ a.s. takes values in $\CI$.

	Take now any $ \theta> 0$ such that $\omega/2-\theta,\alpha/2-\theta,(\gamma/2-\theta)/2 >-1/2$. Since $X\in \init$ a.s., by \theo{thm:lwp}, there exists $\kappa>0$ such that, for $T^{\kappa} \asymp (1+\Theta(X))^{-2}$,
	there exists a unique solution $A$ to \eqref{eq:A_eq} on $(0,T)\times\T^n$ given by
	\begin{equ}
		A_t = R_t(X)+\CP_t \star \CS_{\<Xi>}^N X\;,\quad R(X) \in \CB_T
	\end{equ}
	and that has initial condition $X$ in the sense of Theorem \ref{thm:lwp}. Likewise for $X^\eps$.

	To prove the convergence in \eqref{eq:convergence}, we write using the triangle inequality
	\begin{equ}[eq:A^eps-A]
		|A_t^\e - A_t|_{\CC^\eta} \le |R_t(X^\e)-R_t(X)|_{\CC^\eta} + |\CP_t \star \CS_{\<Xi>}^N X^\e-\CP_t \star\CS_{\<Xi>}^N X|_{\CC^\eta}\;.
	\end{equ}
	For the first term on the right-hand side of \eqref{eq:A^eps-A}, we know from \theo{thm:lwp} that
	\begin{equ}[eq:R]
		|R_t(X^\e)-R_t(X)|_{\CC^\eta}\leq |R_t(X^\e)-R_t(X)|_{\infty} \leq t^\theta \Theta(X^\e,X)\;.
	\end{equ}
	For the second term, since $\eta<\dd$, we can ensure $\beta_\tau, \omega>\eta$ for all $\tau\in\CT^N$ by taking $\kappa$ sufficiently small.
	Setting $\delta = \max_{\tau\in\CT^N}\delta_\tau$ and using the triangle inequality and the definition of $\Theta\equiv \Theta_{\omega_{\<Xi>},\bbeta,\bdelta}$, we have
	\begin{equ}
		|\CS_{t}^N X^\e-\CS_{t}^N X|_{\CC^\eta}\lesssim t^{-\delta}\Theta( X^\e, X)\;,
	\end{equ}
	from which together with the fact that $\delta<1$, we obtain
	\begin{equ}
		|\CP_t \star \CS^N X^\e-\CP_t \star \CS^N X|_{\CC^\eta}\lesssim t^{1-\delta}\Theta(X^\e, X)\;.
	\end{equ}
	We also know, again from the definition of $\Theta$, that
	\begin{equ}
        |X^\e- X|_{\CC^{\omega_{\<Xi>}}}\le \Theta( X^\e, X)\;.
	\end{equ}
	Therefore, since $\eta<\omega_{\<Xi>}$,
	\begin{equ}
		|\CP_t \star \CS_{\<Xi>}^N X^\e-\CP_t \star \CS_{\<Xi>}^N X|_{\CC^\eta}\lesssim \Theta( X^\e, X)\;,
	\end{equ}
	which combined with \eqref{eq:A^eps-A} and \eqref{eq:R} implies
	\begin{equ}
		\sup_{\!t\in[0,T]}|A^\e_t - A_t|_{\CC^\eta} \lesssim \Theta( X^\e, X)\;.
	\end{equ}
	We similarly obtain, using heat flow estimates,
	\begin{equ}
		\sup_{\!t\in(0,T)} t^{-\eta/2}|A^\e_t - A_t|_\infty\lesssim \Theta( X^\e, X)\;.
	\end{equ}
	The claimed a.s.- and $L^p$-convergence, $p\in [1,\infty)$, in \eqref{eq:convergence}
	now follow from the last two inequalities and that $\Theta(X^\e, X)\to 0$ a.s. and in $L^p$.
	The facts that $A$ is smooth on $(0,T)\times \T^n$ and $\lim_{t\downarrow 0} |A_t-X|_{\CC^\eta} = 0$ a.s. are direct consequences of \theo{thm:lwp} since we can take $\omega>\eta$.

	Finally, for the second assertion in \eqref{eq:convergence},  recall that we have $T^{-1}\asymp (1+\Theta(X))^{2/\kappa}$ for some $\kappa>0$.  Consequently, the boundedness of $\E [T^{-p}]$ for all $p\ge 1$ is immediately deduced from the boundedness of the moments of $\Theta(X)^{2/\kappa}$ for the Gaussian field $X$, which concludes the proof.
\end{proof}

\appendix

\section{Gaussian free field in fractional dimensions} \label{app:GFF}

For $n\geq 2$ and $d\in(2,4)$, define the Gaussian free field (GFF) $X$ on $\T^n$ as the Gaussian distribution with covariance $C = (-\Delta)^{\frac{d-n-2}{2}}$,
where we treat $\Delta$ as an operator on zero-mean functions.
The next proposition shows that $X$ satisfies Assumption~\ref{ass:GF}.
Although the proof is standard (and follows ideas from \cite{stinga2019user}), we give details for completeness.

\begin{proposition}
	Let $n\geq 2$, $d\in(2,4)$, and $C = (-\Delta)^{\frac{d-n-2}{2}}$.
	Then for any $k\in\N^n_0$, there exists $K>0$ such that $|C^{(k)}(x)|\leq K |x|^{2-d-|k|}$ for all $x\in\T^n\setminus\{0\}$.
\end{proposition}

\begin{proof}
We first consider $k=0$.
	Recall that $\lambda^{-\beta} = \Gamma(\beta)^{-1}\int_0^\infty \mre^{-t\lambda}t^{-1+\beta}\mrd t$ for $\lambda,\beta >0$.
	Since $\mre^{t\Delta}-1$, as a convolution operator, has the same spectrum and eigenfunctions as $\mre^{t\Delta}$ except that the eigenvalue of $1$ becomes $0$ (instead of $1$ as for $\mre^{t\Delta}$),
	we obtain the identity
	\begin{equ}[eq:Delta_heat]
	(-\Delta)^{-\beta}(x) = \Gamma(\beta)^{-1}\int_0^\infty \big(\mre^{t\Delta}(x) - 1\big) t^{-1+\beta}\mrd t\;.
	\end{equ}
	The integral converges absolutely since, for $t\geq 1$,
	\begin{equ}
		|\mre^{t\Delta}(x) -1 | \lesssim \Big|\sum_{m\in\Z^n\setminus\{0\}} \mre^{-|m|^2 t}\mre^{2\pi\mri \scal{m,x}}
		\Big| \lesssim \mre^{-t}\;.
	\end{equ}
	In particular, the integral $\int_1^\infty \big(\mre^{t\Delta}(x) - 1\big) t^{-1+\beta}\mrd t$ is of order $1$ for all $x \in \T^n$.

	We now observe that, for $t\in (0,1)$ and $|x|< \frac14$,
	\begin{equ}
		\mre^{t\Delta}(x) = (4\pi t)^{-n/2}\mre^{-|x|^2/(4t)} + f(t,x)
	\end{equ}
	where $f$ is a smooth function,
	and for $|x|\geq \frac14$, $\mre^{t\Delta}(x)$ is smooth in $t,x$.
	Therefore, for $0<2\beta < n$, 
	\begin{equ}
		\int_0^1 \big(\mre^{t\Delta}(x) - 1\big) t^{-1+\beta}\mrd t 
		\asymp \int_1^\infty s^{n/2} \mre^{-s|x|^2/4} s^{-\beta -1} \mrd s +O(1)
		\asymp |x|^{2\beta - n}\;.
	\end{equ}
	Taking $2\beta = n+2-d \in (0,n)$, we obtain $(-\Delta)^{-\beta} \asymp |x|^{2-d}$,
	which proves the case $k=0$.
	For $|k|>0$, we differentiate~\eqref{eq:Delta_heat}
	and use a similar argument to obtain
	\begin{equs}
		|D^k (-\Delta)^{-\beta}(x)|
		&\asymp O(1)+ \Big| \int_0^1 D^k \mre^{-|x|^2/(4t)} t^{-n/2} t^{-1+\beta} \mrd t \Big| \\
		\\
		&\lesssim  \int_0^1 t^{-|k|/2}\mre^{-|x|^2/(8t)} t^{-n/2-1+\beta} \mrd t \\
		&\asymp \int_1^\infty \mre^{-s|x|^2/8} s^{n/2-\beta+|k|/2-1} \mrd s
		\asymp |x|^{-|k|+2\beta - n}\;,
	\end{equs}
	where we used \eqref{eq:G^k-bound} in the second line.
	We conclude with $2\beta = n+2-d$.
\end{proof}

	\endappendix
	
	\paragraph{Acknowledgements}
	We would like to thank the anonymous referees for their constructive and detailed comments that helped to improve the article.
I.C. acknowledges support by the Engineering and Physical Sciences Research Council via the New Investigator Award EP/X015688/1
and by the European Research Council via the Starting Grant SQGT 101116964.

	\bibliographystyle{Martin}
\bibliography{./refs}

\def\cprime{$'$} \def\polhk#1{\setbox0=\hbox{#1}{\ooalign{\hidewidth \lower1.5ex\hbox{`}\hidewidth\crcr\unhbox0}}}
\begin{thebibliography}{BDNY24}
\def\myhref#1#2{\href{#2}{\nolinkurl{#1}}}

\bibitem[BC23]{BC23}
\textsc{B.~{Bringmann}} and \textsc{S.~{Cao}}.
\newblock {A para-controlled approach to the stochastic Yang-Mills equation in two dimensions}.
\newblock \emph{arXiv e-prints} (2023).
\newblock To appear in \textit{Memoirs of the AMS}.
\newblock \myhref{arXiv:2305.07197}{https://arxiv.org/abs/2305.07197}.

\bibitem[BC24]{BC24}
\textsc{B.~{Bringmann}} and \textsc{S.~{Cao}}.
\newblock {Global well-posedness of the stochastic Abelian-Higgs equations in two dimensions}.
\newblock \emph{arXiv e-prints} (2024).
\newblock \myhref{arXiv:2403.16878}{https://arxiv.org/abs/2403.16878}.

\bibitem[BCCH21]{BCCH21}
\textsc{Y.~Bruned}, \textsc{A.~Chandra}, \textsc{I.~Chevyrev}, and \textsc{M.~Hairer}.
\newblock Renormalising {SPDE}s in regularity structures.
\newblock \emph{J. Eur. Math. Soc. (JEMS)} \textbf{23}, no.~3, (2021), 869--947.
\newblock \myhref{doi:10.4171/jems/1025}{https://dx.doi.org/10.4171/jems/1025}.

\bibitem[BCD11]{BookChemin}
\textsc{H.~Bahouri}, \textsc{J.-Y. Chemin}, and \textsc{R.~Danchin}.
\newblock \emph{Fourier analysis and nonlinear partial differential equations}, vol. 343 of \emph{Grundlehren der mathematischen Wissenschaften [Fundamental Principles of Mathematical Sciences]}.
\newblock Springer, Heidelberg, 2011,  xvi+523.
\newblock \myhref{doi:10.1007/978-3-642-16830-7}{https://dx.doi.org/10.1007/978-3-642-16830-7}.

\bibitem[BDNY24]{BDNY_24_Gibbs}
\textsc{B.~Bringmann}, \textsc{Y.~Deng}, \textsc{A.~R. Nahmod}, and \textsc{H.~Yue}.
\newblock Invariant {G}ibbs measures for the three dimensional cubic nonlinear wave equation.
\newblock \emph{Invent. Math.} \textbf{236}, no.~3, (2024), 1133--1411.
\newblock \myhref{doi:10.1007/s00222-024-01254-4}{https://dx.doi.org/10.1007/s00222-024-01254-4}.

\bibitem[BOP19]{BOP19}
\textsc{A.~B\'{e}nyi}, \textsc{T.~Oh}, and \textsc{O.~Pocovnicu}.
\newblock On the probabilistic {C}auchy theory for nonlinear dispersive {PDE}s.
\newblock In \emph{Landscapes of time-frequency analysis}, Appl. Numer. Harmon. Anal.,  1--32. Birkh\"{a}user/Springer, Cham, 2019.
\newblock \myhref{doi:10.1007/978-3-030-05210-2_1}{https://dx.doi.org/10.1007/978-3-030-05210-2_1}.

\bibitem[Bou94]{Bourgain94}
\textsc{J.~Bourgain}.
\newblock Periodic nonlinear {S}chr\"{o}dinger equation and invariant measures.
\newblock \emph{Comm. Math. Phys.} \textbf{166}, no.~1, (1994), 1--26.
\newblock \myhref{doi:10.1007/BF02099299}{https://dx.doi.org/10.1007/BF02099299}.

\bibitem[Bri24]{Bringmann_24_Hartree}
\textsc{B.~Bringmann}.
\newblock Invariant {G}ibbs measures for the three-dimensional wave equation with a {H}artree nonlinearity {II}: dynamics.
\newblock \emph{J. Eur. Math. Soc. (JEMS)} \textbf{26}, no.~6, (2024), 1933--2089.
\newblock \myhref{doi:10.4171/jems/1317}{https://dx.doi.org/10.4171/jems/1317}.

\bibitem[BT08a]{Burq_Tzvetkov_08_I}
\textsc{N.~Burq} and \textsc{N.~Tzvetkov}.
\newblock Random data {C}auchy theory for supercritical wave equations. {I}. {L}ocal theory.
\newblock \emph{Invent. Math.} \textbf{173}, no.~3, (2008), 449--475.
\newblock \myhref{doi:10.1007/s00222-008-0124-z}{https://dx.doi.org/10.1007/s00222-008-0124-z}.

\bibitem[BT08b]{Burq_Tzvetkov_08_II}
\textsc{N.~Burq} and \textsc{N.~Tzvetkov}.
\newblock Random data {C}auchy theory for supercritical wave equations. {II}. {A} global existence result.
\newblock \emph{Invent. Math.} \textbf{173}, no.~3, (2008), 477--496.
\newblock \myhref{doi:10.1007/s00222-008-0123-0}{https://dx.doi.org/10.1007/s00222-008-0123-0}.

\bibitem[CC23]{cao2021yang}
\textsc{S.~Cao} and \textsc{S.~Chatterjee}.
\newblock The {Y}ang-{M}ills heat flow with random distributional initial data.
\newblock \emph{Comm. Partial Differential Equations} \textbf{48}, no.~2, (2023), 209--251.
\newblock \myhref{arXiv:2111.10652}{https://arxiv.org/abs/2111.10652}.
\newblock \myhref{doi:10.1080/03605302.2023.2169937}{https://dx.doi.org/10.1080/03605302.2023.2169937}.

\bibitem[CC24]{cao2024yang_state}
\textsc{S.~Cao} and \textsc{S.~Chatterjee}.
\newblock A state space for 3{D} {E}uclidean {Y}ang-{M}ills theories.
\newblock \emph{Comm. Math. Phys.} \textbf{405}, no.~1, (2024), Paper No. 3, 69.
\newblock \myhref{doi:10.1007/s00220-023-04870-y}{https://dx.doi.org/10.1007/s00220-023-04870-y}.

\bibitem[CCHS22]{CCHS22_2D}
\textsc{A.~Chandra}, \textsc{I.~Chevyrev}, \textsc{M.~Hairer}, and \textsc{H.~Shen}.
\newblock Langevin dynamic for the 2{D} {Y}ang-{M}ills measure.
\newblock \emph{Publ. Math. Inst. Hautes \'{E}tudes Sci.} \textbf{136}, (2022), 1--147.
\newblock \myhref{doi:10.1007/s10240-022-00132-0}{https://dx.doi.org/10.1007/s10240-022-00132-0}.

\bibitem[CCHS24]{CCHS22_3D}
\textsc{A.~Chandra}, \textsc{I.~Chevyrev}, \textsc{M.~Hairer}, and \textsc{H.~Shen}.
\newblock Stochastic quantisation of {Y}ang-{M}ills-{H}iggs in 3{D}.
\newblock \emph{Invent. Math.} \textbf{237}, no.~2, (2024), 541--696.
\newblock \myhref{doi:10.1007/s00222-024-01264-2}{https://dx.doi.org/10.1007/s00222-024-01264-2}.

\bibitem[CG13]{CG13}
\textsc{N.~Charalambous} and \textsc{L.~Gross}.
\newblock The {Y}ang-{M}ills heat semigroup on three-manifolds with boundary.
\newblock \emph{Comm. Math. Phys.} \textbf{317}, no.~3, (2013), 727--785.
\newblock \myhref{doi:10.1007/s00220-012-1558-0}{https://dx.doi.org/10.1007/s00220-012-1558-0}.

\bibitem[CG15]{CG15}
\textsc{N.~Charalambous} and \textsc{L.~Gross}.
\newblock Neumann domination for the {Y}ang-{M}ills heat equation.
\newblock \emph{J. Math. Phys.} \textbf{56}, no.~7, (2015), 073505, 21.
\newblock \myhref{doi:10.1063/1.4927250}{https://dx.doi.org/10.1063/1.4927250}.

\bibitem[CG23]{Camps_Gassot_23}
\textsc{N.~Camps} and \textsc{L.~Gassot}.
\newblock Pathological {S}et of {I}nitial {D}ata for {S}caling-{S}upercritical {N}onlinear {S}chr\"{o}dinger {E}quations.
\newblock \emph{Int. Math. Res. Not. IMRN} , no.~15, (2023), 13214--13254.
\newblock \myhref{doi:10.1093/imrn/rnac194}{https://dx.doi.org/10.1093/imrn/rnac194}.

\bibitem[CH16]{chandra2016analytic}
\textsc{A.~{Chandra}} and \textsc{M.~{Hairer}}.
\newblock {An analytic BPHZ theorem for regularity structures}.
\newblock \emph{ArXiv e-prints} (2016).
\newblock \myhref{arXiv:1612.08138}{https://arxiv.org/abs/1612.08138}.

\bibitem[Che19]{Chevyrev19YM}
\textsc{I.~Chevyrev}.
\newblock Yang-{M}ills measure on the two-dimensional torus as a random distribution.
\newblock \emph{Comm. Math. Phys.} \textbf{372}, no.~3, (2019), 1027--1058.
\newblock \myhref{doi:10.1007/s00220-019-03567-5}{https://dx.doi.org/10.1007/s00220-019-03567-5}.

\bibitem[Che22]{Chevyrev22_YM}
\textsc{I.~Chevyrev}.
\newblock Stochastic quantization of {Yang–Mills}.
\newblock \emph{Journal of Mathematical Physics} \textbf{63}, no.~9, (2022), 091101.
\newblock Proceedings of ICMP XX (2021).
\newblock \myhref{arXiv:2202.13359}{https://arxiv.org/abs/2202.13359}.
\newblock \myhref{doi:10.1063/5.0089431}{https://dx.doi.org/10.1063/5.0089431}.

\bibitem[Che24]{Chevyrev22_norm_inf}
\textsc{I.~Chevyrev}.
\newblock Norm inflation for a non-linear heat equation with gaussian initial conditions.
\newblock \emph{Stoch. Partial Differ. Equ. Anal. Comput.} \textbf{12}, no.~3, (2024), 1745--1768.
\newblock \myhref{doi:10.1007/s40072-023-00317-6}{https://dx.doi.org/10.1007/s40072-023-00317-6}.

\bibitem[COW22]{COW22}
\textsc{I.~{Chevyrev}}, \textsc{T.~{Oh}}, and \textsc{Y.~{Wang}}.
\newblock {Norm inflation for the cubic nonlinear heat equation above the scaling critical regularity}.
\newblock \emph{arXiv e-prints} (2022).
\newblock To appear in \textit{Funkcialaj Ekvacioj}.
\newblock \myhref{arXiv:2205.14488}{https://arxiv.org/abs/2205.14488}.

\bibitem[CS23]{ChevyrevShen23}
\textsc{I.~{Chevyrev}} and \textsc{H.~{Shen}}.
\newblock {Invariant measure and universality of the 2D Yang-Mills Langevin dynamic}.
\newblock \emph{arXiv e-prints} (2023).
\newblock \myhref{arXiv:2302.12160}{https://arxiv.org/abs/2302.12160}.

\bibitem[CW17]{ChandraWeber17}
\textsc{A.~Chandra} and \textsc{H.~Weber}.
\newblock Stochastic {PDE}s, regularity structures, and interacting particle systems.
\newblock \emph{Ann. Fac. Sci. Toulouse Math. (6)} \textbf{26}, no.~4, (2017), 847--909.
\newblock \myhref{doi:10.5802/afst.1555}{https://dx.doi.org/10.5802/afst.1555}.

\bibitem[DNY24]{DNY_24_Gibbs}
\textsc{Y.~Deng}, \textsc{A.~Nahmod}, and \textsc{H.~Yue}.
\newblock Invariant {G}ibbs measures and global strong solutions for nonlinear {S}chr\"{o}dinger equations in dimension two.
\newblock \emph{Ann. of Math. (2)} \textbf{200}, no.~2, (2024), 399--486.
\newblock \myhref{doi:10.4007/annals.2024.200.2.1}{https://dx.doi.org/10.4007/annals.2024.200.2.1}.

\bibitem[DPD02]{DPD02_NS}
\textsc{G.~Da~Prato} and \textsc{A.~Debussche}.
\newblock Two-dimensional {N}avier-{S}tokes equations driven by a space-time white noise.
\newblock \emph{J. Funct. Anal.} \textbf{196}, no.~1, (2002), 180--210.
\newblock \myhref{doi:10.1006/jfan.2002.3919}{https://dx.doi.org/10.1006/jfan.2002.3919}.

\bibitem[Duc25]{Duch25}
\textsc{P.~Duch}.
\newblock Flow equation approach to singular stochastic {PDE}s.
\newblock \emph{Probab. Math. Phys.} \textbf{6}, no.~2, (2025), 327--437.
\newblock \myhref{doi:10.2140/pmp.2025.6.327}{https://dx.doi.org/10.2140/pmp.2025.6.327}.

\bibitem[FV10]{FV10}
\textsc{P.~K. Friz} and \textsc{N.~B. Victoir}.
\newblock \emph{Multidimensional Stochastic Processes as Rough Paths: Theory and Applications}.
\newblock Cambridge Studies in Advanced Mathematics. Cambridge University Press, 2010.
\newblock \myhref{doi:10.1017/CBO9780511845079}{https://dx.doi.org/10.1017/CBO9780511845079}.

\bibitem[GIP15]{gubinelli2015paracontrolled}
\textsc{M.~Gubinelli}, \textsc{P.~Imkeller}, and \textsc{N.~Perkowski}.
\newblock Paracontrolled distributions and singular {PDE}s.
\newblock \emph{Forum Math. Pi} \textbf{3}, (2015), e6, 75.
\newblock \myhref{arXiv:1210.2684v3}{https://arxiv.org/abs/1210.2684v3}.
\newblock \myhref{doi:10.1017/fmp.2015.2}{https://dx.doi.org/10.1017/fmp.2015.2}.

\bibitem[Gro22]{Gross2016FiniteAction}
\textsc{L.~Gross}.
\newblock The {Y}ang-{M}ills heat equation with finite action in three dimensions.
\newblock \emph{Mem. Amer. Math. Soc.} \textbf{275}, no. 1349, (2022), v+111.
\newblock \myhref{doi:10.1090/memo/1349}{https://dx.doi.org/10.1090/memo/1349}.

\bibitem[GRZ24]{GRZ23}
\textsc{S.~Gabriel}, \textsc{T.~Rosati}, and \textsc{N.~Zygouras}.
\newblock The allen--cahn equation with weakly critical random initial datum.
\newblock \emph{Probability Theory and Related Fields} (2024).
\newblock \myhref{doi:10.1007/s00440-024-01312-1}{https://dx.doi.org/10.1007/s00440-024-01312-1}.

\bibitem[Hai14]{Hairer14}
\textsc{M.~Hairer}.
\newblock A theory of regularity structures.
\newblock \emph{Invent. Math.} \textbf{198}, no.~2, (2014), 269--504.
\newblock \myhref{doi:10.1007/s00222-014-0505-4}{https://dx.doi.org/10.1007/s00222-014-0505-4}.

\bibitem[HLR23]{Hairer_Le_Rosati_22}
\textsc{M.~Hairer}, \textsc{K.~L\^{e}}, and \textsc{T.~Rosati}.
\newblock The {A}llen-{C}ahn equation with generic initial datum.
\newblock \emph{Probab. Theory Related Fields} \textbf{186}, no. 3-4, (2023), 957--998.
\newblock \myhref{doi:10.1007/s00440-023-01198-5}{https://dx.doi.org/10.1007/s00440-023-01198-5}.

\bibitem[HQ18]{hairer2018class}
\textsc{M.~Hairer} and \textsc{J.~Quastel}.
\newblock A class of growth models rescaling to {KPZ}.
\newblock \emph{Forum Math. Pi} \textbf{6}, (2018), e3, 112.
\newblock \myhref{doi:10.1017/fmp.2018.2}{https://dx.doi.org/10.1017/fmp.2018.2}.

\bibitem[HS22]{HS22_Support}
\textsc{M.~Hairer} and \textsc{P.~Sch\"{o}nbauer}.
\newblock The support of singular stochastic partial differential equations.
\newblock \emph{Forum Math. Pi} \textbf{10}, (2022), Paper No. e1, 127.
\newblock \myhref{doi:10.1017/fmp.2021.18}{https://dx.doi.org/10.1017/fmp.2021.18}.

\bibitem[HS24]{HS_24_SG}
\textsc{M.~Hairer} and \textsc{R.~Steele}.
\newblock The {BPHZ} theorem for regularity structures via the spectral gap inequality.
\newblock \emph{Arch. Ration. Mech. Anal.} \textbf{248}, no.~1, (2024), Paper No. 9, 81.
\newblock \myhref{doi:10.1007/s00205-023-01946-w}{https://dx.doi.org/10.1007/s00205-023-01946-w}.

\bibitem[Jan97]{janson1997gaussian}
\textsc{S.~Janson}.
\newblock \emph{Gaussian {H}ilbert spaces}, vol. 129 of \emph{Cambridge Tracts in Mathematics}.
\newblock Cambridge University Press, Cambridge, 1997,  x+340.
\newblock \myhref{doi:10.1017/CBO9780511526169}{https://dx.doi.org/10.1017/CBO9780511526169}.

\bibitem[KLS23]{KLS_23}
\textsc{J.~Krieger}, \textsc{J.~L\"{u}hrmann}, and \textsc{G.~Staffilani}.
\newblock Probabilistic small data global well-posedness of the energy-critical {M}axwell-{K}lein-{G}ordon equation.
\newblock \emph{Arch. Ration. Mech. Anal.} \textbf{247}, no.~4, (2023), Paper No. 68, 109.
\newblock \myhref{doi:10.1007/s00205-023-01900-w}{https://dx.doi.org/10.1007/s00205-023-01900-w}.

\bibitem[Kup16]{Kupiainen16}
\textsc{A.~Kupiainen}.
\newblock Renormalization group and stochastic {PDE}s.
\newblock \emph{Ann. Henri Poincar\'{e}} \textbf{17}, no.~3, (2016), 497--535.
\newblock \myhref{doi:10.1007/s00023-015-0408-y}{https://dx.doi.org/10.1007/s00023-015-0408-y}.

\bibitem[LOTT24]{LOTT_24_SG}
\textsc{P.~Linares}, \textsc{F.~Otto}, \textsc{M.~Tempelmayr}, and \textsc{P.~Tsatsoulis}.
\newblock A diagram-free approach to the stochastic estimates in regularity structures.
\newblock \emph{Invent. Math.} \textbf{237}, no.~3, (2024), 1469--1565.
\newblock \myhref{doi:10.1007/s00222-024-01275-z}{https://dx.doi.org/10.1007/s00222-024-01275-z}.

\bibitem[NPS13]{NPS_13_NS_random}
\textsc{A.~R. Nahmod}, \textsc{N.~Pavlovi{\'c}}, and \textsc{G.~Staffilani}.
\newblock Almost sure existence of global weak solutions for supercritical {N}avier-{S}tokes equations.
\newblock \emph{SIAM J. Math. Anal.} \textbf{45}, no.~6, (2013), 3431--3452.
\newblock \myhref{doi:10.1137/120882184}{https://dx.doi.org/10.1137/120882184}.

\bibitem[OOT24]{OOT_24_NLW_GFF}
\textsc{T.~Oh}, \textsc{M.~Okamoto}, and \textsc{N.~Tzvetkov}.
\newblock Uniqueness and non-uniqueness of the {G}aussian free field evolution under the two-dimensional {W}ick ordered cubic wave equation.
\newblock \emph{Ann. Inst. Henri Poincar\'{e} Probab. Stat.} \textbf{60}, no.~3, (2024), 1684--1728.
\newblock \myhref{doi:10.1214/23-aihp1380}{https://dx.doi.org/10.1214/23-aihp1380}.

\bibitem[OP16]{Oh_Pocovnicu_16}
\textsc{T.~Oh} and \textsc{O.~Pocovnicu}.
\newblock Probabilistic global well-posedness of the energy-critical defocusing quintic nonlinear wave equation on {$\mathbb{R}^3$}.
\newblock \emph{J. Math. Pures Appl. (9)} \textbf{105}, no.~3, (2016), 342--366.
\newblock \myhref{doi:10.1016/j.matpur.2015.11.003}{https://dx.doi.org/10.1016/j.matpur.2015.11.003}.

\bibitem[OTW20]{OTW_20_4NLS}
\textsc{T.~Oh}, \textsc{N.~Tzvetkov}, and \textsc{Y.~Wang}.
\newblock Solving the 4{NLS} with white noise initial data.
\newblock \emph{Forum Math. Sigma} \textbf{8}, (2020), Paper No. e48, 63.
\newblock \myhref{doi:10.1017/fms.2020.51}{https://dx.doi.org/10.1017/fms.2020.51}.

\bibitem[OW19]{OW19}
\textsc{F.~Otto} and \textsc{H.~Weber}.
\newblock Quasilinear {SPDE}s via rough paths.
\newblock \emph{Arch. Ration. Mech. Anal.} \textbf{232}, no.~2, (2019), 873--950.
\newblock \myhref{doi:10.1007/s00205-018-01335-8}{https://dx.doi.org/10.1007/s00205-018-01335-8}.

\bibitem[Poc17]{Pocovnicu17}
\textsc{O.~Pocovnicu}.
\newblock Almost sure global well-posedness for the energy-critical defocusing nonlinear wave equation on {$\mathbb{R}^d$}, {$d=4$} and {$5$}.
\newblock \emph{J. Eur. Math. Soc. (JEMS)} \textbf{19}, no.~8, (2017), 2521--2575.
\newblock \myhref{doi:10.4171/JEMS/723}{https://dx.doi.org/10.4171/JEMS/723}.

\bibitem[ST20]{Sun_Tzvetkov_20}
\textsc{C.~Sun} and \textsc{N.~Tzvetkov}.
\newblock New examples of probabilistic well-posedness for nonlinear wave equations.
\newblock \emph{J. Funct. Anal.} \textbf{278}, no.~2, (2020), 108322, 47.
\newblock \myhref{doi:10.1016/j.jfa.2019.108322}{https://dx.doi.org/10.1016/j.jfa.2019.108322}.

\bibitem[Sti19]{stinga2019user}
\textsc{P.~R. Stinga}.
\newblock User's guide to the fractional {L}aplacian and the method of semigroups.
\newblock In \emph{Handbook of fractional calculus with applications. {V}ol. 2},  235--265. De Gruyter, Berlin, 2019.
\newblock \myhref{doi:10.1515/9783110571660-012}{https://dx.doi.org/10.1515/9783110571660-012}.

\bibitem[Wil51]{Wild51}
\textsc{E.~Wild}.
\newblock On {B}oltzmann's equation in the kinetic theory of gases.
\newblock \emph{Proc. Cambridge Philos. Soc.} \textbf{47}, (1951), 602--609.
\newblock \myhref{doi:10.1017/s0305004100026992}{https://dx.doi.org/10.1017/s0305004100026992}.

\end{thebibliography}

\end{document}